\pdfoutput=1

\documentclass[preprint,10pt]{elsarticle}
\usepackage{fullpage}
\usepackage{hyperref}

\usepackage{amsmath,amssymb,amsfonts,mathrsfs,amsthm}
\usepackage[titletoc,toc,title]{appendix}

\usepackage{array}
\usepackage[utf8]{inputenc}
\usepackage{listings}
\usepackage{mathtools}
\usepackage{pdfpages}
\usepackage[textsize=footnotesize,color=green]{todonotes}
\usepackage{bm}
\usepackage[normalem]{ulem}
\usepackage{hhline}

\usepackage{algorithm}
\usepackage[noend]{algpseudocode}
\usepackage{algorithmicx}

\usepackage{graphicx}
\usepackage{subfig}

\usepackage{color}

\newcommand{\vect}[1]{\ensuremath\boldsymbol{#1}}

\newcommand{\td}[2]{\frac{{\rm d}#1}{{\rm d}#2}}
\newcommand{\pd}[2]{\frac{\partial#1}{\partial#2}}

\newcommand{\nor}[1]{\left\| #1 \right\|}

\newcommand{\LRp}[1]{\left( #1 \right)}
\newcommand{\LRs}[1]{\left[ #1 \right]}

\newcommand{\LRb}[1]{\left| #1 \right|}
\newcommand{\LRc}[1]{\left\{ #1 \right\}}

\newcommand{\Grad} {\ensuremath{\nabla}}
\newcommand{\Div} {\ensuremath{\nabla\cdot}}

\newcommand{\jump}[1] {\ensuremath{\LRs{\![#1]\!}}}
\newcommand{\avg}[1] {\ensuremath{\LRc{\!\{#1\}\!}}}

\newcommand{\Gh}{\Gamma_h}

\newcommand{\Oh}{\Omega_h}

\newtheorem{theorem}{Theorem}[section]
\newtheorem{lemma}[theorem]{Lemma}

\newcommand{\eval}[2][\right]{\relax
  \ifx#1\right\relax \left.\fi#2#1\rvert}

\newcommand{\refhex}{\widehat{\mathcal{H}}}

\newcommand{\refwedg}{\widehat{\mathcal{W}}}

\newcommand{\refpyr}{\widehat{\mathcal{P}}}

\newcommand{\hex}{{\mathcal{H}}}
\newcommand{\tet}{{\mathcal{T}}}
\newcommand{\wedg}{{\mathcal{W}}}

\newcommand{\pyr}{{\mathcal{P}}}

\newcolumntype{C}[1]{>{\centering\let\newline\\\arraybackslash\hspace{0pt}}m{#1}}

\newcommand*\diff[1]{\mathop{}\!{\mathrm{d}#1}}

\makeatletter
\renewcommand\d[1]{\mspace{6mu}\mathrm{d}#1\@ifnextchar\d{\mspace{-3mu}}{}}
\makeatother

\journal{Computer Methods in Applied Mechanics and Engineering (CMAME)}

\begin{document}

\begin{frontmatter}

\author[rice]{Jesse Chan\corref{cor1}}
\ead{Jesse.Chan@caam.rice.edu}
\cortext[cor1]{Principal Corresponding author}
\author[rice]{Zheng Wang}
\ead{zw14@caam.rice.edu}
\author[rice]{Axel Modave}
\ead{modave@caam.rice.edu}
\author[ucl]{Jean-Francois Remacle}
\ead{Jean-Francois.Remacle@uclouvain.be}
\author[rice]{T. Warburton}
\ead{timwar@caam.rice.edu}
\address[rice]{Department of Computational and Applied Mathematics, Rice University, 6100 Main St, Houston, TX, 77005}
\address[ucl]{School of Engineering, Universite catholique de Louvain, Avenue Georges Lemaitre 4,  B-1348 Louvain-la-Neuve, Belgium}

\title{GPU-accelerated discontinuous Galerkin methods on hybrid meshes}

\begin{abstract}
We present a time-explicit discontinuous Galerkin (DG) solver for the time-domain acoustic wave equation on hybrid meshes containing vertex-mapped hexahedral, wedge, pyramidal and tetrahedral elements.  Discretely energy-stable formulations are presented for both Gauss-Legendre and Gauss-Legendre-Lobatto (Spectral Element) nodal bases for the hexahedron.  Stable timestep restrictions for hybrid meshes are derived by bounding the spectral radius of the DG operator using order-dependent constants in trace and Markov inequalities.  Computational efficiency is achieved under a combination of element-specific kernels (including new quadrature-free operators for the pyramid), multi-rate timestepping, and acceleration using Graphics Processing Units.  
\end{abstract}

\end{frontmatter}

%\tableofcontents

\section{Introduction}

Highly efficient solution techniques exist for high order finite element and Galerkin discretizations on hexahedral meshes.  Presently, producing high quality hexahedral meshes for complex domains is a difficult and non-robust procedure.  Hybrid meshes, which consist of wedge and pyramidal elements in addition to hexahedra and tetrahedra, have been proposed to leverage the efficiency of hexahedral elements for more general geometries.  Such meshes are commonly produced by adding a transitional layer of wedges and pyramids to a structured hexahedral mesh and constructing an unstructured tetrahedral mesh to fill in the remaining space \cite{schoberl2012netgen}.  Recently, techniques have been developed to automatically generate conforming hex-dominant hybrid meshes with a high percentage of hexahedral elements \cite{baudouin2014frontal}, which we aim to leverage to develop efficient solvers on more general geometries.  We note that the utility of high order time-explicit DG using hybrid (including polygonal) meshes has been explored for polynomial physical frame discretizations in \cite{gassner2009polymorphic}.  We consider in this work approximation spaces defined under a reference-to-physical mapping.  

Spectral and $hp$-finite element methods on hybrid meshes were explored early on by Sherwin, Karniadakis, Warburton, and Kirby \cite{sherwin1998spectral, warburton1999spectral, kirby2000discontinuous}  using orthogonal polynomial basis functions constructed for hexahedral, wedge, pyramidal, and tetrahedral elements.  While approximation spaces for the hexahedron, wedge, and tetrahedron have remained relatively unchanged since their introduction, more recent work has focused on the construction of alternative basis functions and approximation spaces for the pyramid.  
Bedrosian introduced a set of low-order vertex functions for the pyramid which yielded polynomial traces, allowing for conformity with low order finite elements for other shapes \cite{bedrosian1992shape}.  Unlike other elements, however, his pyramidal shape functions were rational in nature.  Bergot, Cohen, and Durufle extended this construction in \cite{bergot2010higher, bergot2013higher} to produce a high order basis for the pyramid (which we refer to as the \emph{rational basis}).  It was additionally shown that, under non-affine mappings of the reference pyramid, the mapped rational basis contains a complete polynomial space.  As a result, the rational basis provides optimal high order convergence rates on vertex-mapped elements, which is not true of the polynomial pyramid basis.  
Rational pyramidal bases were also adapted to the $H({\rm curl}),H({\rm div})$ setting and  exact sequence spaces in \cite{bergot2013high, bergot2013approximation, nigam2012numerical}.  The construction of exact sequence high order bases for elements of all types (including pyramids) is addressed in recent work by Fuentes, Keith, Demkowicz, and Nagaraj \cite{fuentes2015orientation}.  

Our focus is on high order discontinuous Galerkin (DG) methods on hybrid meshes.  We pay particular attention to the efficient implementation of solvers on hardware accelerators.  The computation-intensive nature of time-explicit methods lends itself well to modern accelerator architectures such as Graphics Processing Units (GPUs), and was exploited by Kl{\"o}ckner, Warburton, Bridge and Hesthaven to construct an efficient high order DG solver on tetrahedral meshes using a single GPU \cite{klockner2009nodal}.  Exposing both coarse and fine-grained parallelism resulted in several times speedup for meshes with around a hundred thousand elements at high orders of approximation.  Increases in on-board memory in modern GPUs have allowed for the solution of even larger problems on a single unit.  Additional gains in efficiency and problem size may be achieved via multi-rate timestepping and distributed-memory parallelism \cite{godel2010gpu, godel2010scalability}.  However, these technologies have been developed on all-hex or all-tetrahedra meshes, and while hetergeneous computing on hybrid meshes has been explored in the context of fluid dynamics and the closely related Flux Reconstruction method \cite{witherden2014pyfr}, GPU-accelerated DG methods on hybrid meshes have not yet been analyzed in detail.  

Crucial to the efficiency of GPU-accelerated DG methods is limiting the memory required.  Since each GPU has a relatively small amount of on-device storage, maximum problem sizes are typically memory-bound.  Classical techniques for hexahedral and tetrahedral elements limit the amount of data that must be stored, and recent developments in the construction of basis functions \cite{warburton2010low, warburton2013low, chan2015orthogonal} make it possible to develop low-storage DG methods on pyramids and wedges as well.  This work leverages each of these technologies to construct a low-memory GPU-accelerated high order DG solver on hybrid meshes for the acoustic wave equation.  Additionally, order-dependent global and local timestep restrictions for general hybrid meshes are derived using constants in discrete trace inequalities for each element.  

The structure of this paper is as follows: Section~\ref{sec:method} introduces a low-storage DG method for hybrid meshes.  The low-storage treatment of mass matrices is addressed through a careful choice of basis functions for each specific type of element.  The variational formulation and element-specific operations are given, and a computational implementation on many-core architectures is described.  Section~\ref{sec:timestep} derives order-dependent timestep restrictions for each type of element based on the constants in discrete trace inequalities.  Finally, Section~\ref{sec:numerics} reports numerical and computational results.  

\section{A Low-storage DG method on hybrid meshes}
\label{sec:method}
%\remark{Give memory motivation for this section}

Typical DG methods result in a system of ordinary differential equations of the form
\[
\td{u}{\tau} = M^{-1}Au,
\]
where $\tau$ is time, $A$ is a linear operator defined by the discretization, and $M$ is the global mass matrix, which is block diagonal in nature.  Efficient implementations of DG must account for the inverse of $M$ in a fast and memory-efficient manner.  One common strategy is to store the factorization or explicit inverse of each block of $M^{-1}$; however, on GPU and other accelerator architectures, on-device memory is typically limited to a few gigabytes.  Explicit storage of factorizations can rapidly exhaust available memory and limit the maximum problem size on a single GPU at high orders of approximation.  For example, assuming a mesh of $1$ million planar tetrahedra in single precision, at order $N=4$, storage of $M^{-1}$ already exceeds the 1 GB of memory available on early GPUs used to accelerate DG methods \cite{klockner2009nodal}, without accounting for the storage cost of other relevant data.  

%\begin{table}
%\centering
%\begin{tabular}{| c | c | c | c | c |}
%\hline
%Order $N$ & $N=1$ & $N=2$ & $N=3$ & $N=4$ \\
%\hline
%Block matrix $M^{-1}$ & .298 GB  & 1.86 GB  & {7.45 GB} & {22.82 GB} \\
%\hline
%Solution dofs & .075 GB  & .186 GB  & .373 GB & .652 GB \\
%\hline 
%\end{tabular}
%\label{table:mem}
%\caption{Storage costs for block factorizations of $M$ and solution dofs for high order DG ($5$M tets in single precision).}
%\end{table}

We consider meshes consisting of hexahedra, wedges, pyramids, and tetrahedra.  Given an order $N=1$ basis on each element, we may then define low-order vertex functions.  A map from the reference element to a physical element may then be defined by interpolating physical vertex positions with these vertex functions.  We refer to such elements as \emph{vertex-mapped} elements, and restrict ourselves to meshes consisting of such elements in this work.   Furthermore, we define $J$ to be the the determinant of the mapping Jacobian, such that
\[
\int_{K} u = \int_{\widehat{K}}u J,
\]
where $K,\widehat{K}$ are the physical and reference elements, respectively.  

\begin{figure}
\centering
\includegraphics[width=.23\textwidth]{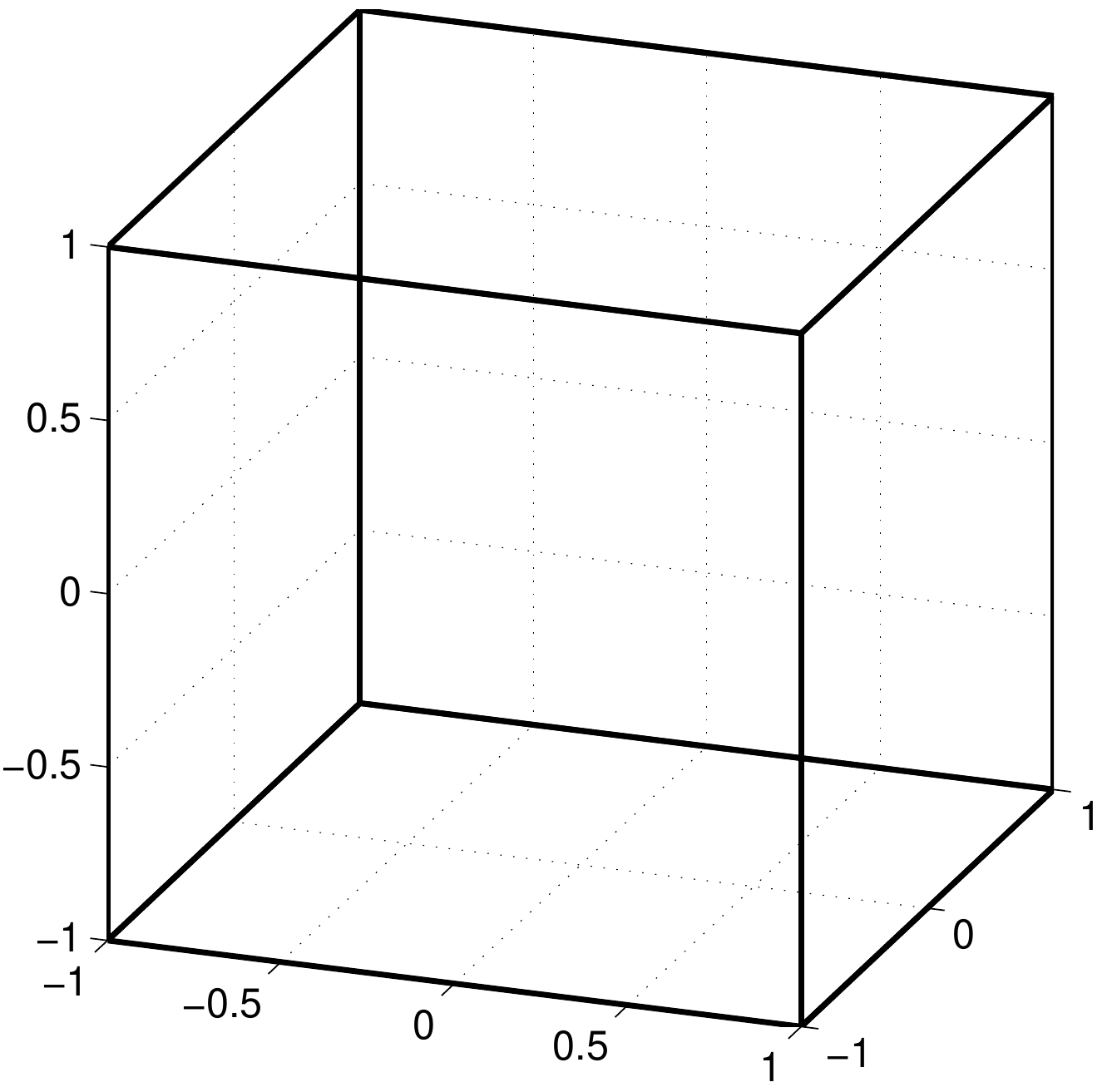}
\includegraphics[width=.23\textwidth]{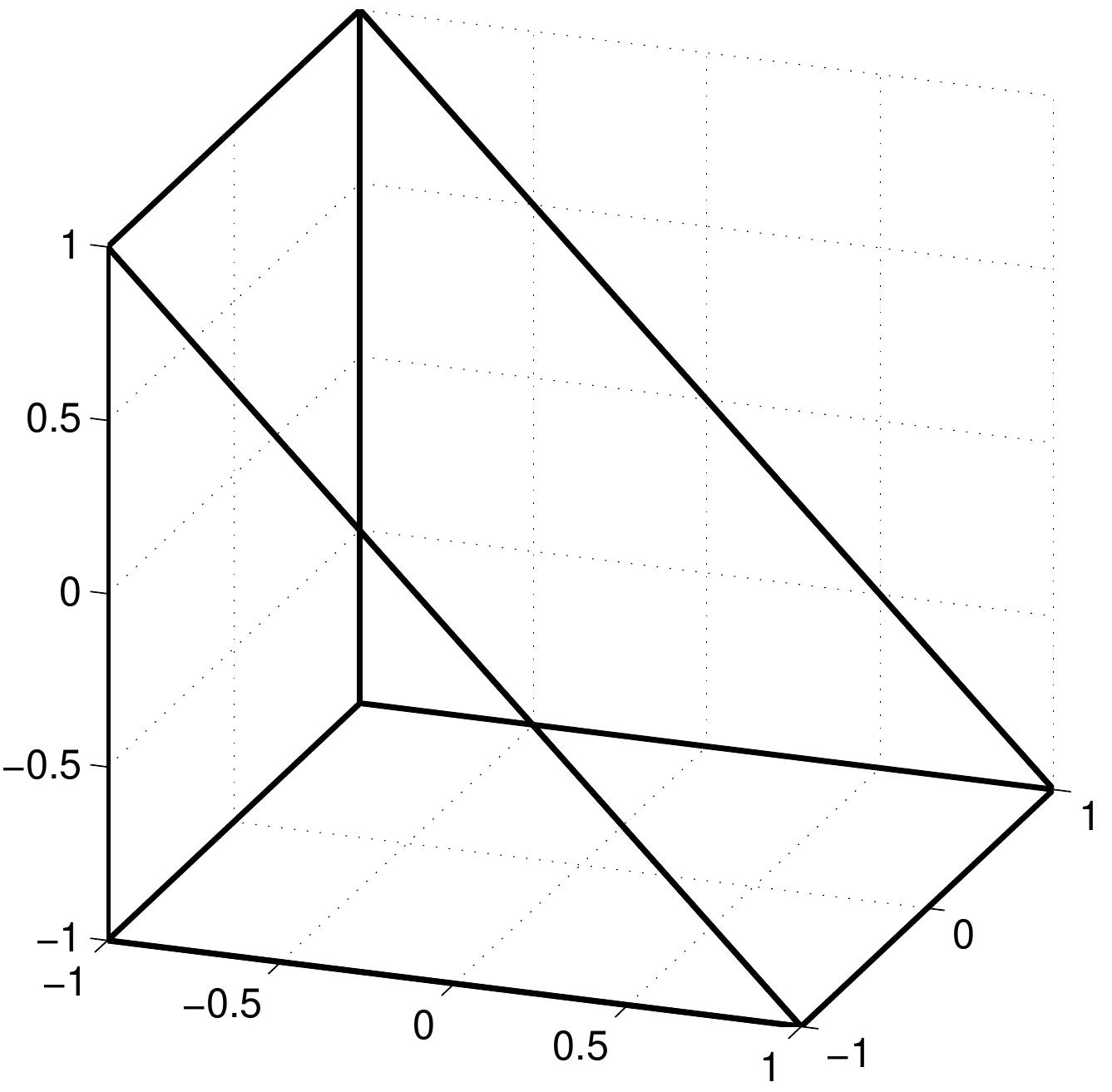}
\includegraphics[width=.23\textwidth]{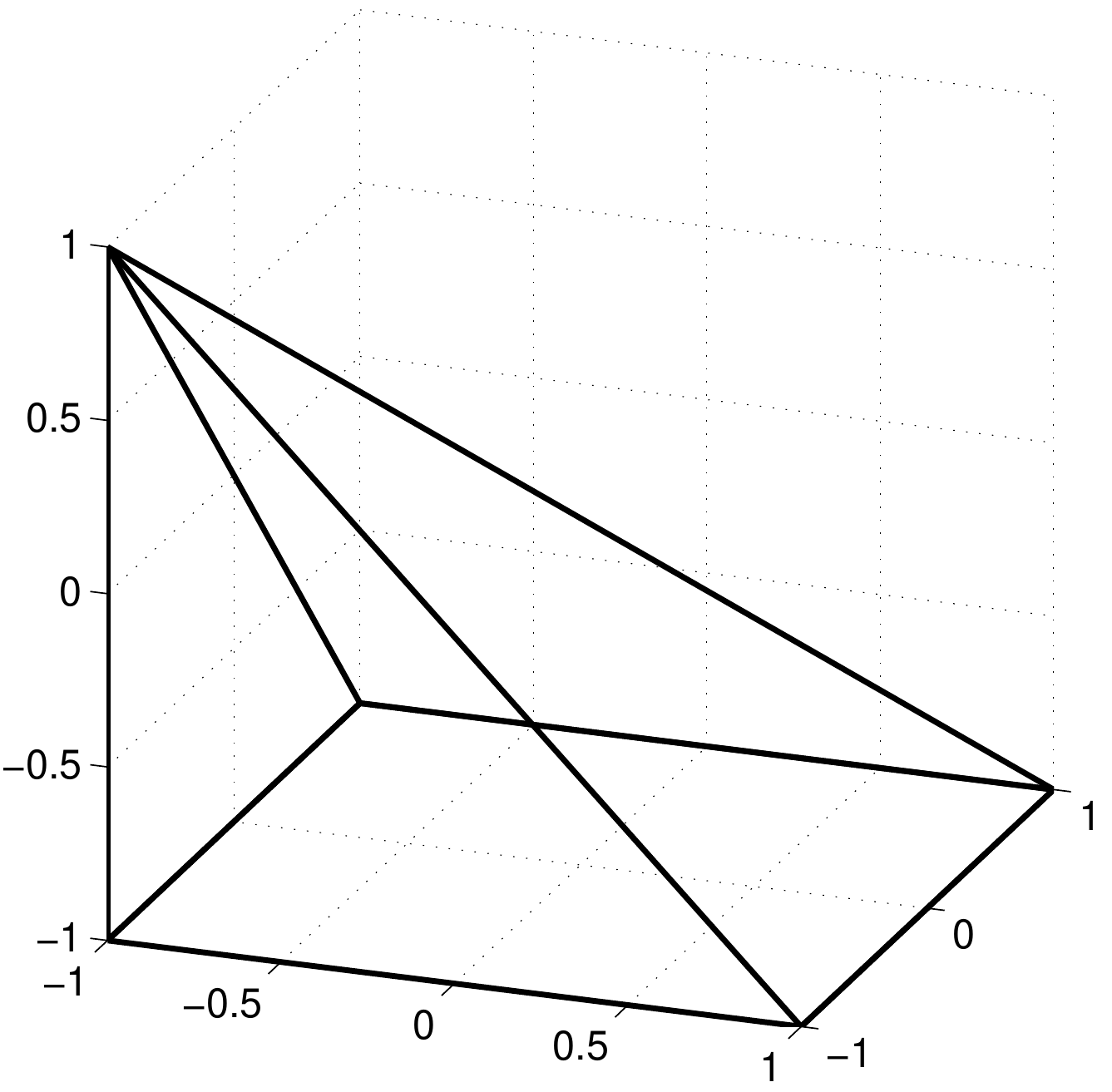}
\includegraphics[width=.23\textwidth]{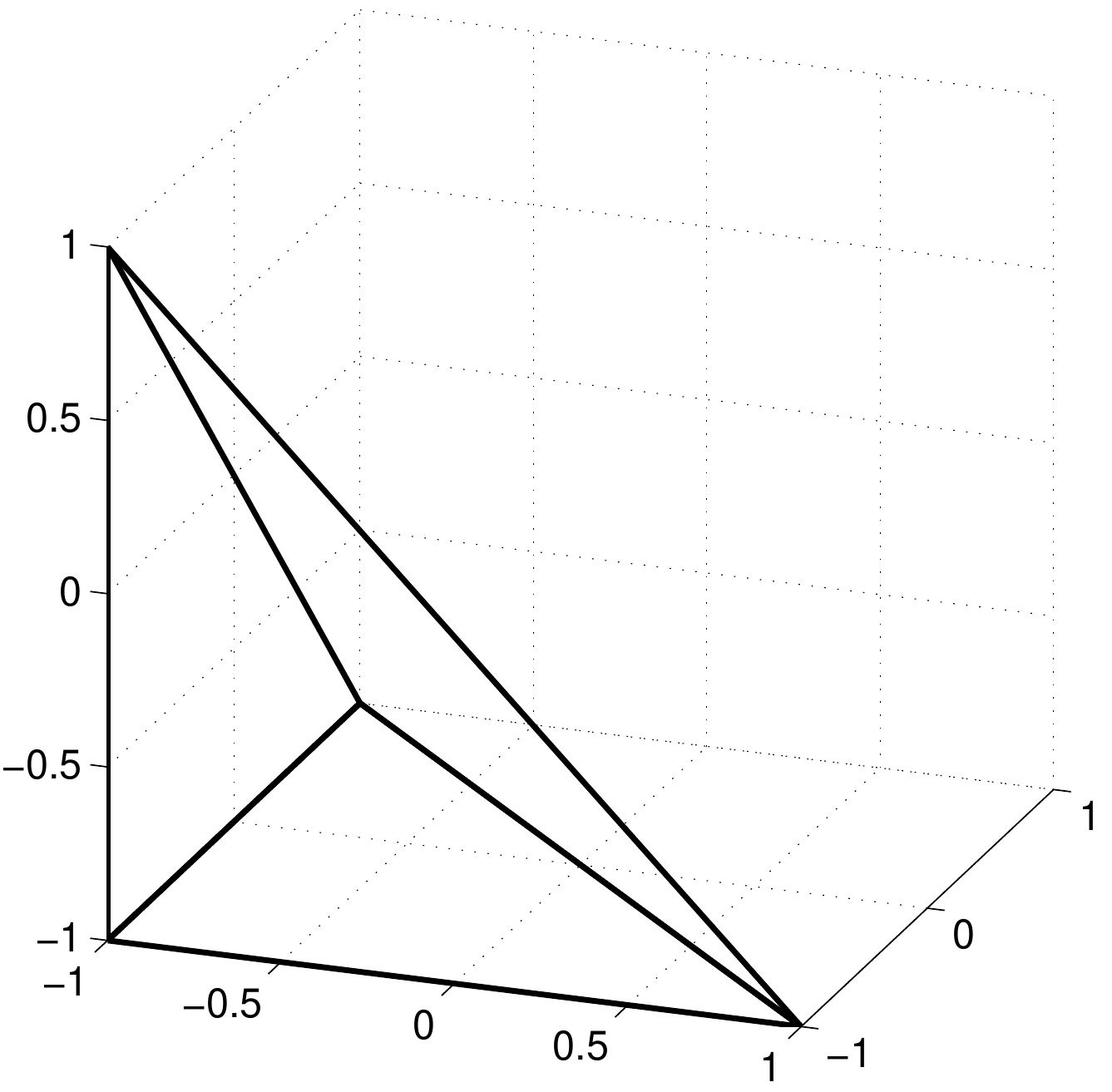}
\caption{Types of elements present in hex-dominant meshes.  From left to right: reference hexahedra, wedge, pyramid, and tetrahedra.}
\label{fig:hybrid}
\end{figure}

\subsection{Basis functions} 
\label{sec:bases}

We will sidestep the cost of storage of $M^{-1}$ by tailoring our choice of basis functions for each element.  Assuming that all elements are vertex-mapped, we will construct basis functions for the hexahedron, wedge, and pyramid which yield diagonal mass matrices.  For vertex-mapped tetrahedra, the mass matrix will not be diagonal, but will be a scalar multiple of the reference mass matrix.  

\subsubsection{Hexahedra}

We take the reference hexahedron to be the bi-unit cube, with coordinates $(a,b,c )\in [-1,1]^3$.  
%An orthogonal basis is given in terms of Jacobi (Legendre) polynomials
%\[
%\psi_{ijk}(a,b,c) = P_i^{0,0}(a) P_j^{0,0}(b) P_k^{0,0}(c), \quad 0 \leq i,j,k\leq N.
%\]
The approximation space of order $N$ on the hexahedron is simply the tensor product space of polynomials of order $N$ in each direction.  For a vertex-mapped hexahedra $\hex$, this space and its dimension $N_p$ are given as follows
\[
P_N(\hex) = \LRc{x^iy^jz^k, \quad 0\leq i,j,k\leq N}, \quad N_p = (N+1)^3.
\]
We also have that $J \in P_1(\refhex)$, and an appropriate quadrature rule can be constructed using a tensor product of standard 1D Gauss-Legendre rules.  If we take the orthogonal basis
\[
\phi_{ijk}(a,b,c) = \ell_i(a) \ell_j(b)\ell_k(c) \quad 0 \leq i,j,k\leq N,
\]
where $\ell_i$ is the $i$th Lagrange polynomial at the $i$th Gauss-Legendre node, the mass matrix defined by
\[
M_{ijk,lmn} = \int_{\hex} \phi_{ijk} \phi_{lmn} \diff x = \int_{\refhex}  \phi_{ijk}\phi_{lmn} J \diff {\widehat{x}}
\]
is diagonal for $J \in P_1(\refhex)$.  To apply $M^{-1}$, we store only of the inverse of the diagonal of the full mass matrix, which requires the storage of $N_p K$ values, instead of storing the $O(N_p^2 K)$ entries required for the factorization of the global block-diagonal mass matrix.  

\begin{figure}
\centering
\subfloat[SEM nodes]{\includegraphics[width=.2\textwidth]{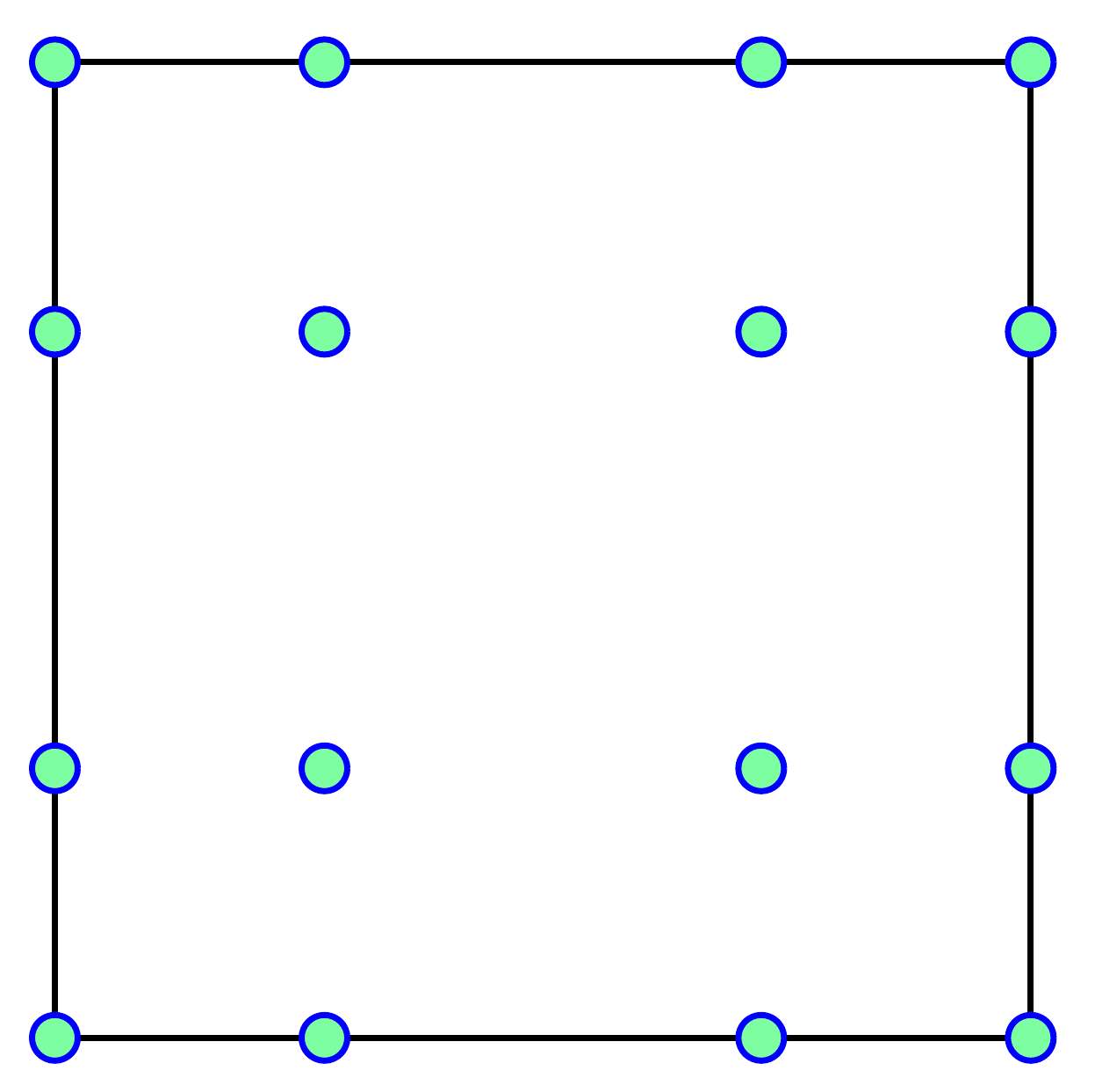}}
\hspace{5em}
\subfloat[GL nodes]{\includegraphics[width=.2\textwidth]{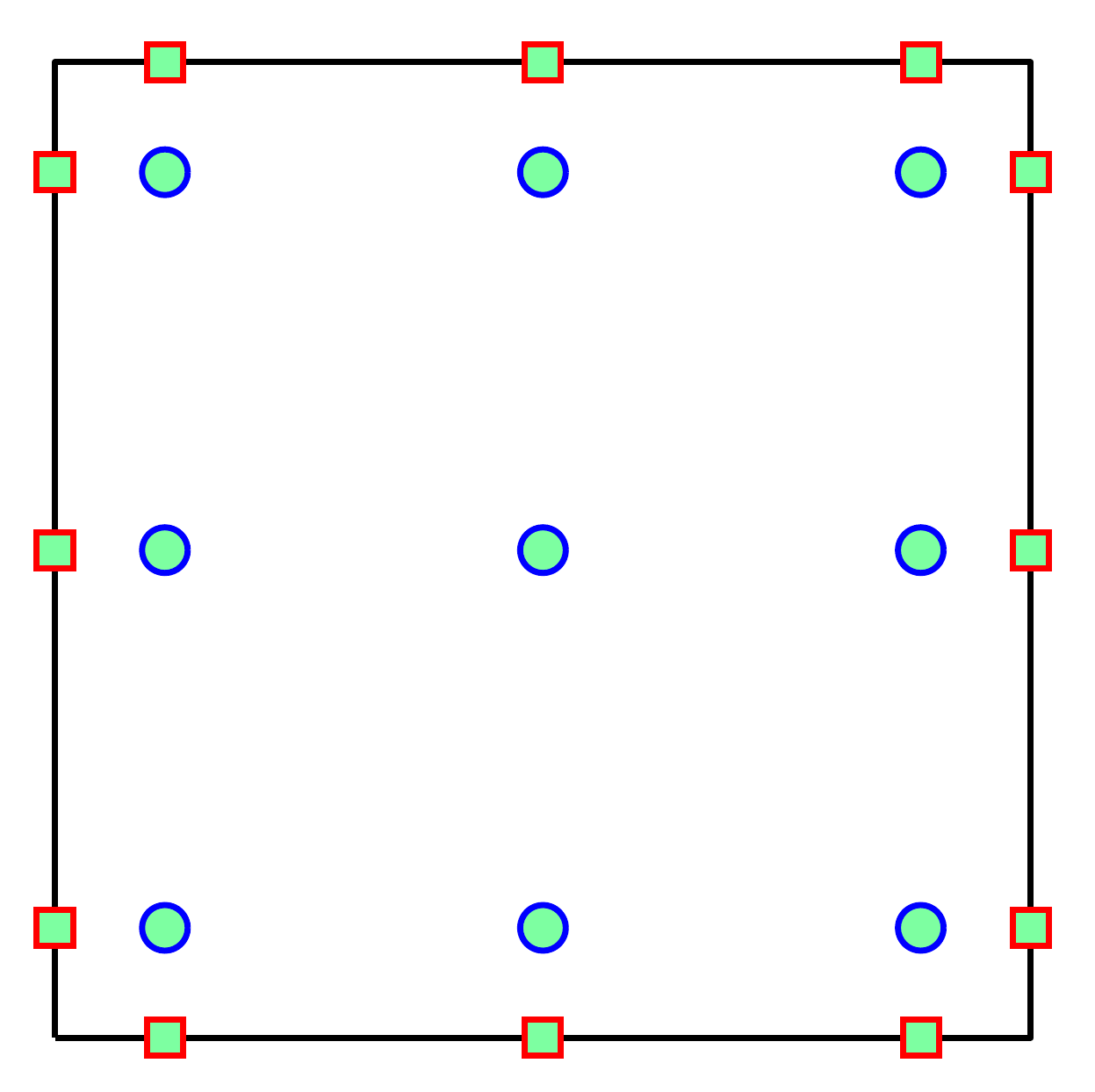}}
\caption{SEM and GL nodes on the quadrilateral.  For GL nodes, additional surface nodes must be specified (shown as squares).}
\end{figure}

We may also achieve a diagonal mass matrix by taking Lagrange polynomials at tensor product Gauss-Legendre-Lobatto nodes, which includes points at both $-1$ and $1$.  Applying Gauss-Legendre-Lobatto quadrature results in a diagonal mass matrix, and is the basis behind the Spectral Element Method \cite{patera1984spectral, karniadakis2013spectral}.  We refer to the choice of basis using Gauss-Legendre nodes as the GL basis, and the basis using Gauss-Legendre-Lobatto nodes as the SEM basis.  We will compare both of these choices in this work.  

%\remark{Introduce GL/SEM abbreviations here + picture of GL vs SEM.}
\subsubsection{Tetrahedra}
%\remark{Move first - then hex, wedge, pyramid}
For the reference tetrahedron, an orthogonal basis may be given as the product of Jacobi polynomials on the bi-unit cube \cite{proriol1957famille, koornwinder1975two, dubiner1991spectral}
\[
\phi_{ijk}(a,b,c) = P_i^{0,0}(a)(1-b)^i P_j^{2i+1,0}(b) (1-c)^{i+j} P_k^{2(i+j +1),0}(c).
\]
The tetrahedral basis is then given by the Duffy mapping from the bi-unit cube to the bi-unit right tetrahedron with coordinates $(r,s,t)$
\begin{align*}
r &= (1+a)\LRp{\frac{1-b}{2}}\LRp{\frac{1-c}{2}}-1, \quad s = (1+b)\LRp{\frac{1-c}{2}}-1, \quad t = c.
%a &= -2\LRp{\frac{1+r}{s+t}}, \quad b = 2\LRp{\frac{1+s}{1-t}} - 1, \quad c = t.  
\end{align*}
%with inverse mapping %with change of variables factor $(1-b)(1-c)^2/8$ and
The approximation space for a vertex-mapped tetrahedra is the space of polynomials of total order $N$ with dimension $N_p$
\[
P(\tet) = \LRc{x^iy^jz^k, \quad i + j + k \leq N}, \quad N_p = \frac{(N+1)(N+2)(N+3)}{6}.
\]
A convenient fact about vertex-mapped (planar) tetrahedra is that $J$ is constant on each element.  Thus, each local mass matrix is simply a constant scaling of the reference mass matrix, and only the reference mass matrix is stored.  In our implementation, we choose a nodal Lagrange basis (defined implicitly using the above orthogonal basis) where the nodes are placed at optimized interpolation points on the simplex \cite{warburton2006explicit}.  

\subsubsection{Wedges}

An orthogonal basis for the reference wedge may be given in terms of Jacobi polynomials on the bi-unit cube
\[
\phi_{ijk}(a,b,c) = P_i^{0,0}(a) P_j^{0,0}(b) \LRp{\frac{1-c}{2}}^i P_k^{2i+1,0}(c), \qquad 0\leq i,j \leq N, \quad 0\leq k \leq N-i.
\]
Applying a Duffy-type transform gives a basis on the reference bi-unit wedge with coordinates $(r,s,t)$
\[
r = (1+a)\LRp{\frac{1-c}{2}} - 1, \quad s = b, \quad t = c.%, \qquad a = 2\LRp{\frac{r+1}{1-t}}-1, \quad b = s, \quad c = t.
\]
%which has a change of variables factor of $\LRp{\LRp{1-c}/{2}}$.  
For a vertex-mapped wedge $\wedg$,  $J \in P_1(\refwedg)$, and the approximation space of order $N$ on the wedge and its dimension are 
\[
P_N(\wedg) = \LRc{x^iy^jz^k, \quad 0\leq i,j\leq N, \quad 0\leq  k \leq N-i},\quad N_p = \frac{(N+1)^2(N+2)}{2}.
\]
We may construct an appropriate quadrature rule on the reference wedge based on the tensor product of Gauss-Legendre rules in the $r,s$ coordinate, and a Gauss-Legendre rule with weights $(1,0)$ in the $t$ coordinate.  Alternatively, a wedge cubature may be constructed by taking the tensor product of an optimized cubature on the triangle with Gauss-Legendre cubature in the $s$ coordinate.  We take the latter approach in this work, using rules from Xiao and Gimbutas \cite{xiao2010quadrature}.  The triangle quadrature is chosen such that the rule is exact for polynomials of total order $2N+1$, similar to Gauss-Legendre quadrature.  

Decreasing storage costs for the wedge mass matrix is less straightforward.  One common approach is to construct a basis using the tensor product of triangular basis functions with Lagrange polynomials at $N+1$ Gauss-Legendre points.  The resulting mapped basis is orthogonal in the $b$ direction, and the local mass matrix becomes block diagonal, with each block corresponding to a triangular slice of the wedge.  Unfortunately, for general vertex-mapped wedges, $J$ is bilinear on each triangular slice, and each block of the local mass matrix corresponds to a $J$-weighted $L^2$ inner product on a triangle.  Since there is clear basis which is orthogonal in this inner product for an arbitrary bilinear $J$, the blocks of each local mass matrix are typically dense and distinct from element to element.  Factorizations of each block are required in order to efficiently apply the mass matrix inverse, resulting in increased storage costs.  

We may address this storage cost by using a Low-Storage Curvilinear DG (LSC-DG) approach \cite{warburton2010low, warburton2013low}, where we define the $H^1$ non-conforming basis $\tilde{\phi}_{ijk}$
\[
\tilde{\phi}_{ijk} = \frac{\phi_{ijk}}{\sqrt{J}},
\]
such that 
\[
M_{ijk,lmn} = \int_{\refwedg}  \tilde{\phi}_{ijk}\tilde{\phi}_{lmn} J \diff {\widehat{x}}  \int_{\refwedg}  \frac{\phi_{ijk}}{\sqrt{J}}\frac{\phi_{lmn}}{\sqrt{J}} J \diff {\widehat{x}} = \widehat{M}_{ijk,lmn}.
\]
As a result, each mass matrix is identical to the reference mass matrix, allowing us to apply $M^{-1}$ while accounting for only a single mass matrix inverse over all elements.  

The use of LSC-DG results in a rational basis; consequentially, standard quadratures are not exact and aliasing errors may result from underintegration.  To sidestep these issues, we restrict ourselves to variational formulations that may be written in a skew-symmetric fashion, such that DG is energy stable irregardless of the choice of quadrature \cite{warburton2013low}.  Additionally, since our approximation space now contains rational functions, standard polynomial approximation results do not hold for LSC-DG wedges.  Under a quasi-regular scaling assumption on elements in our mesh, we recover instead an $L^2$ approximation bound of the form
\[
\nor{u-\Pi_w u}_{L^2(K)} \leq C h^{N+1}\nor{\frac{1}{\sqrt{J}}}_{L^{\infty}(K)}\nor{\sqrt{J}}_{W^{N+1,\infty}(K)}\nor{u}_{W^{N+1,2}(K)},
\]
where the weighted projection $\Pi_w u$ is the projection onto the LSC-DG space, $h$ is the size of the element $K$, and $\nor{\cdot}_{W^{N+1,\infty}(K)}$ denotes the $L^{\infty}$ Sobolev norm of order $N+1$ over $K$.  For comparison, the projection error bound for curvilinear mappings using standard mapped bases replaces the Soblev in the above bound with an $L^{\infty}$ norm.  In other words, while the approximation error for standard curvilinear elements depends only on the extremal values of $J$, it depends additionally on the smoothness of $J$ when using rational LSC-DG basis functions.  

\subsubsection{Pyramids}

%Bergot, Cohen, and Durufle defined an orthogonal basis for the pyramid on the bi-unit cube using Jacobi polynomials \cite{bergot2010higher}
%\[
%\psi_{ijk}(a,b,c) = P_i^{0,0}(a)P_j^{0,0}(b) \LRp{\frac{1-c}{2}}^{\mu_{ij}} P_k^{2\mu_{ij} + 2}(c), \qquad 0\leq i,j \leq N, \quad 0 \leq k \leq N-\mu_{ij}
%\]
%where $\mu_{ij} = \max(i,j)$.  

Unlike hexahedral, wedge, and tetrahedral elements, the basis functions for the pyramid are rational in nature.  For a vertex-mapped pyramid $\pyr$, the approximation space and dimension $N_p$ are 
\[
B_N(\pyr) = P^N(\pyr) \oplus \sum\limits_{k=0}^{N-1}\LRc{\LRp{\frac{xy}{1-z}}^{N-k}x^i y^j, 0\leq i + j \leq k}, \quad N_p = \frac{(N+1)(N+2)(2N+3)}{6}, 
\]
where $P_N(\pyr)$ is the space of polynomials of total order $N$ on the physical pyramid.
%\[
%P_N(\pyr) = \LRc{x^iy^jz^k, \quad 0\leq i + j + k \leq N}.
%\]
For a vertex-mapped pyramid, we have that the change of variables factor $J \in B_1(\refpyr)$.  An appropriate quadrature rule is defined on the reference pyramid using a tensor product of Gauss-Legendre rules in the $r,s$ coordinate, and a Gauss-Legendre rule with weights $(2,0)$ in the $t$ coordinate.  

Since $J$ is again non-constant, each local mass matrix is distinct and dense, increasing storage costs for DG methods on pyramids.  However, it was shown in \cite{chan2015orthogonal} that there exists an orthogonal basis $\phi_{ijk}$ on vertex-mapped pyramids which spans $B_N(\pyr)$.  This basis is defined on the bi-unit cube as
\[
\phi_{ijk}(a,b,c) = \sqrt{\frac{C^N_{N-k}}{{w_iw_j}}} {\ell_i^{k}(a)\ell_j^{k}(b)}\LRp{\frac{1-c}{2}}^{k} P^{2k+3}_{N-k}(c), \quad 0\leq i,j \leq k, \quad 0\leq k \leq N,
\]
where $\ell_i^{k}(a), \ell^k_j(b)$ are order $k$ Lagrange polynomials at $(k+1)$ Gauss-Legendre nodes, and $P^{2k+3}_{N-k}(c)$ is the Jacobi polynomial of degree $k$ with order-dependent weight $2k+3$.  $w_i$ is the $i$th Gauss-Legendre quadrature weight, and the normalization constant $C^N_{N-k} = (N+2)/(2^{2k+2}(2k+3))$. 
The pyramid basis is then defined under the Duffy-type mapping from the bi-unit cube to the bi-unit right pyramid 
\begin{align*}
r &= (1+a)\LRp{\frac{1-c}{2}}-1, \quad s = (1+b)\LRp{\frac{1-c}{2}}-1,\quad t= c.
\end{align*}
 The mass matrix is diagonal for vertex-mapped pyramids under such a basis, and the entries of the mass matrix are the evaluation of $J$ at the quadrature points $a^{k}_i, b^{k}_j$.  As with hexahedral elements, the application of $M^{-1}$ requires only storage of the diagonals of each local mass matrix.  

\subsection{Variational formulation}
\label{sec:form}
We consider the acoustic wave equation on domain $\Omega$ with a free surface boundary condition $p=0$ on the boundary of the domain $\partial \Omega$.  This may be written in first order form
\begin{align*}
\frac{1}{\kappa}\pd{p}{\tau}{} + \Div u &= f\\
\rho\pd{\bm{u}}{\tau}{} + \Grad p &= 0,
\end{align*}
where $p$ is acoustic pressure, $\bm{u}$ is velocity, and $\rho$ and $\kappa$ are density and bulk modulus, respectively.  We assume also a triangulation of the domain $\Omega_h$ consisting of elements $K$ with faces $f$, and that $\rho$ and $\kappa$ are piecewise constant on each element.  We will denote the union of the faces $f$ as $\Gamma_h$.  Let $(p^-,\bm{u}^-)$ denote the solution fields on the face of an element $K$, and let $(p^+,\bm{u}^+)$ denote the solution on the neighboring element adjacent to that face.  We may then define the jump of $p$ and the normal jump and average of the vector velocity $\bm{u}$ componentwise
\[
\jump{p} = p^+ - p^-, \qquad \jump{\bm{u}} = \bm{u}^+ - \bm{u}^-, \qquad \avg{\bm{u}} = \frac{\bm{u}^+ + \bm{u}^-}{2}.
\]
Defining $\vect{n}^-$ as the outward unit normal on a given face, the (local) variational formulation for the discontinuous Galerkin method is then
\begin{align}
%\int_K \frac{1}{\kappa}\pd{p}{t}{}\phi^- \diff x = -\int_K \Div\bm{u}\phi^- \diff x + \int_{\partial K} \frac{1}{2}\LRp{\tau_p\jump{p} - \bm{n}\cdot \jump{\bm{u}} }\phi^- \diff x   &= 0\\
\int_K \frac{1}{\kappa}\pd{p}{\tau}{}\phi^- \diff x &= \int_K \bm{u}\cdot \Grad\phi^- \diff x + \int_{\partial K}\LRp{ \frac{1}{2}\tau_p\jump{p} - \bm{n}^-\cdot \avg{\bm{u}} }\phi^- \diff x   \\
\int_K \rho\pd{\bm{u}}{\tau}{}\bm{\psi}^- \diff x &= - \int_K \Grad p \cdot  \bm{\psi}^- \diff x + \int_{\partial K} \frac{1}{2}\LRp{\tau_u \jump{\bm{u}}\cdot \bm{n}^- - \jump{p}}\bm{\psi}^-\cdot \bm{n}^- \diff x,
\label{eq:skewform}
\end{align}
where $\tau_p = 1/\avg{\rho c}$, $\tau_u = \avg{\rho c}$, and $c^2 = \kappa/\rho$ is the speed of sound.  Only one equation has been integrated by parts, resulting in a skew-symmetric and energy-stable formulation (the choice of which equation to integrate by parts is arbitrary; both produce skew-symmetric formulations).  Proofs of energy stability typically rely on the exact evaluation of integrals \cite{malm2013stabilization}.  Due to the use of LSC-DG on wedges, physical basis functions become rational in nature and are not exactly integrable using standard quadrature rules.  The skew-symmetry of the variational form guarantees energy stability irregardless of quadrature rule \cite{warburton2013low}.  

While we adopt a skew-symmetric variational formulation when necessary for energy stability, we do not always compute with it.  Instead, we use the fact that the skew-symmetric form (\ref{eq:skewform}) is sometimes equivalent to the twice integrated-by-parts ``strong form'' \cite{hesthaven2007nodal}
\begin{align*}
\int_K \frac{1}{\kappa}\pd{p}{\tau}{}\phi^- \diff x &= -\int_K \Div\bm{u}\phi^- \diff x + \int_{\partial K} \frac{1}{2}\LRp{\tau_p\jump{p} - \bm{n}\cdot \jump{\bm{u}} }\phi^- \diff x  \\
\int_K \rho\pd{\bm{u}}{\tau}{}\bm{\psi}^- \diff x &= - \int_K \Grad p \cdot  \bm{\psi}^- \diff x + \int_{\partial K} \frac{1}{2}\LRp{\tau_u \jump{\bm{u}}\cdot \bm{n}^- - \jump{p}}\bm{\psi}^-\cdot \bm{n}^- \diff x.
\end{align*}
This holds, for example, if integration by parts holds when volume and surface integrals are replaced with quadrature approximations, and is related to the ``summation-by-parts'' property in finite differences.   

%\subsubsection{Element-specific energy-stable formulations}

The choice between the strong/skew formulations is governed by the choice between GL and SEM nodal bases for the hexahedron, which determines the quadrature rule on quadrilateral faces (see Figure~\ref{fig:surface_nodes}).  When SEM quadrature is used, surface integrals over the quadrilateral face are computed inexactly; as a result, the strong formulation may not be equivalent to the weak formulation, and may not be energy stable.  

\begin{figure}
\centering
\subfloat[SEM]{\includegraphics[width=.32\textwidth]{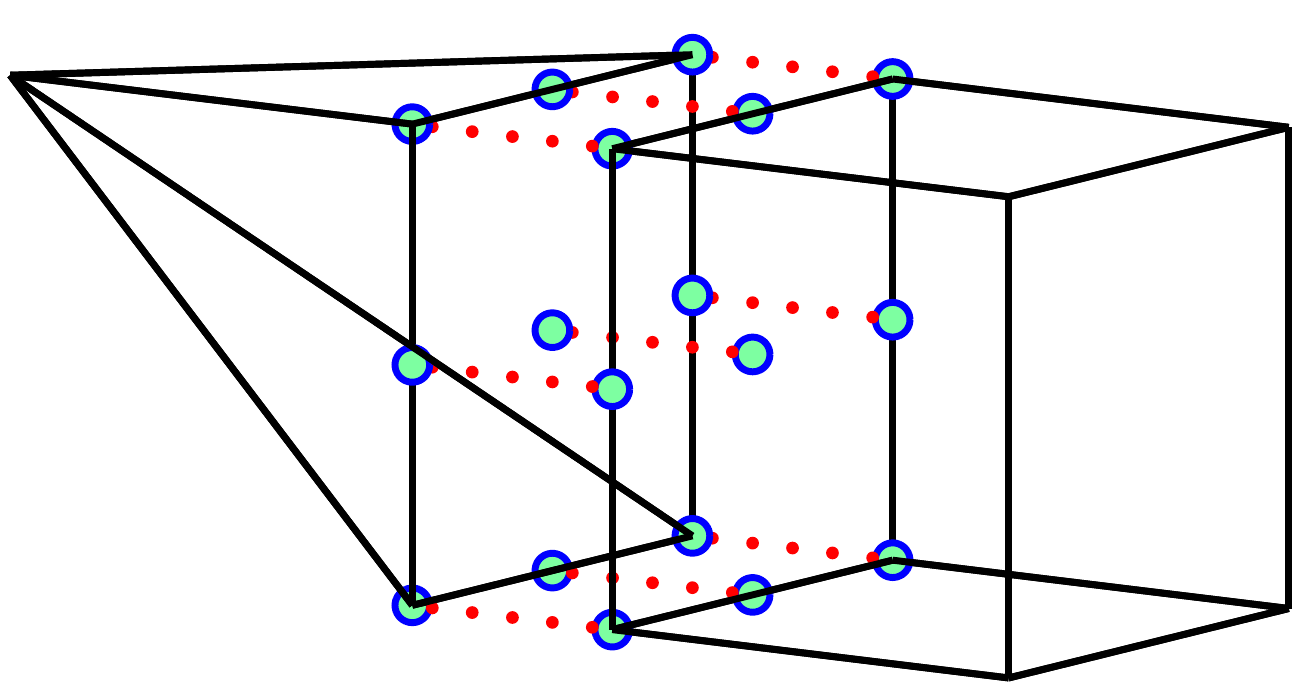}}
\hspace{2em}
\subfloat[Gauss-Legendre]{\includegraphics[width=.32\textwidth]{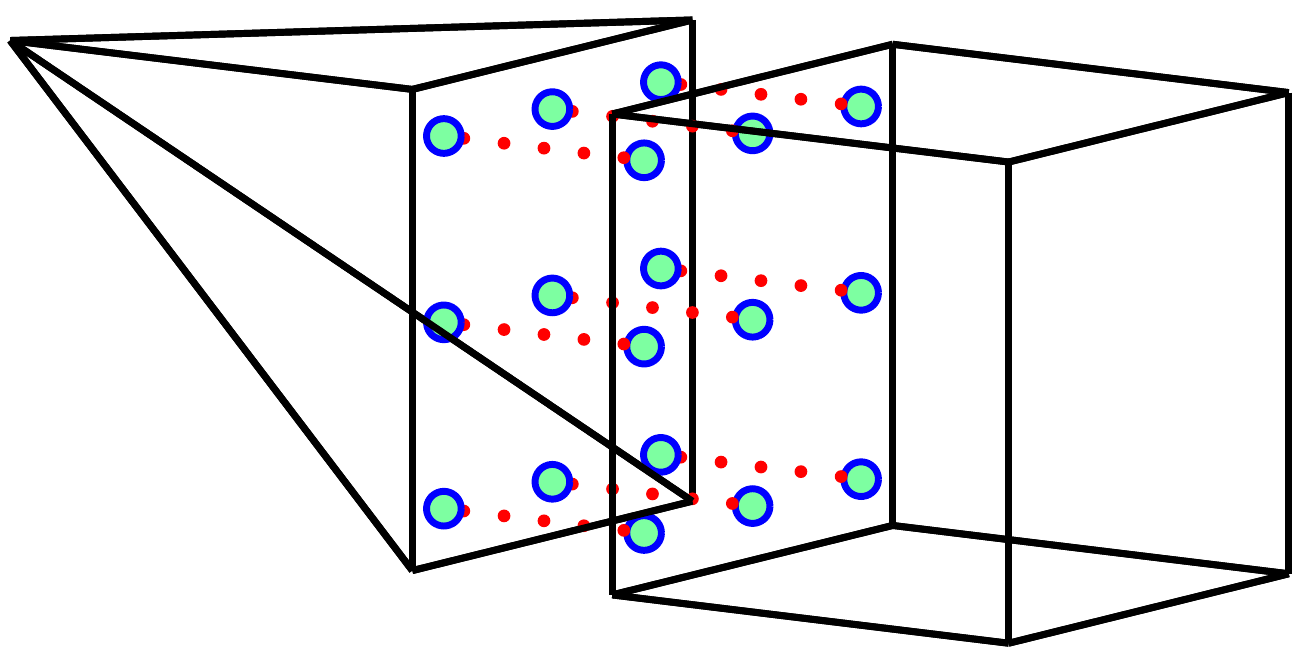}}
\caption{The choice of nodal basis for the hexahedron determines the quadrature rule on quadrilateral faces of both pyramids (shown above) and wedges.}
\label{fig:surface_nodes}
\end{figure}

For hexahedral elements, the strong and skew-symmetric forms are equivalent under both GL and SEM bases.  This was shown by Kopriva and Gassner in \cite{kopriva2010quadrature} and for SEM, relies on the exact cancellation of Gauss-Legendre-Lobatto quadrature errors in volume and surface integrals.  We may thus exploit the improved computational efficiency of the strong form for hexahedra for both SEM or GL nodal bases.  For tetrahedra, since the mapping is constant for planar simplices, volume integrals are computed exactly using quadrature free techniques, and surface integrals are computed exactly using quadrature. Thus, the skew-symmetric and strong forms are equivalent for vertex-mapped tetrahedra, and we adopt the strong form for improved computational efficiency.  

For the wedge, we are restricted to the skew symmetric form, since integrals over the rational LSC-DG basis are inexact for any polynomial cubature rule.  For the pyramid, the energy stability of the variational form depends on the choice of SEM or GL quadrature for the quadrilateral face.  Under Gauss-Legendre quadrature, surface integrals for the pyramid are exact, and both the strong and skew-symmetric form are equivalent.  However, if SEM quadrature is used for the quadrilateral face, the surface integral is inexact and we must use the skew-symmetric form for stability.  

We may now choose freely between a SEM or GL nodal basis for the hexahedron, which also determines the choice of quadrature on quadrilateral faces.  SEM nodes tends to be more efficient than Gauss-Legendre nodes --- with Gauss-Legendre quadrature, since the nodal points lie in the interior of the hexahedron, an additional evaluation step is necessary to compute the solution at points on the surface.  Since the SEM nodal basis contains both volume and surface quadrature points as degrees of freedom, evaluation of the solution on the surface requires only retrieval of a single degree of freedom, and this extra step is avoided.  Additionally, due to the Lagrange property of the SEM nodal basis functions, the surface contributions are sparser for the SEM formulation than for the GL formulation.  

However, it is also known that the inexact integration in SEM quadrature reduces the accuracy of the resulting solution \cite{kopriva2010quadrature}, though the inner products generated by both underintegrated SEM quadrature and full Gauss-Legendre quadrature are known to be equivalent (with constants that do not grow in $N$).  We will also show in Section~\ref{sec:elemverify} that SEM underintegration can reduce the observed order of convergence for smooth solutions.  Additionally, the use of SEM on hybrid meshes also restricts the pyramid to the skew-symmetric form, which is typically less computationally efficient than the strong form, while the use of Gauss-Legendre nodes on the hex allows for the use of the strong form for the pyramid.  %On hex-dominant meshes, we expect that this drawback for a small number of pyramidal elements.  

We naturally arrive at two energy stable formulations based on the two nodal bases for the hexahedron, which are summarized for each element type in Table~\ref{table:forms}.  We refer to the formulation using SEM hexahedra as the ``SEM formulation'', and the formulation using Gauss-Legendre hexahedra as the ``GL formulation''.  We will compare the accuracy and efficiency of each formulation in Section~\ref{sec:numerics}.  

\begin{table}
\centering                                                                                   
\begin{tabular}{|c|c|c|}
\hline                                                   
 & SEM Formulation & GL Formulation \\ 
\hline                                                   
Tetrahedra & Nodal, strong form & Nodal, strong form \\ 
\hline                                                   
Pyramid & Semi-nodal, skew-symmetric form & Semi-nodal, strong form\\ 
\hline                                                   
Wedge & Modal LSC-DG, skew-symmetric form & Modal LSC-DG, skew-symmetric form \\ 
\hline                                                   
Hexahedra & Nodal SEM, strong form & Nodal Gauss-Legendre, strong form\\ 
\hline                                                   
\end{tabular}
\caption{Summary of stable bases and formulations for element type-specific variational forms.}
\label{table:forms}
\end{table}

\subsection{Element-specific operations}

Our numerical implementation tailors both the local variational formulation and computational operations to each specific type of element.  For each element type, we construct three kernels 
\begin{enumerate}
\item \textbf{Volume kernel:} compute the volume contribution (integrals over the interior of the element $K$).
\item \textbf{Surface kernel:} compute the surface contribution (integrals over the surface of the element $\partial K$). 
\item \textbf{Update kernel:} apply the inverse of the mass matrix where necessary, execute a step of Adams-Bashforth, and evaluate/store the solution at cubature points on the surface.  
\end{enumerate}
%For tetrahedra, pyramids, and wedges, the surface and update kernels follow the implementations described in \cite{chan2015orthogonal} for pyramidal DG with cubature, and the surface kernels for hexahedra are implemented as described in Section~\ref{sec:hexspecific}.  

Algorithms describing the implementation of relevant kernels for each element type are given in the following sections.  For the wedge, pyramid, and tetrahedron, surface and update kernels are very similar to those given in \cite{chan2015orthogonal} for evaluation of the solution at surface cubature points.  Their implementation differs only by the specific surface cubature, which is constructed by combining appropriate cubatures for triangular and quadrilateral faces.  For triangular faces, we use quadrature rules for the triangle computed by Xiao and Gimbutas \cite{xiao2010quadrature} that are exact for polynomials of degree $2N$.  For quadrilateral faces, we use a tensor product Gauss-Legendre or Gauss-Legendre-Lobatto (SEM) cubature rule with $(N+1)^2$ points, depending on which hexahedral nodal basis is chosen.

\subsubsection{Hexahedral elements}
\label{sec:hexspecific}

The volume contribution for the hexahedron may be computed efficiently by exploiting both the Lagrange property of the nodal basis and the tensor-product form of the basis functions.  Basis functions for the hexahedron are constructed 
\[
\phi_{ijk}(r,s,t) = \ell_i(r)\ell_j(s)\ell_k(t),
\]
where $\ell_i(r)$ is the Lagrange polynomial at the $i$th node in the $r$ direction, and similarly for $\ell_j(s),\ell_k(t)$.  This implies that the evaluation of volume integrals using quadrature reduces to
\[
%\int_K \phi_{lmn}\pd{u}{x} = 
\int_{\widehat{K}} \phi_{lmn}\pd{u}{x} J = \sum_{l',m',n'=1}^{N+1} w_{l',m',n'} J_{l',m',n'} \left.\pd{u}{x}{}\right|_{l',m',n'} \left.\phi_{lmn}\right|_{l',m',n'} = w_{lmn}J_{lmn} \left.\pd{u}{x}{}\right|_{l,m,n}
\]
due to the Lagrange property of $\phi_{lmn}$ at quadrature nodes.  The values of derivatives $\pd{u}{x}$ require computation of $\pd{u}{r},\pd{u}{s},\pd{u}{t}$ at nodal points, which reduces to the computation of 1D derivatives
\[
\left.\pd{u}{r}{}\right|_{lmn} 
= \sum_{i=1}^{N+1} \pd{\ell_i(r_{l})}{r}{} \sum_{j=1}^{N+1} \ell_j(s_{m})\sum_{k=1}^{N+1} u_{ijk}\ell_k(t_{n}) = \sum_{i=1}^{N+1} \pd{\ell_i(r_{l})}{r}{} u_{imn}.
\]
The integral is then scaled by $M^{-1}_{lmn,lmn} = \LRp{w_{lmn}J_{lmn}}^{-1}$, removing the weight and geometric factor multiplying the derivative.  This implementation is described in detail in Algorithm~\ref{alg:hex_vol}.%, as well as in \cite{vNek} for the related continuous Galerkin Spectral Element Method on hexahedral grids.  

\begin{algorithm}
\begin{algorithmic}[1]
\Procedure{Hexahedron Volume kernel}{}
\State Load nodal values $u_{ijk}$ (for $1 \leq i,j,k \leq N+1$) and the 1D operator $D_{ij}$ into shared memory.  
\For{each node $\vect{x}_{ijk}$}
\State Compute derivatives with respect to reference coordinate $r,s,t$.  
\[
\pd{u(\vect{x}_{ijk})}{r}{} = \sum_{l = 1}^{N+1} D_{il} u_{ljk}, \qquad 
\pd{u(\vect{x}_{ijk})}{s}{} = \sum_{l = 1}^{N+1} D_{jl} u_{ilk}, \qquad 
\pd{u(\vect{x}_{ijk})}{t}{} = \sum_{l = 1}^{N+1} D_{kl} u_{ijl}.
\]
\State Scale by change of variables factors $\pd{rst}{xyz}{}$ to compute $\pd{u}{x}{}, \pd{u}{y}{}, \pd{u}{z}{}$ at point $\vect{x}_{ijk}$.  
\[
%\pd{u(\vect{x}_{ijk})}{x}{} = \pd{u(\vect{x}_{ijk})}{r}{} \pd{r}{x}{} + \pd{u(\vect{x}_{ijk})}{s}{} \pd{s}{x}{} + \pd{u(\vect{x}_{ijk})}{t}{} \pd{t}{x}{}.
\pd{u}{x}{} = \pd{u}{r}{} \pd{r}{x}{} + \pd{u}{s}{} \pd{s}{x}{} + \pd{u}{t}{} \pd{t}{x}{}.
\]
\EndFor
\EndProcedure
\end{algorithmic}
\caption{Algorithm for the hexahedron volume kernel (both SEM and GL).}
\label{alg:hex_vol}
\end{algorithm}

The computation of numerical fluxes requires the evaluation of the solution on the surface of a hexahedron.  This is done in the update kernel for the hexahedron, and is described in Algorithm~\ref{alg:hex_upd}.  For a 1D nodal basis at GL points, we may evaluate the solution at endpoints $\pm 1$ via
\[
u(-1) = \sum_{k=1}^{N+1} \ell_k(-1) u_k, \qquad u(-1) = \sum_{k=1}^{N+1} \ell_k(1) u_k= \ell_{N-k}(-1) u_k, 
\]
where we have used symmetry of the GL points across $0$.  We store values of $\ell_k(-1)$ in an array $V^f_k = \ell_k(-1)$.  
\begin{algorithm}
\begin{algorithmic}[1]
\Procedure{Update kernel}{}
\State Load nodal values $u_{ijk}$ (for $1 \leq i,j,k \leq N+1$) and $V^f$ into shared memory.  
\For{each $1\leq i,j,k\leq N+1$}
%\For{each faces $1,\ldots, 6$ of the hexahedron}
\State For face cubature points $\vect{x_{jk}}$ (face $r = \pm 1$), $\vect{x_{ik}}$ (face $s = \pm 1$), and $\vect{x_{ij}}$ (face $t = \pm 1$).  
\begin{align*}
u( \vect{x_{jk}}; r = -1 ) &= \sum_{m=1}^{N+1}V^f_m u_{mjk}, \qquad u( \vect{x_{jk}}; r = 1 ) = \sum_{m=1}^{N+1}V^f_{N-m} u_{mjk}, \\
u( \vect{x_{ik}}; s = -1 ) &= \sum_{m=1}^{N+1}V^f_m u_{imk}, \qquad u( \vect{x_{ik}}; s = 1 ) = \sum_{m=1}^{N+1}V^f_{N-m} u_{imk}, \\
u( \vect{x_{ij}}; t = -1 ) &= \sum_{m=1}^{N+1}V^f_m u_{ijm}, \qquad u( \vect{x_{ij}}; t = 1 ) = \sum_{m=1}^{N+1}V^f_{N-m} u_{mjk}.
\end{align*}
%\EndFor
\EndFor
\EndProcedure
\end{algorithmic}
\caption{Algorithm for the GL hexahedron update kernel.}
\label{alg:hex_upd}
\end{algorithm}
For a SEM basis, $\ell_{k}(-1) = \delta_{k1}$, and the evaluation of the basis on the boundary reduces to the loading of the proper nodal degrees of freedom on each face.  

The computation of the hex surface contribution differs slightly between SEM and Gauss-Legendre nodal bases.  We compute contributions face by face, mapping surface integrals to the reference quadrilateral.  For example, on the face corresponding to $t=-1$, this gives
\[
\int_{r}\int_s \ell_{ijk} u J^s = \int_{r}\int_s \ell_{i}(r)\ell_j(s) \ell_k(-1) \left.uJ^s\right|_{r,s,-1}= w_i w_j \ell_{k}(-1) \left.u J^s\right|_{r_i,s_j,-1}.
\]
where $w_i, w_j$ are 1D quadrature weights and $J^s$ is the determinant of the Jacobian mapping from the physical to reference quadrilateral face.   For a Gauss-Legendre basis, $\ell_{k}(-1)$ must be evaluated explicitly.  For the SEM formulation, the evaluation of $\ell_k(\pm 1)$ may be skipped, and further optimizations may be done by noting that by the inverse mass matrix may be premultiplied into the geometric factors, and that due to the Lagrange property of nodal polynomials, surface RHS contributions for interior nodes are zero.  This quantity is then scaled by the inverse mass matrix to compute the surface contribution to the RHS, and the full procedure is outlined in Algorithm~\ref{alg:hex_surf}.

\begin{algorithm}
\begin{algorithmic}[1]
\Procedure{Hexahedron Surface kernel}{ replace $u$ with an appropriate numerical flux.}
\For{each $1\leq i,j,k\leq N+1$}
\State For faces $r \pm 1$, compute the surface RHS contribution
\[
M^{-1}_{ijk,ijk} w_jw_k \left.J^s\right|_{s_j, t_k}\ell_i(\pm 1) u(\pm 1, s_j, t_k).
\]
\State Compute similar contributions for faces $s,t\pm 1$
\[
M^{-1}_{ijk,ijk} w_iw_k \left.J^s\right|_{r_i, t_k}\ell_j(\pm 1) u(r_i, \pm 1, t_k), \qquad 
M^{-1}_{ijk,ijk} w_iw_j \left.J^s\right|_{r_i, s_j}\ell_k(\pm 1) u(r_i, s_j,\pm 1). 
\]
\EndFor
\EndProcedure
\end{algorithmic}
\caption{Algorithm for the hexahedron GL surface kernel.  }
\label{alg:hex_surf}
\end{algorithm}

\subsubsection{Tetrahedral elements}

Assuming Lagrange polynomials $\ell_j(r,s,t)$ defined by $N_p$ nodal points $(r_i,s_i,t_i)$ on a tetrahedron, we may define operators which evaluate derivatives at those same nodal points
\[
D^r_{ij}  = \pd{\ell_j(r_i,s_i,t_i)}{r}{}, \qquad D^s_{ij}  = \pd{\ell_j(r_i,s_i,t_i)}{s}{}, \qquad D^t_{ij}  = \pd{\ell_j(r_i,s_i,t_i)}{t}{}.
\]
As shown in \cite{hesthaven2007nodal}, RHS contributions may be computed using only the above derivative operators 
\[
S^x_{ij} = \int_{K} \ell_i \pd{\ell_j}{x}, \qquad M^{-1}S^x \vect{u} = \pd{u(\vect{x})}{x}{},
\]
where $\vect{x},\vect{u}$ are vectors of the nodal positions $\vect{x}_i$ and nodal degrees of freedom $u_j$.

\begin{algorithm}
\begin{algorithmic}[1]
\Procedure{Tetrahedron Volume kernel}{}
\State Load nodal degrees of freedom $u_j$ into shared memory.
\For{each node $\vect{x}_i$, $i = 1,\ldots, N_p$}
\State Compute derivatives with respect to reference coordinates $r,s,t$
\[
\pd{u(\vect{x}_i)}{r}{} = \sum_{j = 1}^{N_p} D^r_{ij} u_j, \qquad \pd{u(\vect{x}_i)}{s}{} =  \sum_{j = 1}^{N_p}  D^s_{ij} u_j, \qquad \pd{u(\vect{x}_i)}{t}{} = \sum_{j = 1}^{N_p}  D^t_{ij}u_j
\]
\State Scale by change of variables factors $\pd{rst}{xyz}{}$ to compute $\pd{u}{x}{}, \pd{u}{y}{}, \pd{u}{z}{}$ at point $\vect{x}_{i}$.  
\EndFor
\EndProcedure
\end{algorithmic}
\caption{Algorithm for the tetrahedron volume kernel.}
\label{alg:tet_vol}
\end{algorithm}

In the update kernel, we evaluate solution fields at cubature points on the surface.  For nodal tetrahedra, this evaluation may also be done face-by-face to reduce cost at high orders, since traces of solution fields depend only on face nodal values.  

\subsubsection{Wedge elements}

Evaluation of the skew-symmetric form involves the computation of two types of volume integrals
\[
\int_K \tilde{\phi_i} \pd{u}{x}{}, \qquad \int_K \pd{\tilde{\phi_i}}{x}{}u, \qquad \tilde{\phi}_{i} = \frac{\phi_{i}}{\sqrt{J}}.
\]
Since derivatives in $x$ (and similarly for $y,z$) of the LSC-DG basis are given by
\[
\pd{\tilde{\phi}}{x} = \pd{\LRp{\phi/\sqrt{J}}}{x} = \frac{1}{\sqrt{J}}\LRp{\pd{\phi}{x} - \frac{\phi}{2J}\pd{J}{x}},
\]
we precompute and store the values of the physical gradient $\Grad_{xyz}{J}$ at each cubature point.  Derivatives of $u$ and the first volume integral $\int_K \tilde{\phi_i} \pd{u}{x}{}$ reduce to integrals over the reference element
\[
u = \sum_{j=1}^{N_p} u_j \frac{\phi_j}{\sqrt{J}}, \qquad \pd{u}{x} = \sum_{j=1}^{N_p} u_j \LRp{\pd{\phi_j}{x} - \frac{\phi_j}{2J}\pd{J}{x}},\qquad \int_K \tilde{\phi_i} \pd{u}{x}{} = \int_{\widehat{K}} \frac{\phi_{i}}{\sqrt{J}} \pd{u}{x}{} J = \int_{\widehat{K}} \phi_{i} \pd{u}{x}{}.
\]
Similarly, the second volume integral $\int_K \pd{\tilde{\phi_i}}{x}{}u$ reduces to 
\begin{align*}
\int_K \pd{\tilde{\phi_i}}{x} u &=  \int_{\widehat{K}} \LRp{\pd{\phi_i}{x} - \frac{\phi_i}{2J}\pd{J}{x}} u = 
\int_{\widehat{K}} \LRp{\pd{\phi_i}{r}{} \pd{r}{x}{} + \pd{\phi_i}{s}{} \pd{s}{x}{} + \pd{\phi_i}{t}{} \pd{t}{x}{} - \phi_i \frac{1}{2J}\pd{J}{x}{}}u \\
&= \int_{\widehat{K}}
\LRs{\pd{\phi_i}{r}{}, \pd{\phi_i}{s}{}, \pd{\phi_i}{t}{}, \phi_i}
\cdot
\LRs{u\pd{r}{x}{}, u\pd{s}{x}{}, u\pd{t}{x}{}, -u \frac{1}{2J}\pd{J}{x}{}}.
\end{align*}
For LSC-DG, we split the evaluation of the skew-symmetric formulation into the computation of the trial integrand (involving the solution) and the test integrand (involving test functions $\phi_i$).  Intermediate values used in the test integrand are stored in global memory.  In the first step, derivatives of the solution are computed, and values of the solution at cubature points are premultiplied by appropriate change of variable factors, as indicated in the computation of the second volume integral above.  %All geometric factors, including physical derivatives of $\sqrt{J}$, are loaded in the first step.  
In the second step, test functions and their derivatives in the reference $r,s,t$ coordinates are evaluated at cubature points and summed up to compute all integrals.  

The splitting of the wedge volume kernel into two steps allows us to load geometric factors and derivatives of $J$ only in the first step, while maintaining fast coalesced memory access patterns.  Additionally, splitting into two kernels and writing intermediate values to global memory avoids the overuse of shared memory, which can reduce workgroup occupancy.  

Finally, since the wedge is a tensor product of triangle and line elements, we may additionally decompose interpolation and derivative operators into 2D triangle and 1D operators.  Exploiting the tensor product nature of these operators on the wedge leads to a reduction in the cost of quadrature.  We assume that wedge basis functions may be decomposed as
$\phi_i(r,t)\phi_j(s)$, where $\phi_i(r,t)$ are basis functions over the triangle.  Then, 
\[
u(r,s,t) = \sum_{i=1}^{N_p^{\rm tri}} \sum_{j=1}^{N+1} u_{ij}\phi_i(r,t)\phi_j(s)
\]
where $u_{ij}$ are coefficients for the wedge, decomposed into triangle and 1D indices $i$ and $j$, respectively.  Our cubature is defined also as a tensor product of a triangle and 1D cubature rule with $N_c^{\rm tri}$ and $(N+1)$ points, respectively.  For cubature points $\vect{x}_{kl}$ and weights $w_{kl}$, where $k$ is the index for a triangle cubature point and $l$ is the index for a 1D cubature rule, we precompute the values of $\phi_i(r,t), \phi_j(s)$ (as well as their derivatives) 
\begin{align*}
V_{ik} &= \phi_i(r_k,t_k), \qquad V^{r,t}_{ik} = \pd{\phi_i(r_k,t_k)}{r,t}{} \qquad \quad 0 \leq k \leq N_c^{\rm tri},\\
V^{\rm 1D}_{jl} &= \phi_j(s_l), \qquad D^{\rm 1D}_{jl} = \pd{\phi_j(s_l)}{s}{} \qquad 0\leq l \leq N+1,
\end{align*}
which are used in Algorithms~\ref{alg:wedge_vol1} and \ref{alg:wedge_vol2}.

\begin{algorithm}
\begin{algorithmic}[1]
\Procedure{Wedge volume kernel Part 1}{}
\State Load coefficients $u_{ij}$ and 1D interpolation/derivative operators $V^{\rm 1D} , D^s_{\rm 1D}$ into shared memory.  
\For{each cubature point $\vect{x}_k$, $k = 1,\ldots N_c$}
\State Compute (using the tensor product form) the solution and reference coordinate derivatives 
\begin{align*}
u(\vect{x}_k) &= \sum_{i = 1}^{N_p^{\rm tri}} V_{ki} \sum_{j=1}^{N+1} V^{\rm 1D}_{kj} u_{ij},\\ 
\pd{u(\vect{x}_k)}{r,t}{} &= \sum_{i = 1}^{N_p^{\rm tri}} D^{r,t}_{ki} \sum_{j=1}^{N+1} V^{\rm 1D}_{kj} u_{ij}, \qquad
\pd{u(\vect{x}_k)}{s}{} = \sum_{i = 1}^{N_p^{\rm tri}} V_{ki} \sum_{j=1}^{N+1} D^{\rm 1D}_{kj} u_{ij}.
\end{align*}
\State Scale by change of variables factors $\pd{rst}{xyz}{}$ to compute $\pd{u}{x}{}, \pd{u}{y}{}, \pd{u}{z}{}$ at cubature point $\vect{x}_{k}$.  
\State Write (to global memory) intermediate derivatives and premultiplied values at cubature points
\[
\Grad_{xyz} u, \qquad \LRs{u\pd{r}{x}{}, u\pd{s}{x}{}, u\pd{t}{x}{}, -u \frac{1}{2J}\pd{J}{x}{}}.
\]
\EndFor
\EndProcedure
\end{algorithmic}
\caption{Algorithm for part 1 of the wedge volume kernel.}
\label{alg:wedge_vol1}
\end{algorithm}

\begin{algorithm}
\begin{algorithmic}[1]
\Procedure{Wedge Volume kernel part 2}{}
\For{each cubature point $\vect{x}_k$, $k = 1,\ldots N_c$}
\State Load intermediate values at cubature points
\[
\Grad_{xyz} u, \qquad \LRs{u\pd{r}{x}{}, u\pd{s}{x}{}, u\pd{t}{x}{}, -u \frac{1}{2J}\pd{J}{x}{}}.
\]
\State Compute integrals via cubature
\begin{align*}
\int_{\widehat{K}} \phi_{ij} \pd{u}{x} &= \sum_{k=1}^{N^{\rm tri}_c} \LRp{V_{ki}}^T \sum_{l=1}^{N+1}\LRp{V^{\rm 1D}}^T_{li} \pd{u(\vect{x}_{kl})}{x}{}w_{kl}\\
\int_{\widehat{K}} \pd{\phi_{ij}}{x} u &=  \sum_{k=1}^{N^{\rm tri}_c} \sum_{l=1}^{N+1}w_{kl}\LRs{D^r_{ki} V^{\rm 1D}_{lj},V_{ki}{D^{\rm 1D}_{lj}},D^t_{ki}V^{\rm 1D}_{lj} ,V_{ki}V^{\rm 1D}_{lj}} \cdot \LRs{u\pd{r}{x}{}, u\pd{s}{x}{}, u\pd{t}{x}{}, -u \frac{1}{2J}\pd{J}{x}{}}.
\end{align*}
\EndFor
\EndProcedure
\end{algorithmic}
\caption{Algorithm for part 2 of the wedge volume kernel.}
\label{alg:wedge_vol2}
\end{algorithm}
Compared to a sum-factorization approach, the tensor product sum in Algorithms~\ref{alg:wedge_vol1}, \ref{alg:wedge_vol2} does not decrease the total number of computations.  We found that implementations of sum factorization required larger amounts of either shared or register memory, both of which decrease occupancy.  Instead, we exploit the fact that triangle operators are constant for each 1D summation, increasing data reuse.  

For Algorithm~\ref{alg:wedge_vol2}, we again reuse loaded triangle operators, which are constant for each 1D summation.  Additionally, we may load premultiplied information and execute dot products in the summation using \verb+float4+ data and operations, which are pipelined for faster execution on certain GPUs.

\subsubsection{Pyramidal elements}

The form of the pyramid mapping and the semi-nodal pyramid basis lend themselves to simpler evaluation of integrals.  We introduce a quadrature-free method of computing volume contributions, which improves upon the quadrature-based method for the pyramid basis described in \cite{chan2015orthogonal}.  Let $u$ be a function in the span of the pyramid basis on the bi-unit reference pyramid $\refpyr$; then,
\begin{align}
\int_{\refpyr} u(r,s,t) \phi_{lmn}(r,s,t) &= \int_a \frac{\ell^n_l(a)}{\sqrt{w^n_l}} \int_b\frac{\ell^n_m(b)}{\sqrt{w^n_m}}\int_c u(a,b,c)  \frac{P^{2n+3}_{N-n}(c)}{\sqrt{C^N_{N-n}}} \LRp{\frac{1-c}{2}}^{2+n} \nonumber\\
&= \sqrt{\frac{w^n_l w^n_m}{C^N_{N-n}}} \int_c u(a^n_l,b^n_m,c)  P^{2n+3}_{N-n}(c) \LRp{\frac{1-c}{2}}^{2+n}.
\label{eq:pyr_int}
\end{align}
%As a result, integrals over the pyramid may be reduced to 1D integrals in $c$ for each degree of freedom.  

Since we assume $u$ is in the pyramidal space, it may be represented with degrees of freedom $u_{ijk}$.  Evaluating $\pd{u}{r}{}$ using the chain rule gives
%Let $\pd{u}{a}{}$ be represented using degrees of freedom $u^a_{ijk}$, which may be determined using standard tensor product techniques.  Then, 
\[
\pd{u}{a}{}\pd{a}{r}{} = \LRp{\frac{2}{1-c}} \sum_{k=0}^N\sum_{i=0}^k\sum_{j=0}^k\frac{1}{\sqrt{w^k_iw^k_j}}\pd{\ell^k_i(a)}{a}{}\ell^k_j(b) \frac{P^{2k+3}_{N-k}(c)\LRp{\frac{1-c}{2}}^{k}}{\sqrt{C^N_{N-k}}}u_{ijk}.
\]
Putting this together with (\ref{eq:pyr_int}), we have
\[
\int_{\refpyr}\phi_{lmn}\pd{u}{r} = \sum_{k=0}^N \int_c \frac{P^{2k+3}_{N-k} P^{2n+3}_{N-n}\LRp{\frac{1-c}{2}}^{1+k+n}}{\sqrt{C^N_{N-n}C^N_{N-k}}}  \sum_{i=0}^k\sum_{j=0}^k \sqrt{\frac{w^n_l w^n_m}{{w^k_iw^k_j }}}{\pd{\ell^k_i(a^n_l)}{a}{}\ell^k_j(b^n_m)} u_{ijk}.
\]
Denoting
\[
D^r_{lmn,ijk} = \sqrt{\frac{w^n_l w^n_m}{{w^k_iw^k_j }}}{\pd{\ell^k_i(a^n_l)}{a}{}\ell^k_j(b^n_m)}M^c_{nk}, \quad M^c_{nk} = \int_c \frac{P^{2k+3}_{N-k} P^{2n+3}_{N-n}\LRp{\frac{1-c}{2}}^{1+k+n}}{\sqrt{C^N_{N-n}C^N_{N-k}}},
\]
we have that $D^r$ behaves as a weak derivative operator over the reference element
\[
\int_{\refpyr}\phi_{lmn}\pd{u}{r}{} = \sum_{k=0}^N\sum_{i=0}^k\sum_{j=0}^k D^r_{lmn,ijk} u_{ijk}.
\]
We may similarly define $D^s$ such that
\[
\int_{\refpyr}\phi_{lmn} \pd{u}{s}{}= D^s_{lmn,ijk} = \sqrt{\frac{w^n_l w^n_m}{{w^k_iw^k_j }}}{\ell^k_i(a^n_l)\pd{\ell^k_j(b^n_m)}{b}{}}M^c_{nk}.
\]
To take derivatives with respect to the $t$ coordinate, we use the chain rule for $\pd{u}{t}{}$
\begin{align*}
\pd{u}{t}{} &= \pd{u}{a}{} \pd{a}{t}{}+ \pd{u}{b}{} \pd{b}{t}{}+\pd{u}{c}{}  = 
\pd{u}{a}{} \LRp{\frac{1+a}{2}}\LRp{\frac{2}{1-c}} + \pd{u}{b}{} \LRp{\frac{1+b}{2}}\LRp{\frac{2}{1-c}} +\pd{u}{c}{} \\
&= \LRp{\frac{1+a}{2}}\pd{u}{r}{}  + \LRp{\frac{1+b}{2}}\pd{u}{s}{}  +\pd{u}{c}{} 
\end{align*}
An operator for the $t$ derivative is constructed similarly.  We first define a derivative operator in $c$
\begin{align*}
D^c_{lmn,ijk} &= \sqrt{\frac{w^n_l w^n_m}{{w^k_iw^k_j }}}{\ell^k_i(a^n_l){\ell^k_j(b^n_m)}} \int_c \pd{\LRs{P^{2k+3}_{N-k}(c)\LRp{\frac{1-c}{2}}^{k}}}{c}{}\frac{P^{2n+3}_{N-n}(c)\LRp{\frac{1-c}{2}}^{2+n}}{\sqrt{C^N_{N-n}C^N_{N-k}}}.
\end{align*}
Using this, $D^t$ may then be defined as
\begin{align*}
D^t_{lmn,ijk} &= \LRp{\frac{1+a^k_i}{2}} D^r_{lmn,ijk} + \LRp{\frac{1+b^k_j}{2}}D^s_{lmn,ijk} + D^c_{lmn,ijk}.
\end{align*}
We now exploit the fact that both geometric factors $\pd{rst}{xyz}$ and the determinant of the Jacobian $J$ are constant in the $c$ (and thus $t$) direction for all vertex-mapped pyramids \cite{bergot2013high, chan2015orthogonal}.  As a result, geometric factors are constant in the index $n$.  This implies that the integral of a derivative in the $x$ coordinate over a physical pyramid may be given as
\begin{align*}
\int_{\pyr} \phi_{lmn} \pd{u}{x} &= \int_{\refpyr} \phi_{lmn}\LRp{ \pd{u}{r}\pd{r}{x} +  \pd{u}{s}\pd{s}{x} +  \pd{u}{t}\pd{t}{x}} J \\
&= J_{lm} \LRc{\LRp{\pd{r}{x}{}}_{lm}\sum_{ijk} D^r_{lmn,ijk} u_{ijk} +  \LRp{\pd{s}{x}{}}_{lm}\sum_{ijk} D^s_{lmn,ijk} u_{ijk} +\LRp{\pd{t}{x}{}}_{lm}\sum_{ijk} D^t_{lmn,ijk} u_{ijk}}.
\end{align*}
The above expression is simply multiplication by (square) semi-nodal operators $D^r,D^s, D^t$ and entry-wise scalings by geometric factors, similarly to nodal methods for hexahedral or tetrahedral elements.  Furthermore, the scaling by $J_{lm}$ is removed when applying the inverse of the diagonal mass matrix 
\[
M_{lmn,lmn}^{-1} = \frac{1}{J_{lm}}.
\]
When computing integrals in the skew-symmetric form, we may simply apply the transpose of each operator
\[
\int_{\refpyr} \pd{\phi_{lmn}}{r} u J = \sum_{ijk} \LRp{D^r}^T_{lmn,ijk}J_{ij} u_{ijk}.
\]
and similarly for $\LRp{D^s}^T, \LRp{D^t}^T$.  In this case, the factor of $J$ is not cancelled out after multiplying by $M^{-1}$.  The pyramid kernel is described in detail in Algorithm~\ref{alg:pyr_vol}.  Similarly to the wedge, the algorithm is made up of two parts (computation of the trial and test portion).  However, unlike the wedge, the semi-nodal nature of the pyramid basis allows us to store intermediate values efficiently in shared memory.  

\begin{algorithm}
\begin{algorithmic}[1]
\Procedure{Pyramid Volume kernel}{}
\State Load semi-nodal degrees of freedom $u_i$ into shared memory.
\For{each semi-nodal basis function $j = 1,\ldots, N_p$}
\State Store premultiplied $u_j$ by geometric change of variables and mapping factors
\[
u^r_j = \pd{r(\vect{x}_j)}{x}{} u_j J(\vect{x}_j), \qquad u^s_j = \pd{s(\vect{x}_j)}{x}{} u_jJ(\vect{x}_j), \qquad u^t_j = \pd{t(\vect{x}_j)}{x}{} u_jJ(\vect{x}_j)
\]
\EndFor
\For{each semi-nodal basis function $i = 1,\ldots, N_p$}
\State Compute weak derivatives (no integration by parts) with respect to reference coordinates $(r,s,t)$ using semi-nodal derivative operators
\[
\int_{\widehat{K}}\phi_i\pd{u}{r}{} = \sum_{j=1}^{N_p}D^r_{ij} u_j, \qquad
\int_{\widehat{K}}\phi_i\pd{u}{s}{} = \sum_{j=1}^{N_p}D^s_{ij} u_j, \qquad
\int_{\widehat{K}}\phi_i\pd{u}{t}{} = \sum_{j=1}^{N_p}D^t_{ij} u_j.
\]
\State Scale by change of variables factors $\pd{rst}{xyz}{}$ to compute weak physical derivatives and RHS contributions. 
\[
\int_{\widehat{K}}\phi_i\pd{u}{x}{} = \int_{\widehat{K}}\phi_i \LRp{\pd{u}{r}{} \pd{r(\vect{x}_i)}{x}{} + \pd{u}{s}{} \pd{s(\vect{x}_i)}{x}{} + \pd{u}{t}{} \pd{t(\vect{x}_i)}{x}{}}
\]
\State Compute weak derivatives (integrated by parts) with respect to physical coordinates $(x,y,z)$ using transposed semi-nodal derivative operators and premultiplied degrees of freedom.  Scale by the inverse mass matrix to compute RHS contributions.  
\begin{align*}
M^{-1}_{ii}\int_{K}\pd{\phi_i}{x}{} u &= M^{-1}_{ii}\int_{\widehat{K}}\LRp{\pd{\phi_i}{r}{}\pd{r}{x} + \pd{\phi_i}{s}{}\pd{s}{x} + \pd{\phi_i}{t}{}\pd{t}{x}} u J\\
&= \frac{1}{J(\vect{x}_i)}\sum_{j = 1}^{N_p}\LRp{\LRp{D^{r}}^T_{ij} u^r_j + \LRp{D^{s}}^T_{ij} u^s_j + \LRp{D^{t}}^T_{ij} u^t_j}
\end{align*}
\EndFor
\EndProcedure
\end{algorithmic}
\caption{Algorithm for the (skew-symmetric) pyramid volume kernel.}
\label{alg:pyr_vol}
\end{algorithm}

\subsection{Many-core implementation}
\label{sec:implementation}

%\remark{Add computational section with algorithms + further optimizations?}

The low-storage DG method described above has been implemented in \verb+hybridg+, a portable, GPU accelerated solver for the acoustic wave equation on hybrid meshes.  Portability between OpenMP, OpenCL and CUDA platforms is achieved through the OCCA programming model \cite{medina2014occa}.  The solution is evolved in time using a multi-rate Adams-Bashforth (MRAB) scheme.  Support for multiple GPUs is facilitated through MPI, though load-balancing strategies for scalability remain to be investigated.  The kernels for each element type are optimized according to the operations described above.  Additional optimizations may be done by tuning the number of elements processed per workgroup (batching).  In Section~\ref{sec:comp}, we give results where optimal batch sizes are determined experimentally for each element type and approximation order $N$.  

Due to large variations in element sizes, as well as variations in element type-dependent trace constants, local timestep restrictions may vary significantly over the mesh.  To sidestep overly restrictive global timesteps, we use multi-rate timestepping in the form of a 3rd order Adams-Bashforth scheme.  The local timestep $d\tau_K$ for each element is derived in Section~\ref{sec:timestep}, and the global timestep $d\tau_{\rm min}$ is taken as the minimum over all local timesteps.  For the MRAB scheme, $N_{\rm levels}$ timestep levels are determined via
\[
d\tau_{\rm lev} = 2^{N_{\rm levels}-{\rm lev}} d\tau_{\rm min}, \quad 1\leq {\rm lev}\leq N_{\rm levels}.
\]
Elements are then binned into each timestep level --- a given element $K$ is assigned a timestep of $d\tau_{\rm lev}$ if $d\tau_{{\rm lev}-1} \leq d\tau_K \leq d\tau_{\rm lev}$.  A second sweep through the mesh moves elements from coarser timesteps to a finer timesteps in order to guarantee that neighboring elements differ only by one level at most (or that the timesteps of neighboring elements differ only by a factor of 1 or 2).  The solution is then updated according to the coarsest timestep, and each timestep level is evolved $2^{{\rm lev}-1}$ times for every coarse timestep.  

Since the details of the implementation are independent of element type, we refer the reader to \cite{godel2010gpu, gandham2014gpu, modave2015accelerated} for a description of the multi-rate scheme on triangular and tetrahedral meshes.

\section{Timestep restrictions}
\label{sec:timestep}

Stable local and global timesteps are necessary for multi-rate schemes.  We derive here timestep restrictions for hybrid meshes by bounding the DG operator norm.  These bounds are given in terms of order-dependent constants and physical/geometric quantities for each element type.  

To simplify notation, we define group trial and test variables $U,V$ 
\[
U = \LRp{\begin{array}{c}p\\ \bm{u}\end{array}}, \quad V = \LRp{\begin{array}{c}v\\ \bm{\tau}\end{array}}.
\]
Under time-explicit DG methods, the discretized acoustic wave equation results in a system of ODEs
\[
\td{U}{\tau} = M^{-1}A(U), \qquad 
%\]
%where $M$ is the scaled mass matrix, such that
%\[
V^TMU = \sum_{K\in \Omega_h}\LRp{\int_K \frac{1}{\kappa} p v + \int_K \rho \bm{u} \bm{\tau}.}
\]
where $M$ is the scaled mass matrix.  
%We may derive local bounds by noting that 
The global matrix $A$ is given as the sum of local matrices $A = \sum_{K\in \Omega_h} A_K$, where $A_K$ is given by the local spatial discretization over an element $K \in \Omega_h$
\begin{align}
V^TA_K U &= \int_K \bm{u}\cdot \Grad\phi^- \diff x  - \int_K \Grad p \cdot  \bm{\psi}^- \diff x \label{eqn:AK}\\
&+ \int_{\partial K}\LRp{ \frac{1}{2}\tau_p\jump{p} - \bm{n}^-\cdot \avg{\bm{u}} }\phi^- \diff x  + \int_{\partial K} \bm{n}^-\frac{1}{2}\LRp{\tau_u \jump{\bm{u}}\cdot \bm{n}^- - \jump{p}}\bm{\psi}^- \diff x.\nonumber
\end{align}
We are interested in deriving bounds on the real and imaginary parts of the spectra of the DG operator in order to accurately determine a stable timestep restriction.  For example, we may bound the timestep by $d\tau < 1/\rho(M^{-1}A)$, the inverse of the spectral radius of $M^{-1}A$.  

Prevous work has been done in estimating stable timesteps for structured grids \cite{kubatko2008time, toulorge2011cfl}, and explicit information about the spectra of the discrete DG operator has been derived by Krivodonova and Qin in 1D \cite{krivodonova2013analysis,krivodonova2013analysis2}.  We use an approach similar to Cohen, Ferrieres, and Pernet in \cite{cohen2006spatial} and derive bounds for $\rho(M^{-1}A)$ that depend on explict expressions for the constants in discrete trace and Markov inequalities.  The first bounds the $L^2$ norm of a function on the boundary $\partial K$ of an element $K$ by the $L^2$ norm of the function over the element
\[
\nor{u}^2_{L^2(\partial K)} \leq C_T(N,K) \nor{u}^2_{L^2(K)}
\]
while the latter bounds the $L^2$ norm of the gradient of a function by the $L^2$ norm of the function in the interior
\[
\nor{\Grad u}^2_{L^2(K)} \leq C_M(N,K) \nor{u}^2_{L^2(K)}.
\]
For SEM, the $L^2$ norm is replaced by the equivalent discrete norm computed using GLL quadrature.  The mesh and order-dependent constants $C_T(N,K), C_M(N,K)$ determine the timestep restriction required for stability.  Since we wish to take the timestep as large as possible, we require sharp expressions for these constants in order to accurately estimate $\rho(M^{-1}A)$.  

We note that the bounds we derive assume polynomial basis functions.  This assumption is violated for non-affine mappings of the wedge, where we use rational LSC-DG basis functions.  However, for such bases, we may use weighted Markov and trace inequalities from \cite{warburton2010low} in place of standard polynomial inequalities.  

\subsection{Bounds on the spectra}

We derive here bounds on the spectra of the DG operator based on the constants in trace and Markov inequalities over reference elements.  These bounds apply for both the GL formulation with full mass matrix $M$ (or true $L^2$ inner product) and the SEM lumped mass matrix $M_{\rm SEM}$ (or discrete inner product based on GLL quadrature); the difference between the two is reflected in the constants in each inequality.  

% so long as these constants in the trace inequality are computed with respect to the proper cubature rules.  The constants in the Markov inequalities for non-hexahedral elements are identical regardless of whether the SEM or full mass matrix is used, and for hexahedral elements, Markov inequalities for SEM may be derived using e.  

%may be modified by considering the equivalence relation for degree $N$ polynomials 
%\[
%u^TMu \leq u^TM_{\rm SEM}u \leq \LRp{2 + \frac{1}{N}}^d u^TMu,
%\]
%where $M$ is the full mass matrix \cite{quarteroni2008numerical}.  

The eigenvalues $\lambda$ and corresponding eigenvectors $v$ of $M^{-1}A$ are given by the generalized eigenvalue problem $Av = \lambda Mv$.  
We may derive a more detailed bound on the spectra of $M^{-1}A$ by decomposing the matrix into symmetric and skew-symmetric parts
\[
A = \LRp{A^s + A^k}, \quad A^s = \frac{1}{2}(A + A^T), \quad A^k = \frac{1}{2}(A - A^T).
\]
From \cite{amir1962elements}, we have that
\begin{align*}
\lambda_{\min}(M^{-1}A^s) &\leq {\rm Re}\LRp{\lambda(M^{-1}A)} \leq \lambda_{\max}(M^{-1}A^s)\\
\lambda_{\min}(M^{-1}A^k) &\leq {\rm Im}\LRp{\lambda(M^{-1}A)} \leq \lambda_{\max}(M^{-1}A^k),
\end{align*}
In other words, the magnitude of the real and imaginary parts of the spectra are bounded by $\rho(M^{-1}A^s)$ and $\rho(M^{-1}A^k)$.  
%Additionally, the spectral radii of $M^{-1}A^s$ and $M^{-1}A^k$ are equal to their operator norms.  Since $A$ is the sum of local matrices $A_K$, we may bound 
%\begin{align*}
%\rho(M^{-1}A^s) &= \nor{M^{-1}A^s} \leq \max_K \nor{M_K^{-1}A^s_K} = \max_K \rho\LRp{M_K^{-1}A^s_K}\\
%\rho(M^{-1}A^k) &= \nor{M^{-1}A^k} \leq \max_K \nor{M_K^{-1}A^k_K} = \max_K \rho\LRp{M_K^{-1}A^k_K},
%\end{align*}
%where $M_K$ is the scaled local mass matrix for the element $K$, and $A^s_K$ and $A^k_K$ are the local symmetric and skew parts of $A$.  What remains is to bound the spectral radii of $M_K^{-1}A^s_K$ and $M_K^{-1}A^k_K$.  
Since $A^s$ is symmetric and real, its spectral radius is given by the generalized Rayleigh quotient 
\[
\rho(M^{-1}A^s) = \max \LRb{\frac{U^T A^s U}{U^TMU}}.
\]
From the definition of $A_K$ in (\ref{eqn:AK}), we can show that
\begin{align}
\LRb{U^T A^s U} &= \LRb{-\frac{1}{2}\sum_{f\in \Gh} \int_f \LRc{\tau_p\jump{p}^2 + \tau_u\jump{\bm{u}_n}^2}} \leq \frac{1}{2}\sum_{f\in \Gh} \LRc{\tau_p\nor{\jump{p}}^2_{L^2(f)} + \tau_u\nor{\jump{\bm{u}_n}}_{L^2(f)}^2},
\label{eq:sym}
\end{align}
where we have used the normal jump $\jump{\bm{u}_n} = \bm{u}^+ \bm{n}^+ + \bm{u}^-\bm{n}^-$.  We may bound the norm of the jumps 
\[
\frac{1}{2}\nor{\jump{p}}^2_{L^2(f)} \leq \nor{p^+}_{L^2(f)}^2 + \nor{p^-}_{L^2(f)}^2, \quad \frac{1}{2}\nor{\jump{\bm{u}_n}}^2_{L^2(f)} \leq \nor{\bm{u}^+\bm{n}^+}_{L^2(f)}^2 + \nor{\bm{u}^-\bm{n}^-}_{L^2(f)}^2,
\]
and bound the sum of jumps and averages over faces by boundary values 
\[
\sum_{f\in \Gh} \frac{1}{2}\nor{\jump{p}}_{L^2(f)}^2 \leq \sum_{K\in \Oh} \nor{p}^2_{L^2(\partial K)}, \quad \sum_{f\in \Gh} \nor{\avg{p}}_{L^2(f)}^2 \leq \sum_{K\in \Oh} \nor{p}^2_{L^2(\partial K)}.  
\]
Combined with (\ref{eq:sym}) and changing to a sum over elements, we can apply trace inequalities to arrive at
\begin{align*}
\LRb{U^TA^sU} &\leq \sum_{K\in \Oh} \LRc{ \tau_{p,K}\nor{p}^2_{L^2(\partial K)} + \tau_{u,K}\nor{\bm{u}\cdot \bm{n}}^2_{L^2(\partial K)}} \\
&\leq \sum_{K\in \Oh} C_T(N,K)\LRc{\tau_{p,K}\nor{p}_{L^2(K)}^2+ \tau_{u,K}\nor{\bm{u}}_{L^2(K)}^2},
\end{align*}
where $C_T(N,K)$ is the constant in the trace inequality over the boundary $\partial K$ of the element $K$ and $\tau_{p,K}, \tau_{u,K}$ are the maximum values of the penalties over the faces of an element. 
%\[
%\tau_{p,K} = \sum_{f\in \partial K} \tau_p,\qquad \tau_{u,K} = \sum_{f\in \partial K} \tau_u.
%\]
Combining these bounds, we have
\begin{align*}
\max \LRb{\frac{U^T A^s U}{U^TMU}} &\leq \max \frac{\sum_{K\in \Oh} C_T(N,K)\LRc{\tau_{p,K}\nor{p}_{L^2(K)}^2 + \tau_{u,K}\nor{\bm{u}}_{L^2(K)}^2}}{\sum_{K\in \Omega_h}\LRc{\frac{1}{\kappa_K}\nor{p}_{L^2(K)}^2 + \rho_K\nor{\bm{u}}_{L^2(K)}^2}} \\
& \leq \max_K C_T(N,K) \max\LRp{\tau_{p,K}\kappa_K,\frac{\tau_{u,K}}{\rho_K},}
\end{align*}

%The spectral radius of $M^{-1}A^k$ may be bounded by noting that $A^k$ is block skew symmetric
%\[
%A^k = \left[\begin{array}{cc} & K\\ -K^T & \end{array}\right],
%\]
%where $K$ is a matrix such that, if $V,U$ represent degrees of freedom for test functions $v$ and velocity $\vect{u}$
%\[
%V^TKU = \sum_{K\in\Oh}\int_K \vect{u}\cdot \Grad v - \sum_{f\in \Gh}\int_{f} \avg{\vect{u}}\jump{v}.
%\]
%The block skew symmetric structure of $A^k$ implies that the spectral radius of $A^k$ is equal to the largest generalized singular value of $K$, which has the variational characterization
%\[
%\sigma_{\rm max}(K) = \max_{U,V}\frac{V^TKU}{\nor{U}_M\nor{V}_M}.
%\]
The spectral radius of $M^{-1}A^k$ may be bounded by noting that skew-symmetric matrices are normal, which implies that $\rho(M^{-1}A^k) = \nor{M^{-1}A^k}$.  A characterization of this norm is given as follows: for a skew-symmetric matrix $Z \in \mathcal{R}^{n\times n}$,
\[
\nor{Z} = \max_{f,g\in \mathcal{R}^n} \frac{2g^T Z f}{\nor{f}^2 + \nor{g}^2}.
\]
This characterization has been proven in other settings (see, for example, \cite{chung2014discrepancy}).  A straightforward modification of this theorem to use the norm $\nor{f}_M = \sqrt{f^TMf}$ gives that
\begin{align*}
\nor{M^{-1}A^k} &= \max_{U,V\in \mathcal{R}^n} \frac{2V^T A^k U}{\nor{U}_M^2 + \nor{V}_M^2}
\end{align*}
where the denominator $\nor{U}^2_M + \nor{V}^2_M$ is 
\[
\nor{U}^2_M + \nor{V}^2_M = \sum_{K\in \Oh}\LRc{\frac{1}{\kappa_K}\nor{p}_{L^2(K)}^2 + \rho_K\nor{\bm{u}}_{L^2(K)}^2 + \frac{1}{\kappa_K}\nor{v}_{L^2(K)}^2 + \rho_K\nor{\bm{\tau}}_{L^2(K)}^2}
\]
From (\ref{eqn:AK}), we may compute the skew-symmetric form
\begin{align*}
V^T A^k U &= \sum_{K\in \Oh}\LRc{\int_K \bm{u}\cdot \Grad v - \int_K \Grad p \cdot \bm{\tau}} + \sum_{f\in \Gh} \int_f \LRp{-\avg{\bm{u}}\jump{v_n}  + \avg{\bm{\tau}}\jump{p_n}}
\end{align*}
where we have used the vector-valued jumps 
$\jump{p_n} = p^+\bm{n}^+ + p^-\bm{n}^-$ and  $\jump{v_n} = v^+\bm{n}^+ + v^-\bm{n}^-$.
%of $p, v$
%\[
%\jump{p_n} = p^+\bm{n}^+ + p^-\bm{n}^-, \quad \jump{v_n} = v^+\bm{n}^+ + v^-\bm{n}^-.
%\]

We may bound the volume terms using Markov and Young's inequality
\begin{align*}
\sum_{K\in \Oh}\LRb{\int_K \bm{u}\cdot \Grad v - \int_K \Grad p \cdot \bm{\tau}} &\leq \sum_{K\in \Oh}\nor{\bm{u}}_{L^2(K)}\nor{\Grad v}_{L^2(K)} + \nor{\Grad p}_{L^2(K)}\nor{\bm{\tau}}_{L^2(K)}\\
%& \leq \sum_{K\in \Oh}\sqrt{C_M(N,K)}\LRp{\nor{\bm{u}}_{L^2(K)}\nor{v}_{L^2(K)} + \nor{p}_{L^2(K)}\nor{\bm{\tau}}_{L^2(K)}} \\
&\leq \sum_{K\in \Oh}\frac{1}{2}\sqrt{C_M(N,K)} \LRp{\nor{p}_{L^2(K)}^2 + \nor{\bm{u}}_{L^2(K)}^2 + \nor{v}_{L^2(K)}^2 + \nor{\bm{\tau}}_{L^2(K)}^2}\\
&\leq \frac{1}{2} \sqrt{C_M(N,K)}\max\LRp{\kappa_K,\frac{1}{\rho_K}}\LRp{\nor{U}_M^2 + \nor{V}_M^2}
\end{align*}
where $C_M(N,K)$ is the constant in the Markov inequality for the element $K$.  

The surface terms may also be bounded using Young's inequality
\begin{align*}
&\LRb{\sum_{f\in \Gh} \int_f \LRp{-\avg{\bm{u}}\jump{v_n}  + \avg{\bm{\tau}}\jump{p_n}}} \leq \sum_{f\in \Gh} \LRp{ \nor{\avg{\bm{u}}}_{L^2(f)}\nor{\jump{v_n}}_{L^2(f)}  + \nor{\avg{\bm{\tau}}}_{L^2(f)}\nor{\jump{p_n}}_{L^2(f)}}\\
%&\leq \frac{1}{2}\sum_{f\in \Gh}  \LRp{
%\nor{\jump{p_n}}_{L^2(f)}
%\nor{\bm{u}^+ + \bm{u}^-}_{L^2(f)} + 
%\nor{\jump{v_n}}_{L^2(f)}
%\nor{\bm{\tau}^+ + \bm{\tau}^-}_{L^2(f)}
%}\\
%&\leq \frac{1}{4}\sum_{f\in \Gh} \LRp{
%\nor{\jump{p_n}}_{L^2(f)}^2 +
%\nor{\bm{u}^+ + \bm{u}^-}_{L^2(f)}^2 + 
%\nor{\jump{v_n}}_{L^2(f)}^2 + 
%\nor{\bm{\tau}^+ + \bm{\tau}^-}_{L^2(f)}^2.
%}\\
&\leq \frac{1}{2}\sum_{K\in \Oh}{
\LRp{\nor{p}^2_{L^2(\partial K)} + \nor{\bm{u}}^2_{L^2(\partial K)}
+\nor{v}^2_{L^2(\partial K)} + \nor{\bm{\tau}}^2_{L^2(\partial K)}}}\\
&\leq \frac{1}{2}\sum_{K\in \Oh}C_T(N,K){
\LRp{\nor{p}^2_{L^2(K)} + \nor{\bm{u}}^2_{L^2( K)}
+\nor{v}^2_{L^2( K)} + \nor{\bm{\tau}}^2_{L^2( K)}}}\\
&\leq \frac{1}{2}\max_K  C_T(N,K)\max\LRp{\kappa_K,\frac{1}{\rho_K}}\LRp{\nor{U}_M^2 + \nor{V}_M^2}
\end{align*}
Combining these two bounds gives 
\begin{align*}
\nor{M^{-1}A^k} &=\max_{U,V\in \mathcal{R}^n} \frac{2V^T A^k U}{\nor{U}_M^2 + \nor{V}_M^2} \leq \max_K \max\LRp{\kappa_K,\frac{1}{\rho_K}}\LRp{\sqrt{C_M(N,K)} + {C_T(N,K)}}.
\end{align*}

\subsection{Constants in trace inequalities}

An explicit expression for the trace inequality constant for a general $d$-simplex was given by Warburton and Hesthaven in \cite{warburton2003constants}, and was extended by Hillewaert in his thesis to hexahedra and wedges in \cite{hillewaert2013development}.  While the trace constant for the pyramid does not appear to be available in an explicit form, Hillewaert also proposed an empirical fit of the trace inequality constant for pyramids, which was based on a lower bound for the trace constant that could be derived in closed form.  An improved asymptotically optimal constant was given by Chan and Warburton in \cite{chan2015hp}.  The trace inequality is as follows: for $u$ in the approximation space, the norm of $u$ on a face $\widehat{f}$ of the reference element $\widehat{K}$, 
\[
\nor{u}^2_{L^2\LRp{\widehat{f}}} \leq C_f(N) 
%|\widehat{f}| 
%\frac{\LRb{f}}{\LRb{K}}
\nor{u}^2_{L^2\LRp{\widehat{K}}}.
\]
%where $|\widehat{f}|$ is the area of $\widehat{f}$.  
These constants were derived assuming triangular faces with area 2 and quadrilateral faces of area 4 on bi-unit reference elements.  
%where ${\LRb{f}}/{\LRb{K}}$ is the ratio of face surface area to element volume.  

%For affine mappings of elements with $N_{f}$ faces, we may bound the trace inequality over the element surface $\partial K$.  Using H{\"o}lder's inequality with $p = 1$ and $q = \infty$, we have 
%\[
%\nor{u}_{L^2(\partial K)}^2 = \sum_{f }^{N_{f}} \nor{u}_{L^2(f)}^2  \leq \sum_{f}^{N_f} C_f(N) \frac{\LRb{f}}{\LRb{K}} \nor{u}_{L^2(K)}^2  \leq C_T(N) \frac{\LRb{\partial K}}{\LRb{K}}\nor{u}_{L^2(K)}^2.
%\]
%where $C_T(N) = \max_f C_f(N)$ is the maximum trace inequality constant over all the faces of the element $K$.  

The constants in the trace inequality over the surface $\partial \widehat{K}$ of a reference element may be given in terms of the constants over the different types of faces.  Assume the trace inequality is derived for face $f$.  Then, if we have $N_f$ faces of a reference element $\widehat{K}$, 
\[
\nor{u}^2_{L^2\LRp{\partial \widehat{K}}} = \sum_{j=1}^{N_f} \nor{u}^2_{L^2\LRp{f_j}} = \sum_{j=1}^{N_f} \frac{\LRb{f_j}}{\LRb{f}}\nor{u}^2_{L^2\LRp{f}} \leq \sum_{j=1}^{N_f}\frac{C_f(N)}{\LRb{f}}\LRb{f_j}\nor{u}_{\widehat{K}}^2 \leq \max \frac{C_f(N)}{\LRb{f}} \LRb{\partial \widehat{K}} \nor{u}_{\widehat{K}}^2.
\]
where we have used H{\"o}lder's inequality with $p = 1$ and $q = \infty$ and the fact that, under an affine transform (i.e. rotation), $\nor{u}_{f_j}^2 = {\LRb{f_j}}/{\LRb{f}}\nor{u}_{f}^2$.  We will refer to $C_T(N)$ as the constant in the trace inequality over the full surface $\partial K$, such that
\[
C_T(N) = \max_f \frac{C_f(N)}{\LRb{f}}\LRb{\partial \widehat{K}}
\]
We summarize these trace constants $C_f(N)$ and $C_T(N)$ in Table~\ref{table:trace}.  We remark that the constant $C_T(N)$ for tetrahedra may be equivalently derived using the trace inequality from \cite{warburton2003constants} 
\[
\nor{u}_{\partial K}^2 \leq \frac{(N+1)(N+3)}{3}\frac{\LRb{\partial K}}{\LRb{ K}}\nor{u}_K
\]
when the element $K$ is taken to be the bi-unit right tetrahedron $\widehat{K}$.

\begin{table}
\centering                                                                                   
\begin{tabular}{| c | c | c | c |}
\hline
Element type & Triangular face & Quadrilateral face & $C_T(N)$\\
\hline
Hexahedra &  & $(N+1)^2/2$ & $4(N+1)^2$\\
\hline
Wedge & $(N+1)^2/2$ & ${(N+1)(N+2)}$ & $(\sqrt{2}+3)(N+1)(N+2)$ \\
\hline 
Pyramid & $(N+1)(N+2)/2$ & $(N+1)(N+3)$ & $(\sqrt{2}+2)(N+1)(N+3)$\\
\hline 
Tetrahedra & ${(N+1)(N+3)}/{2}$ & & $(\sqrt{3}+3)(N+1)(N+3)/2$\\
\hline 
\end{tabular}
\caption{Summary of the trace constant $C_f(N)$ in the discrete trace inequalities for different faces of various reference elements, as well as the analytic trace constant $C_T(N) = \max_f(C_f(N)/\LRb{f}) \LRb{\partial \widehat{K}}$ over the surface $\partial \widehat{K}$. }
\label{table:trace}
\end{table}

The constants in Table~\ref{table:trace} have the benefit of being fully explicit in $N$.  However, since the values of $C_T(N)$ are based on face-by-face estimates (instead of considering the whole surface $\partial K$), surface trace inequalities under these constants are not tight.  Figure~\ref{fig:surface_trace} compares numerically computed trace inequality constants to the derived analytic constants $C_T(N)$ in Table~\ref{table:trace}.  For each element up to $N\leq 7$, the estimated analytic constant is a factor of 2-3 larger than the computed constant.  

\begin{figure}
\centering
\subfloat[Hexahedra]{\includegraphics[width=.24\textwidth]{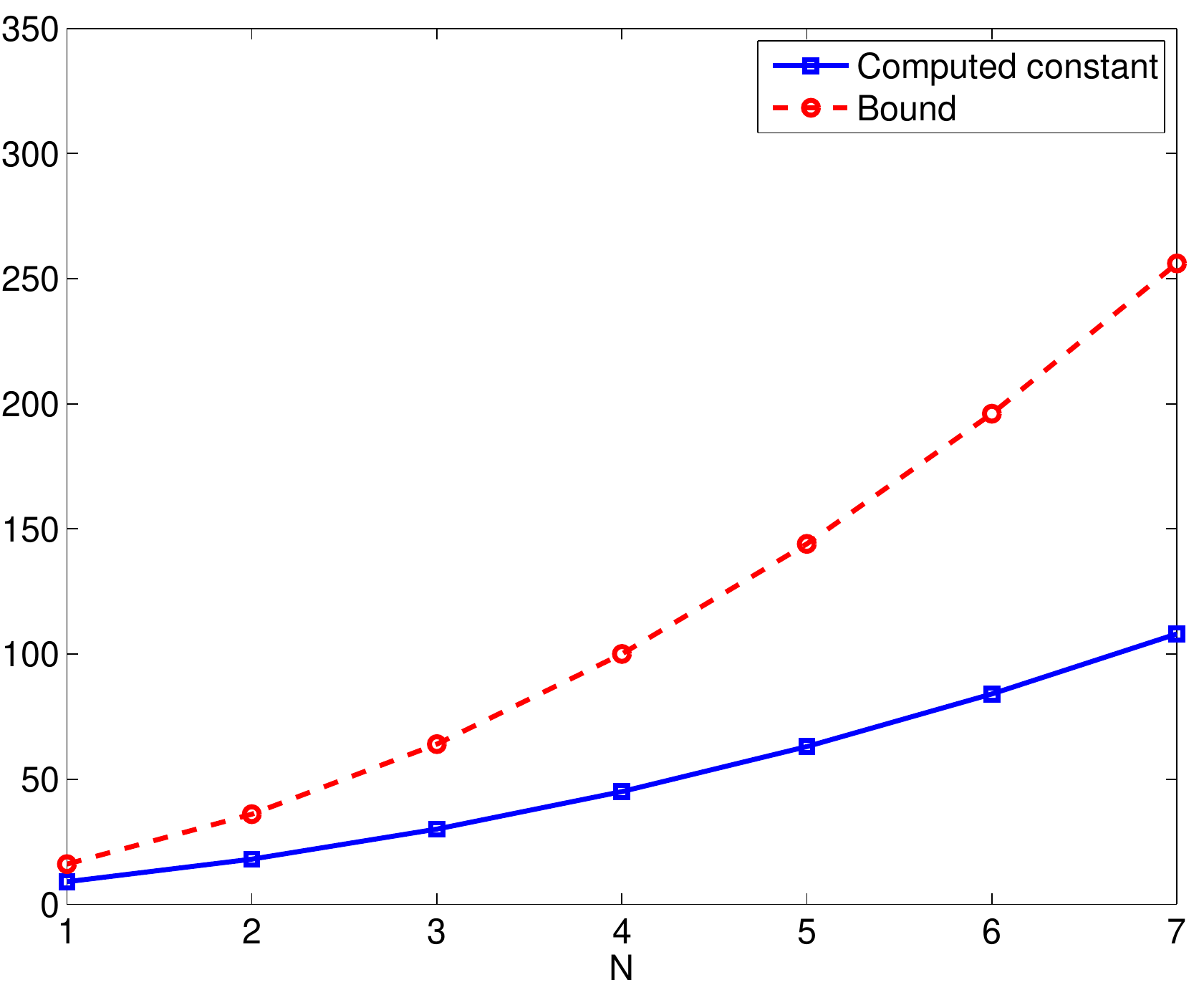}}
\subfloat[Wedge]{\includegraphics[width=.24\textwidth]{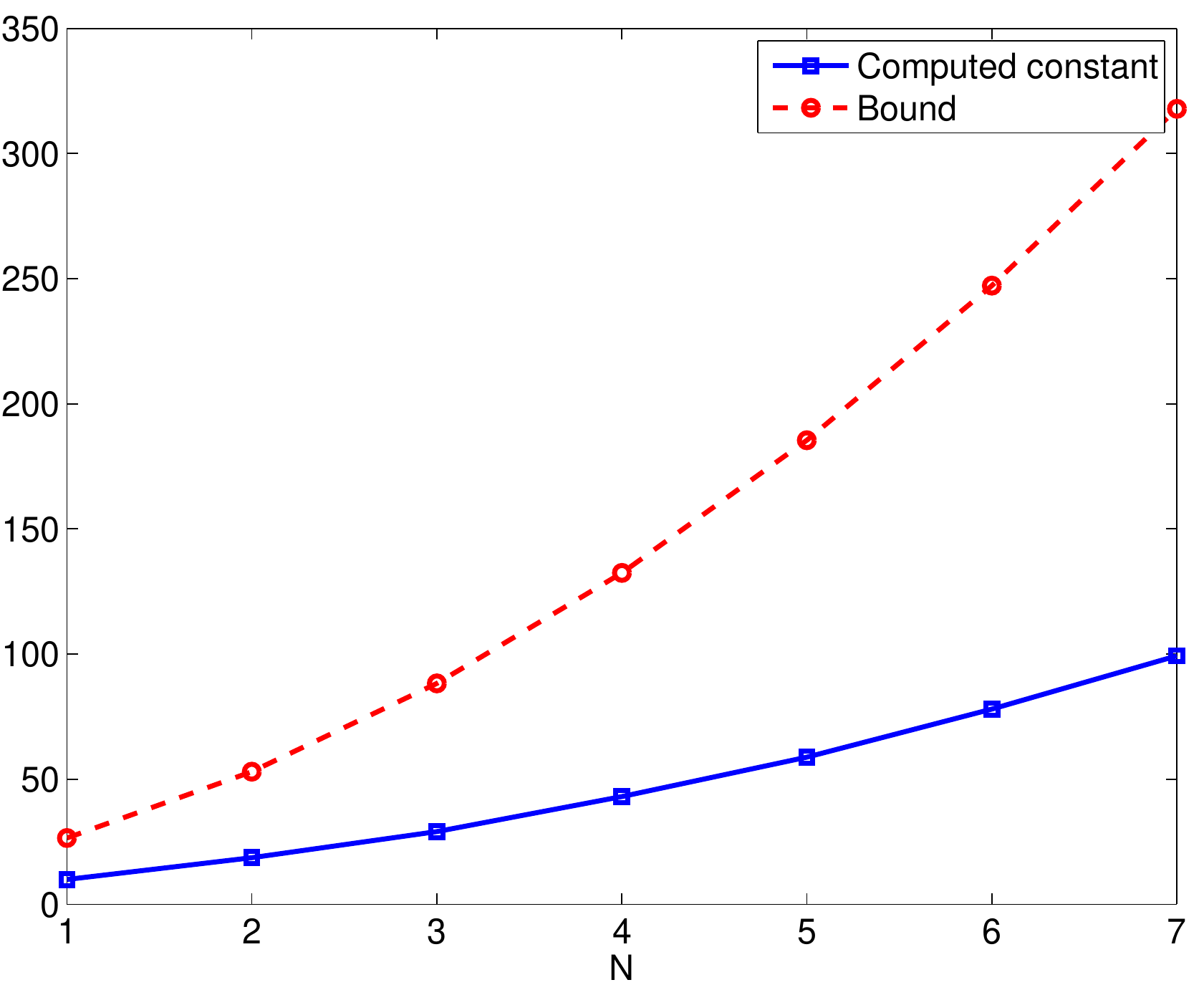}}
\subfloat[Pyramid]{\includegraphics[width=.24\textwidth]{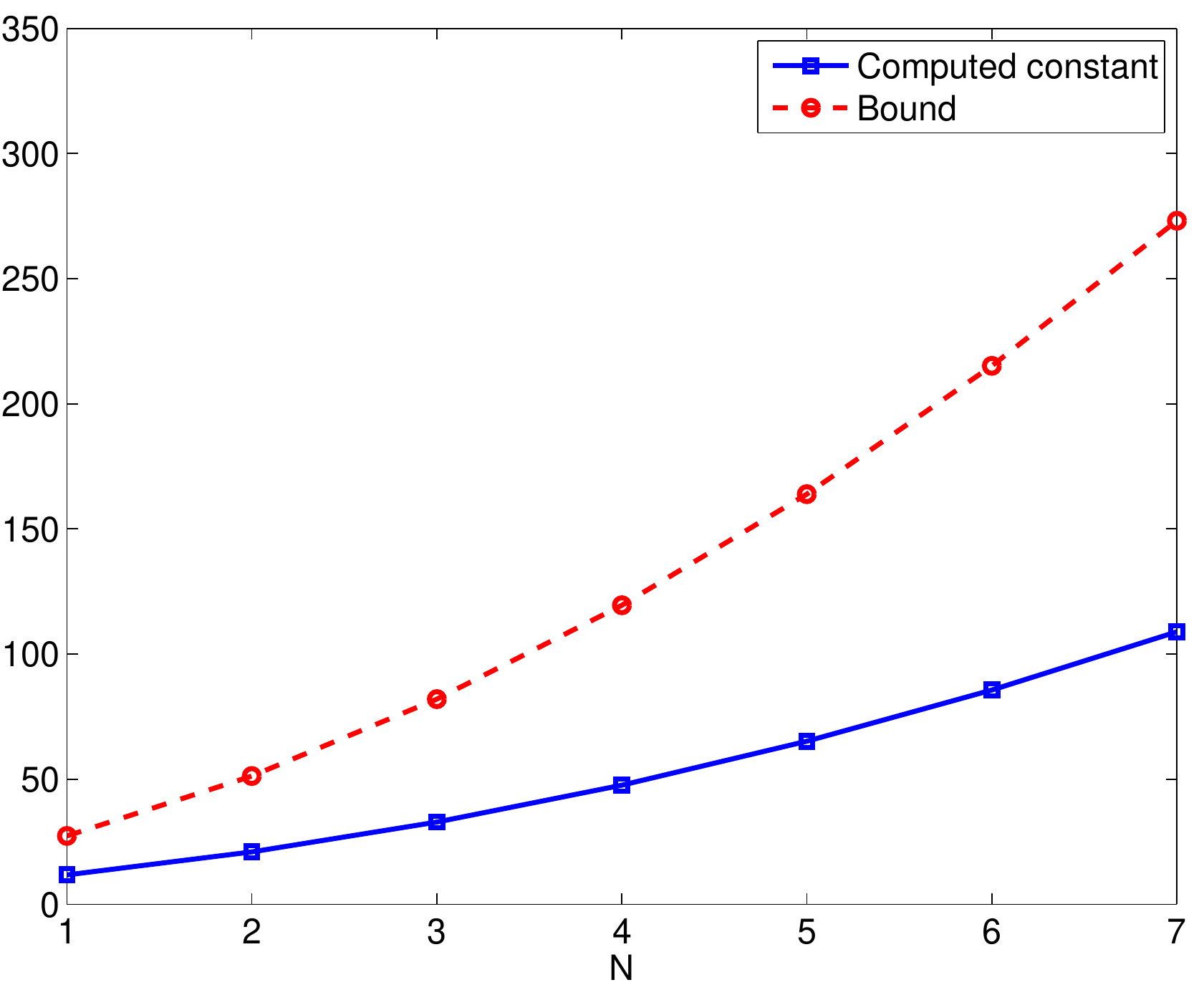}}
\subfloat[Tetrahedra]{\includegraphics[width=.24\textwidth]{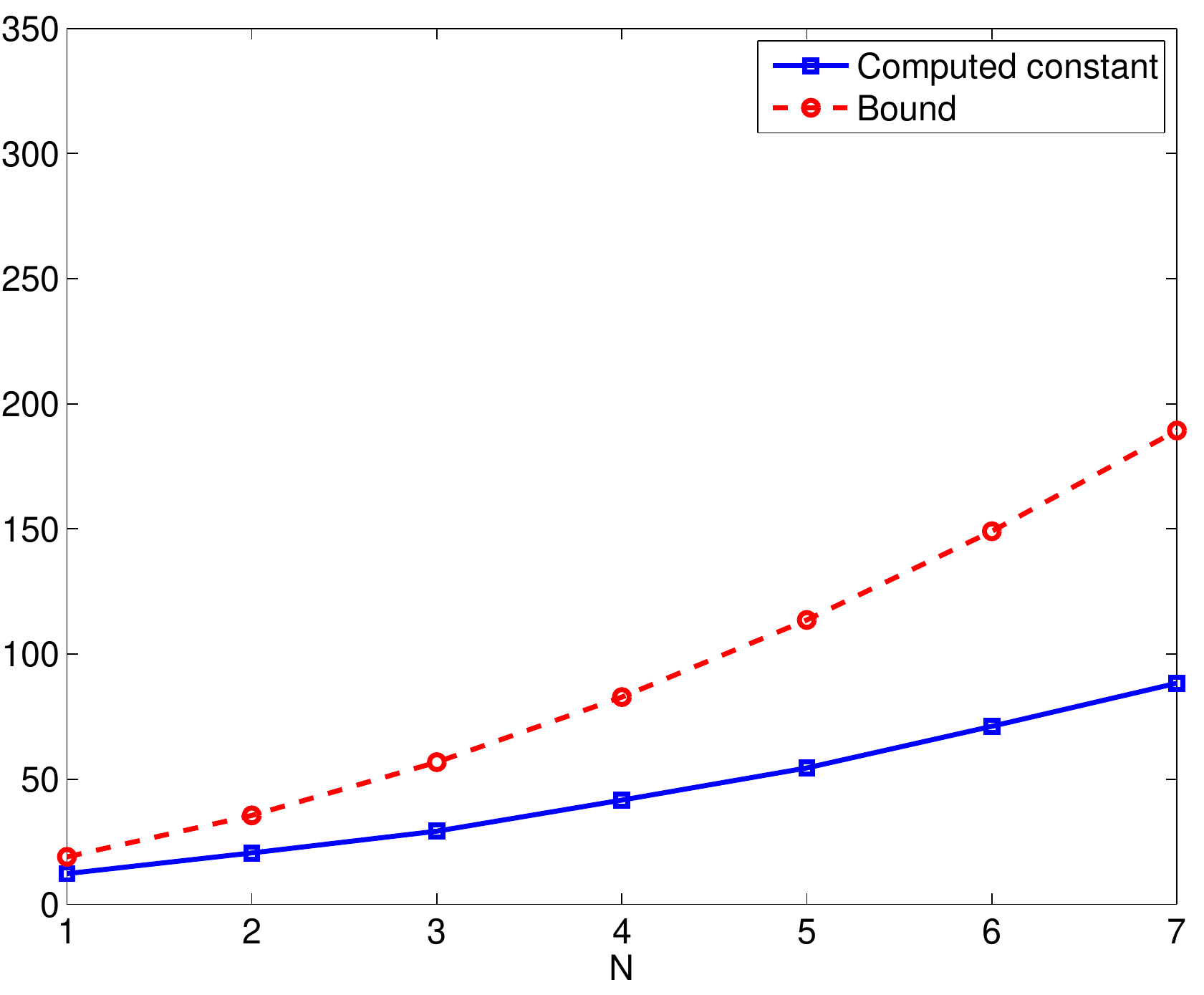}}
\caption{Comparison of face-based bound $C_T(N)$ to numerically computed trace inequality constants on each reference element.  }
\label{fig:surface_trace}
\end{figure}

For the reference wedge, pyramid, and tetrahedra, we have not found explicit expressions for numerically computed constants in surface trace inequalities.  However, for the hexahedron, we are able to derive that the constant in the bounds
\[
\nor{u}^2_{L^2\LRp{\partial \widehat{K}}} \leq C_T(N) \nor{u}^2_{L^2\LRp{\widehat{K}}}, \qquad
\nor{u}^2_{{\rm SEM}\LRp{\partial \widehat{K}}} \leq C_{\rm SEM}(N) \nor{u}^2_{{\rm SEM}\LRp{\widehat{K}}}
\]
are given exactly by 
\[
C_T(N) = \frac{3(N+1)(N+2)}{2}, \qquad C_{\rm SEM}(N) = \frac{3N(N+1)}{2}
\]
where the bound for the SEM inequality is valid for $N\geq 1$.  
These constants are a refinement of explicit bounds given by Ern and Burman \cite{burman2007continuous} and Evans and Hughes \cite{evans2013explicit}, and proofs are included in \ref{app:TPtrace}.  

Under a mapping from reference element $\widehat{K}$ to physical element $K$, we may derive a bound in terms of the determinant of the Jacobian $J$ and the determinant of the surface Jacobian $J^s$ 
\begin{align}
\nor{u}^2_{L^2\LRp{\partial K}} \leq C_T(N)\nor{J^s}_{L^\infty\LRp{\partial \widehat{K}}}\nor{J^{-1}}_{L^\infty\LRp{\widehat{K}}} \nor{u}^2_{L^2\LRp{K}}.
\label{eq:trace_bnd}
\end{align}
If the LSC-DG basis is used for the wedge, the polynomial trace inequalities used here may be replaced by trace inequalities for weighted polynomial spaces \cite{warburton2013low}, resulting in a bound of the form
\[
\nor{u}^2_{L^2\LRp{\partial K}} \leq C_T(N)\nor{\frac{J^s}{J}}_{L^\infty\LRp{\partial \widehat{K}}} \nor{u}^2_{L^2\LRp{K}}.
\]

For affine mappings of faces where ${J^s}, {J}$ are constant, the ratio of their norms reduces to the ratio of the physical surface area to reference surface area divided by the ratio of physical element volume to reference element volume
\[
{\nor{J^s}_{L^\infty\LRp{ \partial \widehat{K}}}}{\nor{J^{-1}}_{L^{\infty}\LRp{\widehat{K}}}} = \frac{\LRb{\partial K}/\LRb{\partial \widehat{K}}}{{\LRb{K}/\LRb{\widehat{K}}}}.
\]

\subsection{Constants in Markov inequalities}
Expressions for the constants in Markov inequalities over simplices are given in \cite{ozisik2010constants}; however, these bounds have not yet been extended to hexahedra, tetrahedra, and pyramids.  We take a more heuristic approach here and compute numerically the constant in Markov inequalities over the reference element.  We may bound
\[
\nor{\Grad{u}}_{L^2(K)}^2 = \sum_{k = 1}^3 \nor{\pd{u}{x_k}}_{L^2({{K}})}^2 
\leq C_{rst}^2 \nor{J}_{L^\infty(\widehat{K})}\nor{\Grad_{rst} u}^2_{L^2\LRp{\widehat{K}}}, 
\]
where 
$C_{rst} = \max_{rst,xyz}\nor{\pd{r,s,t}{x,y,z}}_{L^\infty(\widehat{K})}$ is the max norm of the Jacobian matrix for the element $K$.  The constants in the reference Markov inequality
\[
\nor{\Grad_{rst} u}^2_{L^2\LRp{\widehat{K}}} \leq C_M(N) \nor{u}_{L^2(\widehat{K})}^2
\]
may then be computed numerically over the reference element through the solution of an eigenvalue problem.  The bound is completed by bounding reference quantities by physical quantities using ${J^{-1}}$
\[
\nor{\Grad{u}}_{L^2(K)}^2 \leq C_M(N) C_{rst}^2  \nor{J}_{L^\infty(\widehat{K})} \nor{J^{-1}}_{L^\infty(\widehat{K})}\nor{u}_{L^2(K)}^2.
\]

\subsection{$N$-dependent bounds for $\rho(M^{-1}A)$}

To summarize, the trace constant over a mapped element $K$ is given by 
\[
C_T(N,K) = C_T(N)\nor{J^s}_{L^\infty\LRp{\partial \widehat{K}}}\nor{J^{-1}}_{L^\infty\LRp{\widehat{K}}},
\]
while the Markov constant ${C_M(N,K)}$ is given by 
\[
C_M(N,K) = C_M(N)C_{rst}^2 {\nor{J}_{L^\infty\LRp{\widehat{K}}}\nor{J^{-1}}_{L^\infty\LRp{\widehat{K}}}}.
\]
The real part of the spectra of $M^{-1}A$ is bounded by the trace inequality constant and physical parameters
\begin{align}
\LRb{{\rm Re}(\lambda)} \leq \max_K \max\LRp{\tau_{p,K}\kappa_K,\frac{\tau_{u,K}}{\rho_K}} C_T(N)\nor{J^s}_{L^\infty\LRp{\partial \widehat{K}}}\nor{J^{-1}}_{L^\infty\LRp{\widehat{K}}}.
\label{eq:bnd_re}
\end{align}
Likewise, the imaginary part of the spectra of $M^{-1}A$ depends on the quantity
\begin{align}
\LRb{{\rm Im}(\lambda)} \leq \max_K \max \LRp{\kappa_K,\frac{1}{\rho_K}} \LRp{\sqrt{C_M(N,K)} + {C_T(N,K)}}.
\label{eq:bnd_im}
\end{align}
The latter term expands to
\[
{C_{rst}\sqrt{C_M(N) \nor{J}_{L^\infty\LRp{\widehat{K}}}\nor{J^{-1}}_{L^\infty\LRp{\widehat{K}}}} + {C_T(N)} \nor{J^s}_{L^\infty\LRp{\partial \widehat{K}}}\nor{J^{-1}}_{L^\infty\LRp{\widehat{K}}}}.
\]

We analyze the bounds derived above for the acoustic wave equation with $\kappa = \rho = 1$ and compare them to numerically computed spectra of $M^{-1}A$.  Figure~\ref{fig:spectra} shows the computed spectra of $M^{-1}A$ for a hybrid mesh using the GL formulation.  Bounds on the real and imaginary parts of the spectra derived by computing $\rho(M^{-1}A^s),\rho(M^{-1}A^k)$ are also included as dotted lines.  

\begin{figure}[!h]
\centering
%\subfloat[Single Hexahedra]{\includegraphics[width=.4\textwidth]{figs/symSkewEigHex.pdf}}
\subfloat[Hybrid mesh]{\includegraphics[width=.4\textwidth]{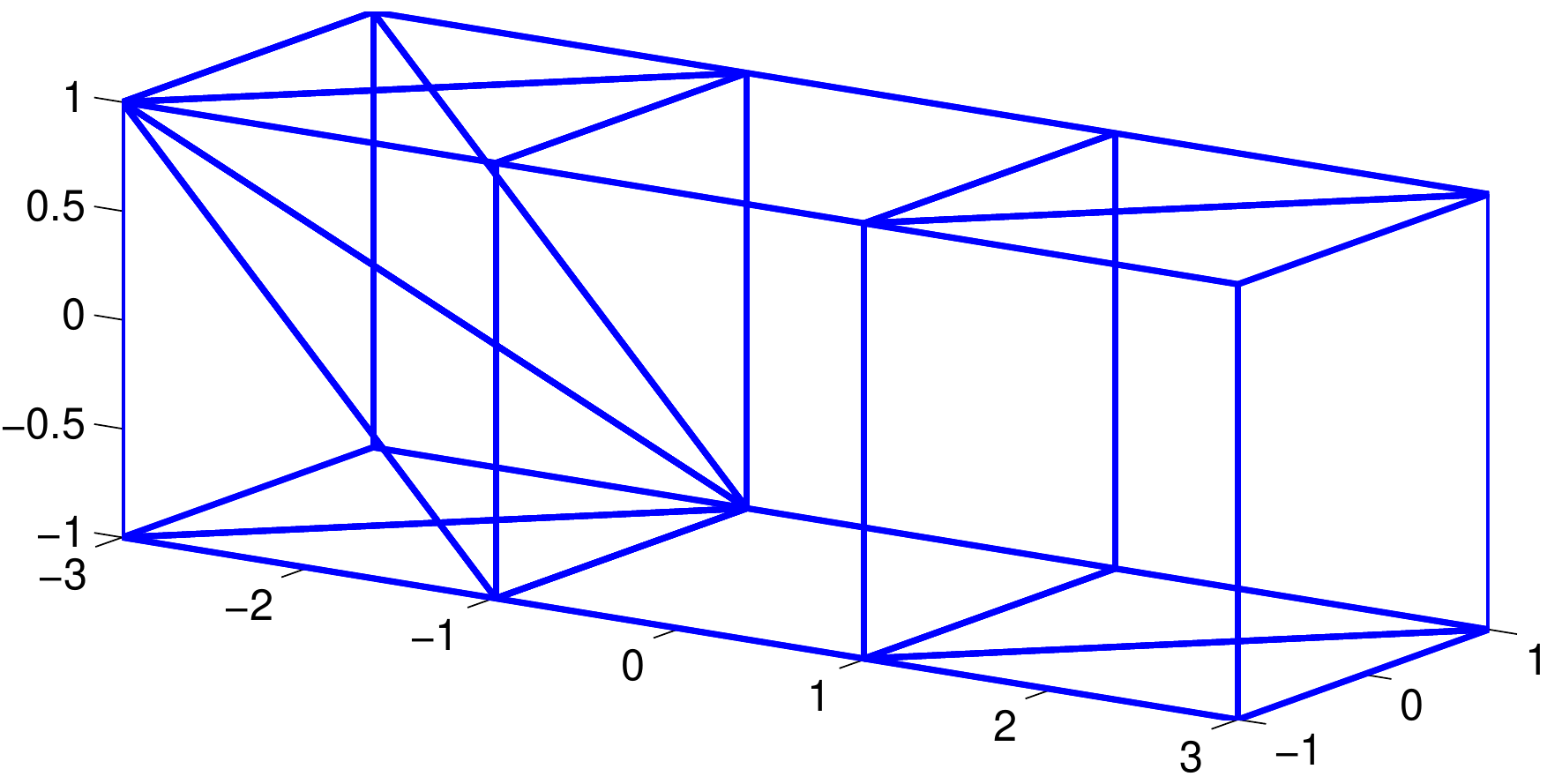}}
\hspace{3em}
\subfloat[Spectra and bounds]{\includegraphics[width=.3\textwidth]{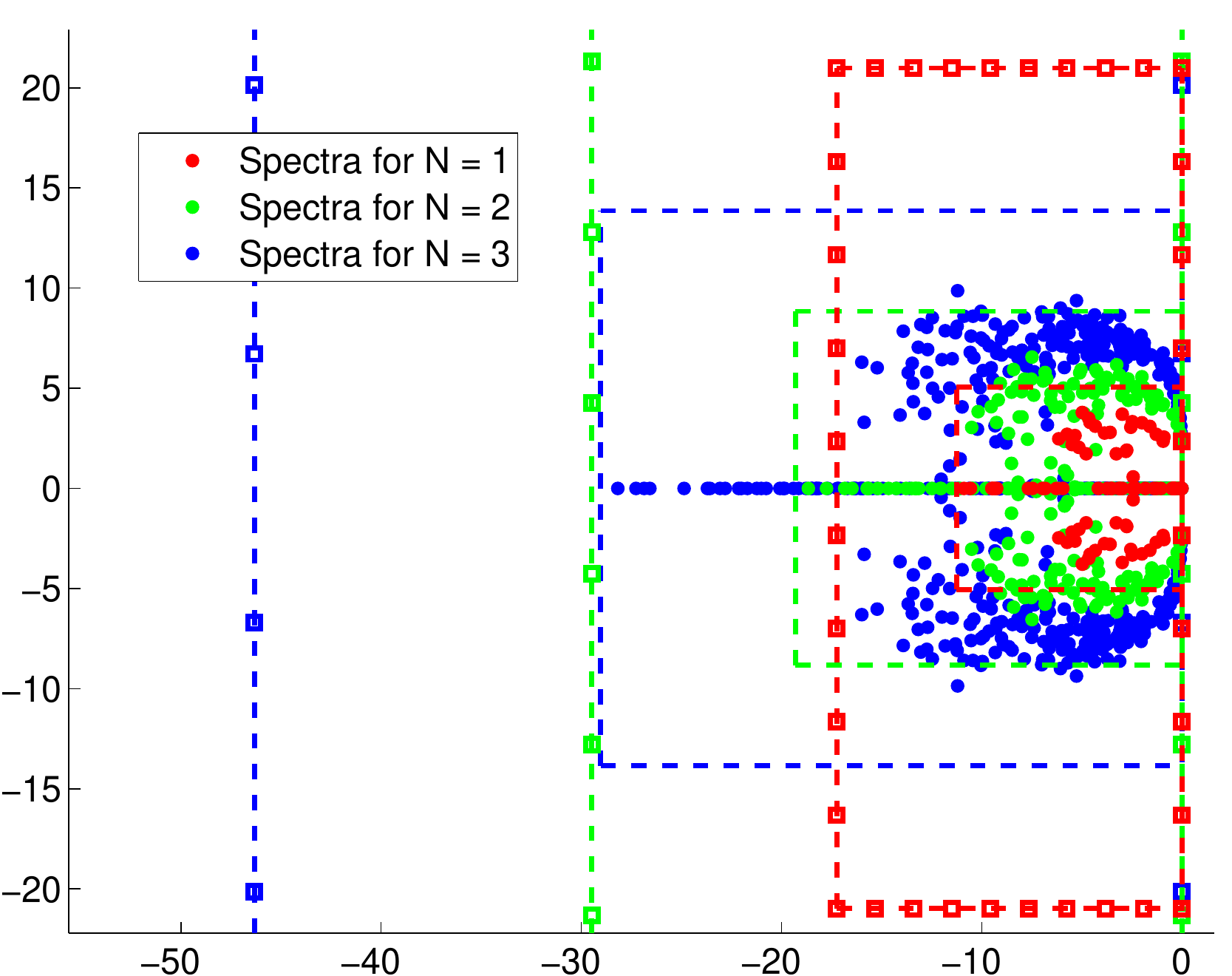}}
%\caption{Spectra for a single hexahedral element and hybrid mesh, along with bounds (dotted lines) derived from the spectra of $A^s$ and $A^k$.  For the single hexahedra, the bounds based on the constants in local trace inequalities matches exactly the bound given by $\rho(A^s)$.}
\caption{Spectra for a hybrid mesh using Gauss-Legendre quadrature, along with bounds (dotted lines) computed from $\rho(A^s)$ and $\rho(A^k)$, and analytic bounds (squared lines).  The bounds for $N=2, 3$ are truncated for visualization purposes.  }
\label{fig:spectra}
\end{figure} 

We note that our estimates for $\rho(M^{-1}A^k)$ provide a loose bound on the imaginary part of the spectra.  However, we observe that the spectral radius of the symmetric part $\rho(M^{-1}A^s)$ provides a relatively tight bound on $\rho(M^{-1}A)$.  Motivated by this observation, we compare $\rho(M^{-1}A)$ to four different bounds based on $\rho(M^{-1}A^s)$:
\begin{enumerate}
\item The true spectral radius of $M^{-1}A^s$.  
\item The bound (\ref{eq:bnd_re}), where the trace inequality constant $C_T(N,K)$ is computed numerically as the maximum trace constant over each mapped element $K$.  
\item The bound (\ref{eq:bnd_re}), where the trace inequality constant $C_T(N)$ is computed numerically over the reference element $\widehat{K}$ (\ref{app:MTconsts}).  
\item The bound (\ref{eq:bnd_re}), where the trace inequality constant $C_T(N)$ is given by the estimate (\ref{eq:trace_bnd}).  
\end{enumerate}
In the latter two bounds, the $L^\infty$ norm of $J, J^{-1}$, and $J^s$ is approximated by taking the maximum value over quadrature points.  %We note that these bounds are tight for a single reference element.  
Table~\ref{table:bounds} shows each bound compared to $\rho(M^{-1}A)$ for various $N$.  The bound on $\rho(M^{-1}A)$ using computed trace inequality constants over the reference element is less than a factor of 2 away from the true spectral radius, and we use this to estimate the largest stable timestep for a given mesh and discretization.  The bound on $\rho(M^{-1}A)$ using analytically derived trace constants over faces provides a looser bound (factor of 4-5 away from the true spectral radius), but has the advantage of being fully explicit in $N$.  

Bounds for the spectral radius under the SEM formulation may be derived by substituting in trace inequality constants using SEM quadrature over quadrilateral faces.  Computed trace inequality constants $C_T(N)$ for both the GL and SEM formulations are given in \ref{app:MTconsts}.  Analytic expressions for constants in trace inequalities may also be derived using norm equivalences between SEM and $L^2$ norms for polynomials of order $N$.  

Numerical experiments confirm that ${\rho\LRp{M^{-1}A}}$ is bounded by the maximum trace inequality constant $\max_K C_T(N,K)$.  Since a stable timestep is given by $d\tau \leq 1/{\rho\LRp{M^{-1}A}}$, we estimate the global timestep by 
\begin{align*}
d\tau = \frac{C}{\max_K \LRp{C^K_{\rho,\kappa}C_T(N,K)}} = \frac{C}{\max_K \LRp{C^K_{\rho,\kappa}C_T(N)C_J^K}},
\end{align*}
where $C^K_{\rho,\kappa} = \max\LRp{\tau_{p,K}\kappa_K,\frac{\tau_{u,K}}{\rho_K}}$, $C$ is a tunable global CFL constant, and $C_J^K$ is 
\begin{align*}
C_J^K &= \nor{J^s}_{L^\infty\LRp{\partial \widehat{K}}}\nor{J^{-1}}_{L^\infty\LRp{\widehat{K}}} \quad \text{if the element is a hex, pyramid, or tet, or}\\
C_J^K &= \nor{\frac{J^s}{J}}_{L^\infty\LRp{\partial \widehat{K}}} \quad \text{if the element is an LSC-DG wedge}.
\end{align*}

Since we have estimated the global timestep by taking the maximum trace constant over all elements, we may also use local trace constants to estimate  stable local timesteps $d\tau_K$ using the constants in local trace inequalities
\begin{align}
d\tau_K = \frac{C}{C^K_{\rho,\kappa}C_T(N)C_J^K}.%\nor{J^s}_{L^\infty\LRp{\partial \widehat{K}}}\nor{J^{-1}}_{L^\infty\LRp{\widehat{K}}}}.
\label{eq:dtK}
\end{align}
%These local timestep restrictions may then be used in multi-rate adaptive timestepping schemes.  

\begin{table}
\centering
\begin{tabular}{|c|c|c|c|c|}          
\hline
N & 1 & 2 & 3 & 4\\
\hline                                
$\rho(M^{-1}A)$ & 10.89 & 18.66 & 28.18 & 41.09 \\  
\hline                            
${\rho(M^{-1}A^s)}$ & 11.26 & 19.30 & 29.05 & 42.86 \\  
\hline                            
Computed $C_T(N,K)$ & 12.22 & 20.84 & 32.76 & 47.47 \\  
\hline                            
Computed $C_T(N)$ & 17.24 & 29.47 & 46.32 & 67.13 \\  
\hline                            
Analytic $C_T(N)$ & 38.53 & 74.73 & 124.54 & 186.81 \\
\hline          
\end{tabular}                         
\caption{Spectral radius of $\rho(M^{-1}A)$ and various bounds for a hybrid mesh and $N = 1, \ldots, 4$.}
\label{table:bounds}
\end{table}

\section{Numerical experiments}
\label{sec:numerics}

In this section, we present results from \verb+hybridg+ for meshes of individual element types and hybrid meshes at various orders of approximation.  

\subsection{Verification: individual element types}
\label{sec:elemverify}

We begin by comparing the SEM and GL formulations for a series of uniformly refined meshes.  Hex meshes are constructed by subdividing the unit cube into hexahedra of size $h$, and we construct wedge and pyramid meshes by further subdividing each hexahedra into either 2 wedges or 6 pyramids, as shown in Figure~\ref{fig:hex_sub}.  
\begin{figure}
\centering
\subfloat[Wedge]{\includegraphics[width=.2\textwidth]{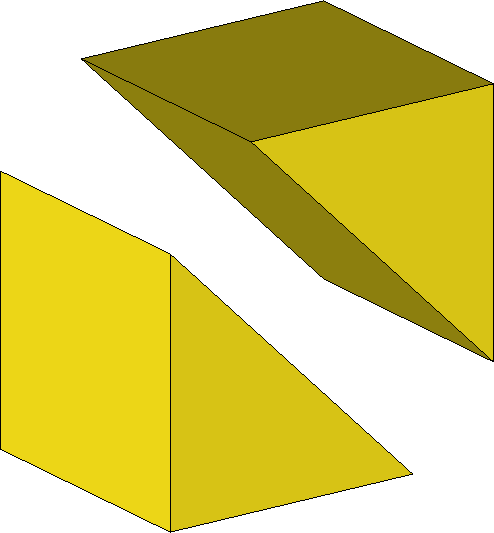}}
\hspace{8em}
\subfloat[Pyramid]{\includegraphics[width=.2\textwidth]{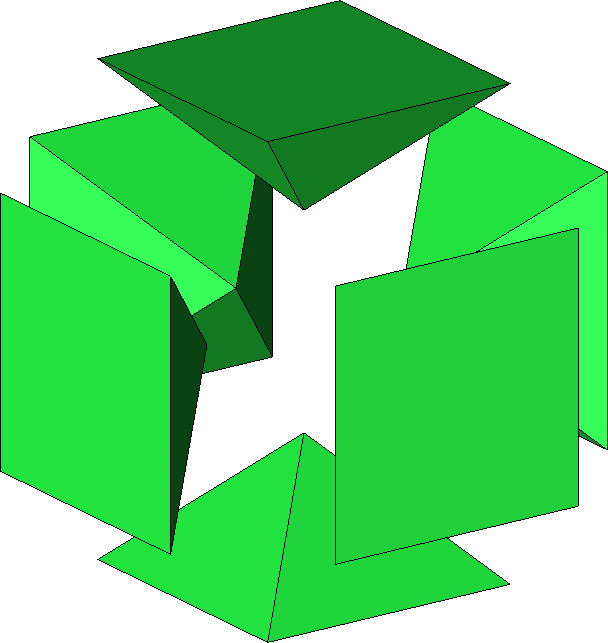}}
\caption{Subdivisions of hexahedra used to construct uniform meshes of wedges and pyramids.}
\label{fig:hex_sub}
\end{figure}
Figure~\ref{fig:elem_cube_rates} shows numerical convergence rates obtained using the resonant cavity solution 
\[
p(x,y,z,\tau) = \sin(\pi x) \sin(\pi y) \sin(\pi z) \cos(\sqrt{3}\pi \tau)
\]
over the unit cube $[0,1]^3$.  Since the solution is smooth, the best approximation in $L^2$ converges with a rate of $h^{N+1}$, though DG is guaranteed only a convergence rate of $h^{N+1/2}$ for general meshes.  We observe optimal $h^{N+1}$ convergence rates for the GL formulation, while the SEM formulation yields computed convergence rates somewhere in between $h^{N+1/2}$ and $h^{N+1}$ (though for the hex, these rates are less informative since SEM errors do not follow a constant rate as closely).  For all orders and meshes, the error for the GL formulation is lower than that of the SEM formulation, though this difference is less pronounced at higher $N$.  

\begin{figure}
\centering
\subfloat[Hex]{\includegraphics[width=.325\textwidth]{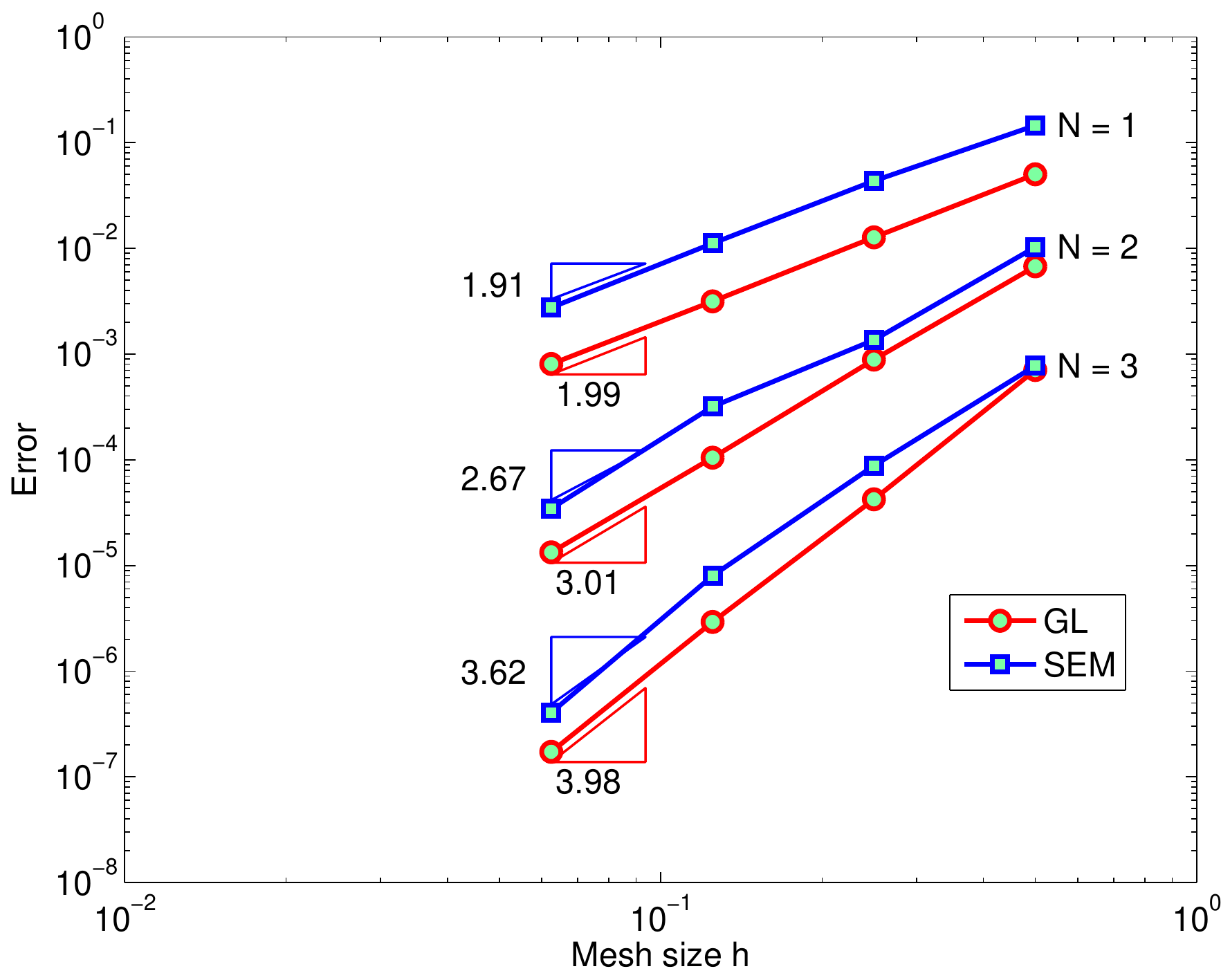}}
\subfloat[Wedge]{\includegraphics[width=.325\textwidth]{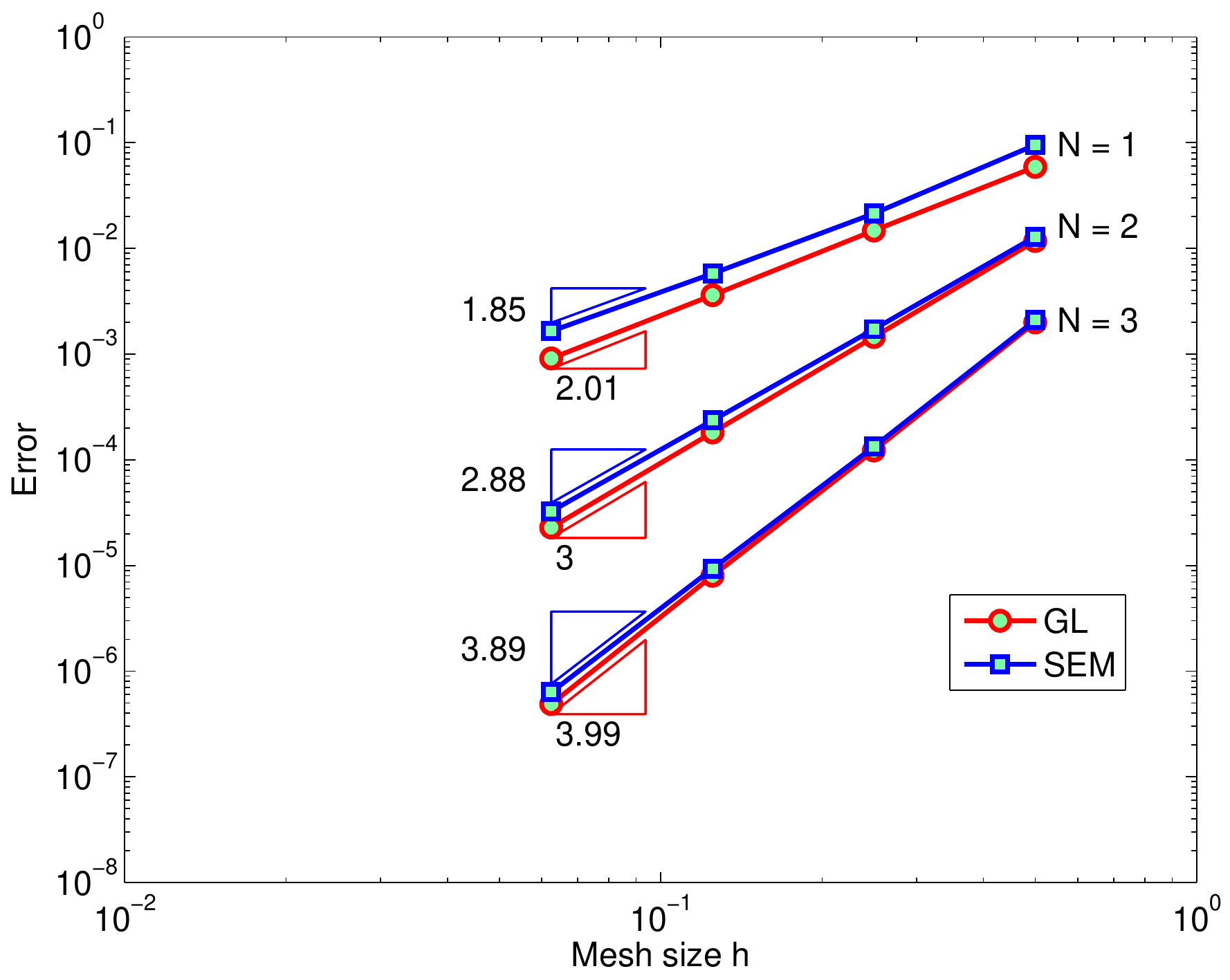}}
\subfloat[Pyramid]{\includegraphics[width=.325\textwidth]{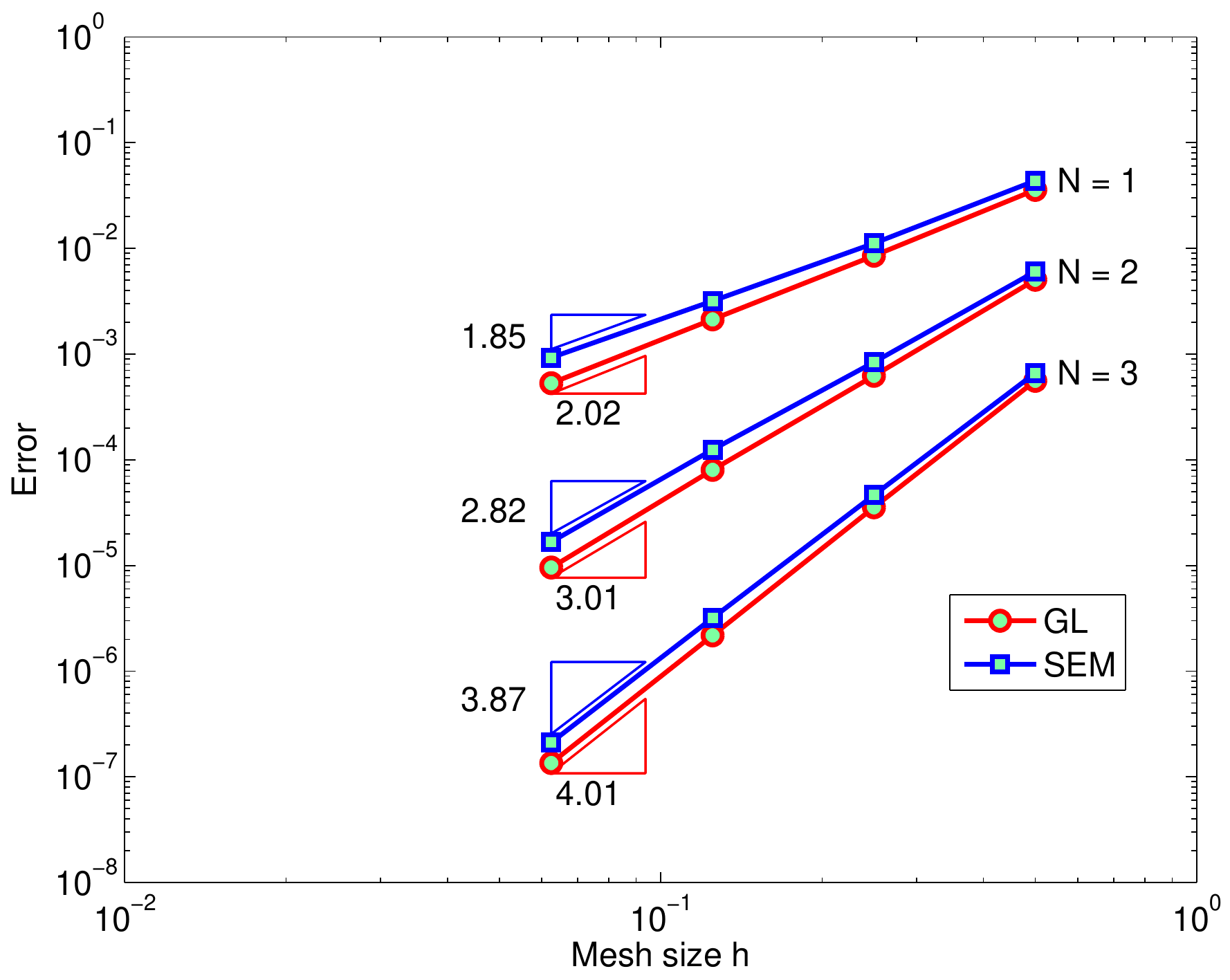}}
\caption{Convergence rates for both the SEM and GL formulations for meshes of hexes, wedges, and pyramids. }
\label{fig:elem_cube_rates}
\end{figure}

We also compare standard wedges to LSC-DG wedges in Figure~\ref{fig:lsc_convergence}  by performing convergence tests on two sequences of meshes.  For one sequence, we refine a uniform mesh of affine wedges with the GL formulation (referred to as ``GL Wedge'').  For the other, we begin with a mesh of randomly warped wedges, and refine the mesh by bisection to produce a sequence of meshes (referred to as ``GL LSC wedge'').   Optimal rates of convergence are observed for both sequences (for the SEM formulation, similar behavior to standard SEM wedges is also observed when using LSC SEM wedges).  We note that the observed optimal convergence rates for LSC-DG are likely dependent on the fact that refinement by bisection produces asymptotically affine elements.  If we perturb vertex positions at each level of mesh refinement, we do not in general observe optimal convergence rates.  This issue is observed also for standard mapped approximation spaces \cite{botti2012influence}, though it is likely exacerbated by the LSC-DG approximation.

\begin{figure}
\centering
\subfloat{\includegraphics[width=.3\textwidth]{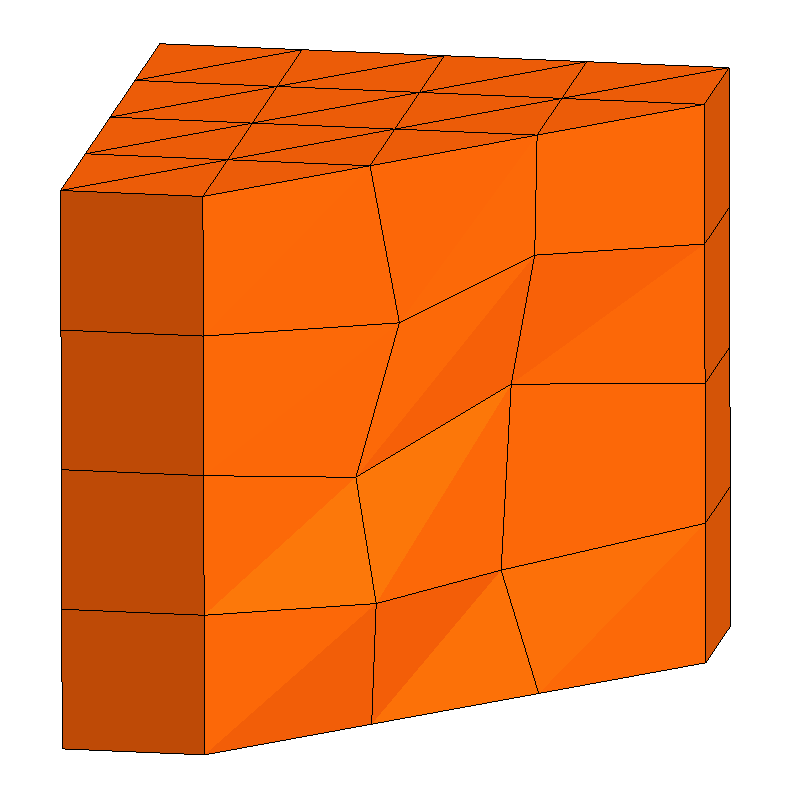}}
\hspace{1em}
\subfloat{\includegraphics[width=.375\textwidth]{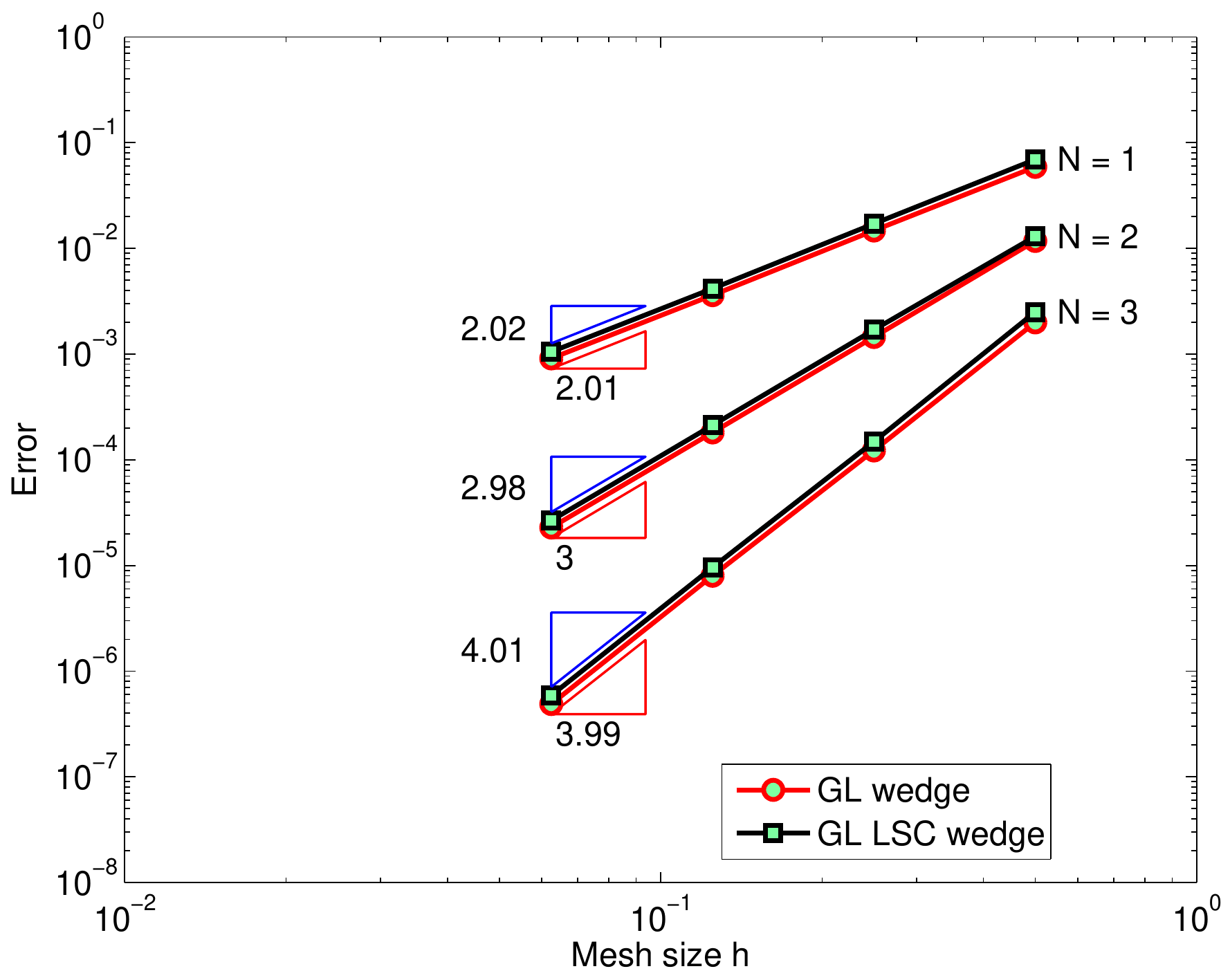}}
\caption{Convergence of LSC-DG (right) on a mesh of warped wedges (left).  The vertex positions are perturbed by $10\%$, and the mesh is then refined by splitting.  The convergence of LSC-DG is plotted using a black line, while red lines compare the convergence of DG on an un-warped mesh of wedges.}
\label{fig:lsc_convergence}
\end{figure}

\subsection{Verification: hybrid cube meshes}

We check also numerical convergence rates for unstructured hybrid meshes.  
%\[
%%p(x,y,z,t) = \cos\LRp{{\pi x}/{2}}\cos\LRp{{\pi y}/{2}}\cos\LRp{{\pi z}/{2}}\cos\LRp{{\sqrt{3}\pi t}/{2}}.
%\]
%over the unit cube $[0,1]^3$.  
$L^2$ errors are computed on a sequence of hybrid meshes containing elements of all types 
%(shown in Figure~\ref{fig:cubeMeshes})
.  Each element is refined by a self-similar splitting, except for pyramids, which are refined into both pyramids and tetrahedra.\footnote{Each element is refined according to the pattern specified by the ``refine by splitting'' option in GMSH \cite{geuzaine2009gmsh}.}  These meshes contain 201, 1704, 14016, and 113664 elements, respectively.  The breakdown of the number of each individual elements of each type is given in Table~\ref{table:hybrid_elems}.  5 MRAB levels are used, with a CFL constant of $0.5$.  Results are computed using double precision; single precision behaves similarly, though the error stalls at around $5\times 10^{-7}$.  

\begin{table}
\centering                                                                                   
\begin{tabular}{|c|c|c|c|c|c|}
\hline                   
& Hexahedra & Wedge & Pyramids & Tetrahedra & Total elements\\
\hline
Mesh 1 & 84 & 10 & 24 & 83 & 201\\
\hline
Mesh 2 & 672 & 80 & 96 & 856 & 1704\\
\hline
Mesh 3 & 5376 & 640 & 384 & 7616 & 14016\\
\hline
Mesh 4 & 43008 & 5120 & 1536 & 64000 & 113664\\
\hline
\end{tabular}
\caption{Number of elements of each type for each hybrid mesh used in convergence tests. }
\label{table:hybrid_elems}
\end{table}

%The first mesh contains 201 elements, with 84 hexahedra, 10 wedges, 24 pyramids, and 83 tetrahedra.  The second mesh contains 1704 elements, with 672 hexahedra, 80 wedges, 96 pyramids, and 856 tetrahedra.  The third mesh contains 14016 elements, with 5376 hexahedron, 640 prisms, 384 pyramids, and 7616 tetrahedron, while the fourth mesh contains 113664 elements, with 43008 hexahedra, 5120 wedges, 1536 pyramids, and 64000 tetrahedra.  
%\remark{Make table instead for element and number of elements.}

Under the assumption that the mesh size $h$ is halved at each time, we estimate the asymptotic convergence rate in Figure~\ref{fig:cubeMeshExplode} for $N=1,2,3$.  For both the SEM and GL formulations, we observe orders of convergence (in between $h^{N+1/2}$ and $h^{N+1}$) similar to those reported for single element-type meshes in Section~\ref{sec:elemverify}.  

\begin{figure}
\centering
\subfloat{\includegraphics[width=.3\textwidth]{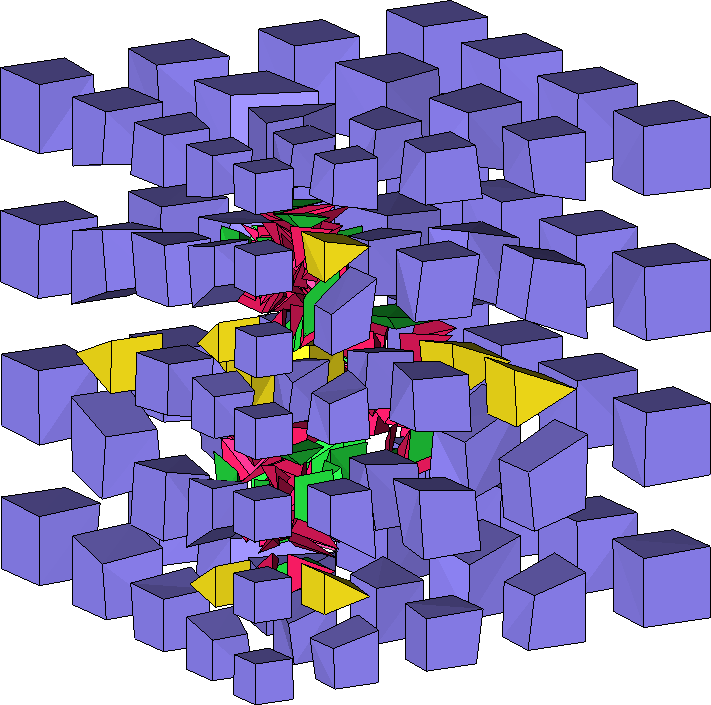}}
\hspace{1em}
\subfloat{\includegraphics[width=.375\textwidth]{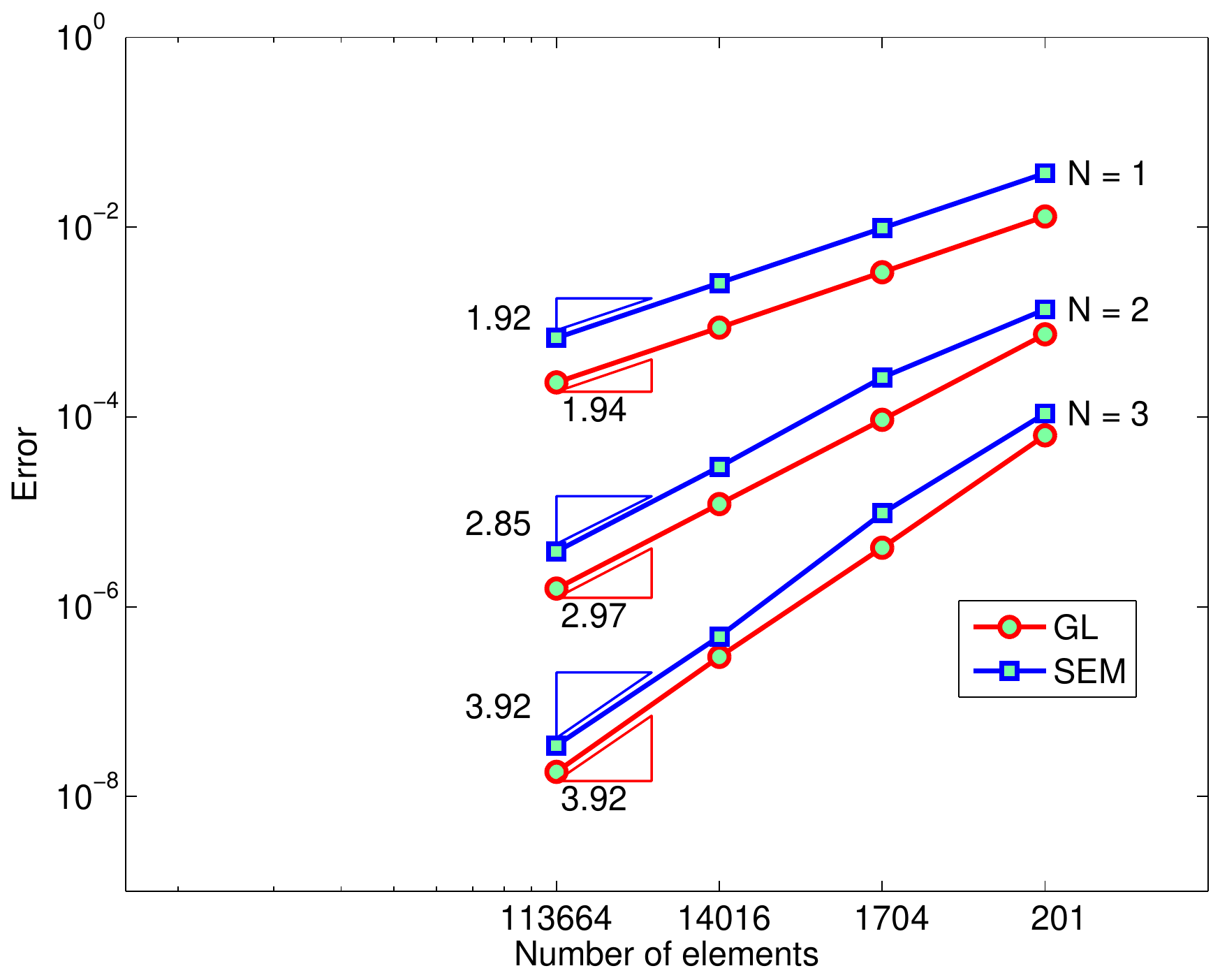}}
\caption{An exploded view of the hybrid mesh with 201 elements (left) and $L^2$ errors for both the SEM and GL formulation (right).}
\label{fig:cubeMeshExplode}
\end{figure} 

%Finally, we compare the utility of 
%
%\remark{Add part on multi-rate vs global timestepping for hybrid meshes.}

\subsection{Computational results}
\label{sec:comp}

In Section~\ref{sec:semvsgl} and \ref{sec:perelemcost}, we quantify the computational cost of DG solvers on hybrid meshes in terms of runtime per degree of freedom.  While the approximation power per degree of freedom varies from element to element, the cost per dof gives a rough estimate of the efficiency of each kernel.  In Section~\ref{sec:gflopsbw}, we quantify the computational efficiency of the solver in terms of estimated bandwidth and GFLOPS for each element kernel.  All computations were run on a single Nvidia GTX 980 GPU in single precision.  

\subsubsection{Cost of SEM vs GL formulations}
\label{sec:semvsgl}
Since the computational structure of the volume kernel is identical between the SEM and GL formulations, any additional cost associated with the GL formulation lies in the surface and update kernels.  We implemented two separate surface and update kernels and report the relative per-dof speedup of SEM compared to GL for orders $N=1,\ldots,5$ in Table~\ref{table:SEMvsGL}.  For all orders, the cost of the SEM update is lower or equal to the cost of the GL update.  For low orders, the SEM-tailored surface kernel performs slightly worse than that of the GL kernel\footnote{The slower runtime may be due to additional conditional statements and the structure of memory accesses in the SEM-tailored kernel.  We note that it is always possible to run the SEM formulation using GL kernels, and thus the cost of the SEM surface kernel can always be made in practice to be less than or equal to the cost of the GL kernel}.  However, as the order $N$ increases, the SEM kernel becomes cheaper due to the fact that data from interior nodes does not need to be accessed.  When comparing total runtime at $N>1$, the SEM formulation becomes roughly $5-10\%$ cheaper per-dof than the GL formulation, similar to the $10-15\%$ cost savings using SEM over GL as reported by Kopriva and Gassner on CPU architectures \cite{kopriva2010quadrature}.  

\begin{table}[!h]
\centering
\begin{tabular}{|c|c|c|c|c|c|}       
\hline                             
$N$ & 1 & 2 & 3 & 4 & 5 \\               
\hline                                
Ratio of surface cost of SEM vs GL & 1.1243 &    1.0597   &  0.9350   &  0.8986   &  0.8329\\
\hline
Ratio of update cost of SEM vs GL  &    0.9884 &   0.7437 &   1.0000  &  0.8408 &   0.9000\\
\hline
Ratio of total cost of SEM vs GL &  1.0432 &    0.9286&     0.9798&     0.9163&     0.9134\\
\hline
\end{tabular}
\caption{Total time-per-dof cost for SEM vs GL hexahedra.}
\label{table:SEMvsGL}
\end{table} 

\subsubsection{Cost per element type}
\label{sec:perelemcost}
In this section, we compare the computational cost of hexahedra, wedges, and pyramids relative to the computational cost of tetrahedra.  The results reported are for tuned computational kernels, where the number of elements processed per workgroup has been chosen in order to minimize the runtimes of the volume, surface, and update kernels for each element \cite{modave2015accelerated, gandham2014gpu}.  As suggested in \cite{klockner2009nodal}, automation of this process is crucial for portable performance across various architectures, especially for hybrid meshes where parameters must be tuned for 12 separate kernels.  

\begin{table}
\centering
\begin{tabular}{|c|ccccc|} 
\hline 
$N$ & 1 & 2 & 3 & 4 & 5 \\ 
\hline 
Hex SEM kernels & & & & & \\ 
\hline 
Volume & 1.2238 & 1.0373 & 1.1147 & 0.9538 & 0.7772\\
Surface & 0.9631 & 0.8735 & 0.5666 & 0.6038 & 0.5808\\
Update & 1.0254 & 0.8427 & 0.9161 & 0.8762 & 0.6667\\
Total (all kernels) & 1.0492 & 0.9108 & 0.8222 & 0.7987 & 0.6738\\
\hline 
Hex GL kernels & & & & & \\ 
\hline 
Volume & 1.2238 & 1.0373 & 1.1147 & 0.9538 & 0.7772\\
Surface & 0.8567 & 0.8243 & 0.6060 & 0.6719 & 0.6974\\
Update & 1.0374 & 1.1332 & 0.9161 & 1.0422 & 0.7407\\
Total (all kernels) & 1.0058 & 0.9808 & 0.8392 & 0.8717 & 0.7377 \\
\hline 
Wedge kernels & & & & & \\ 
\hline 
Volume & 2.965 & 1.800& 2.999& 2.596& 2.975\\
Surface & 0.912& 0.778& 0.99& 2.086& 1.566\\
Update & 1.414 & 1.02& 1.111& 1.321& 1.485\\
Total (all kernels) & 1.5937 & 1.1450 & 1.5972 & 2.0407 & 2.0242\\
\hline 
Pyramid kernels & & & & & \\ 
\hline 
Volume & 1.912& 1.653& 2.418& 3.228& 2.842\\
Surface & 0.959 &0.855 &0.634 &1.657 &1.157\\
Update & 1.51 &1.033& 1.112& 1.169& 1.804\\
Total (all kernels) & 1.3705 & 1.1386 & 1.2788 & 2.0470 & 1.9280 \\
\hline 
\end{tabular}
\caption{Time-per-dof cost (relative to tetrahedra) of hexahedra, wedge and pyramid kernels.}
\label{table:perDofTimes}  
\end{table}
Table~\ref{table:perDofTimes} shows the time-per-dof cost for $K\approx 100000$ hexahedral, wedge, or pyramidal elements relative to the time-per-dof cost of $K\approx 100000$ tetrahedral elements.  While hexahedra are observed to be faster per-dof than tetrahedra at higher orders, wedge and pyramidal volume kernels are observed to be 1-2 times slower per-dof than tetrahedra kernels.  This may be due in part to the fact that, for planar tetrahedra, geometric factors are constant over an element.  As a result, only a few values must be loaded per element independently of $N$.  This is in contrast to hexahedra, wedges and pyramids, for which the number of geometric factors to be loaded increases along with the total number of nodal or cubature points.  

%\begin{table}
%\centering                  
%\begin{tabular}{|c|c|c|c|c|c|}       
%\hline                             
%$N$ & 1 & 2 & 3 & 4 & 5 \\               
%\hline
%Hex total (SEM) &    1.0492  & 0.9108 &    0.8222 &   0.7987 &  0.6738\\
%\hline
%Hex total (GL) &    1.0058  & 0.9808 &    0.8392 &   0.8717 &  0.7377 \\
%\hline
%Wedge total &       1.5937  &  1.1450  &  1.5972  &  2.0407  &  2.0242\\
%\hline
%Pyramid total &        1.3705  &  1.1386  &  1.2788  &  2.0470  &  1.9280 \\
%\hline
%\end{tabular}               
%\caption{Total time-per-dof cost (relative to tetrahedra) of hex (GL/SEM), wedge, and pyramid kernels.  }
%\label{table:perDofTimes}  
%\end{table} 

We note that these numbers give optimistic estimates on per-dof costs for two reasons.  First of all, for meshes containing only planar tetrahedra, the cost of the surface and update kernels may be further reduced by replacing surface quadratures with lift operators for nodal bases \cite{hesthaven2007nodal, klockner2009nodal}.  Secondly, on hex-dominant meshes, the number of wedges and pyramids is expected to be much smaller than the number of hexahedra and tetrahedra.  The runtimes reported previously are for asymptotically ``large'' numbers of elements, and may increase for a small number of elements.

\begin{figure}
\centering
\subfloat[Hexahedron]{\includegraphics[width=.245\textwidth]{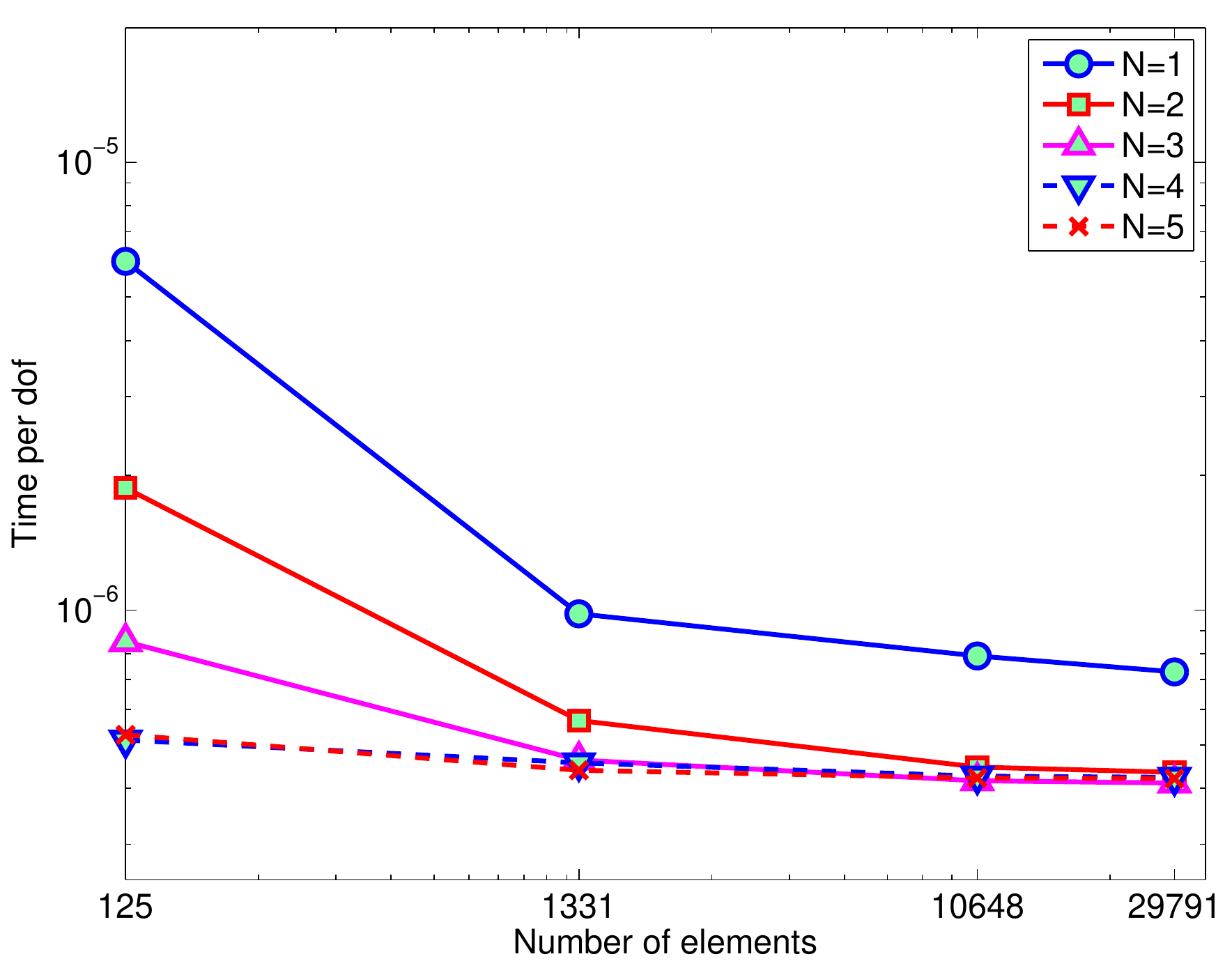}}
\subfloat[Wedge]{\includegraphics[width=.245\textwidth]{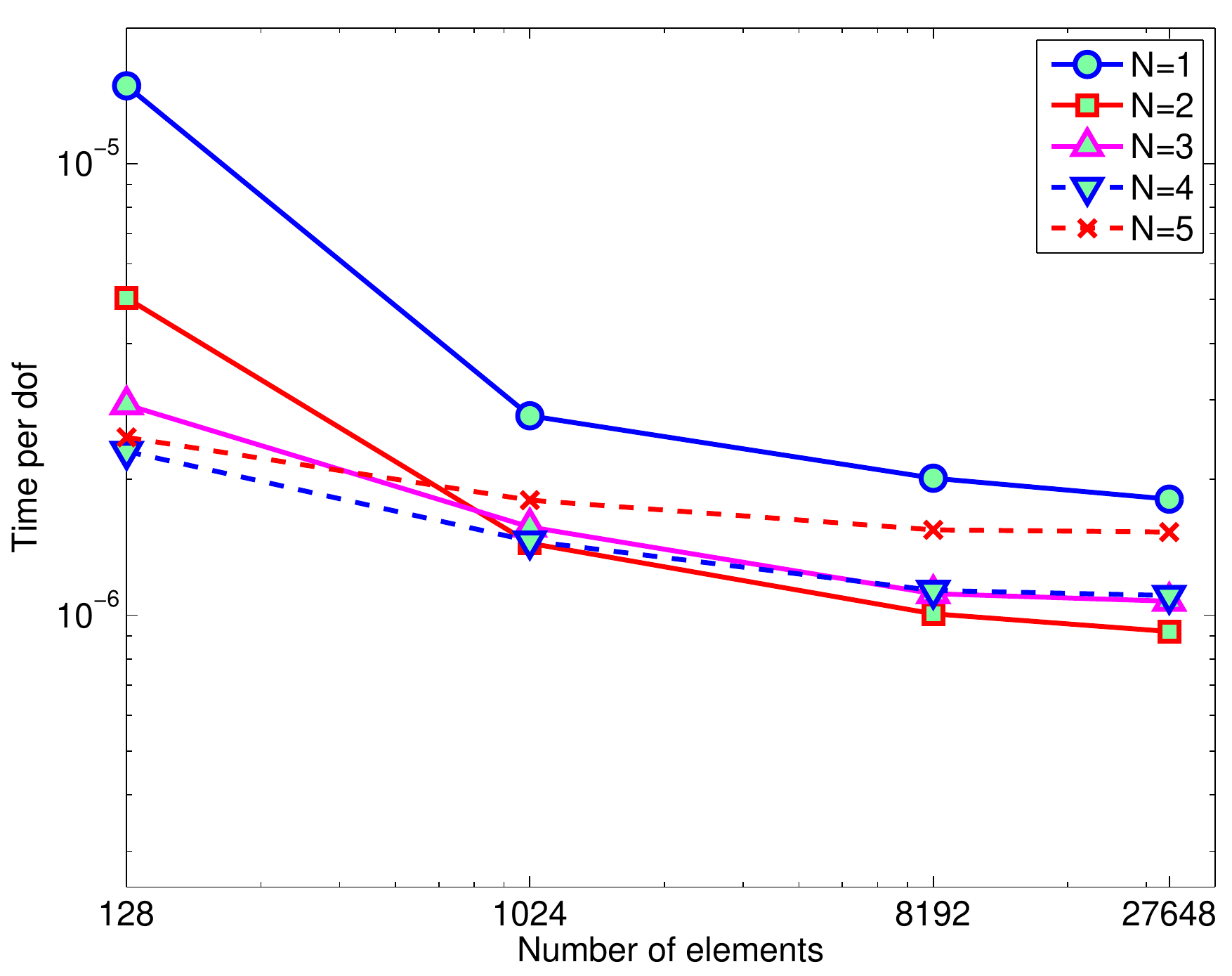}}
\subfloat[Pyramid]{\includegraphics[width=.245\textwidth]{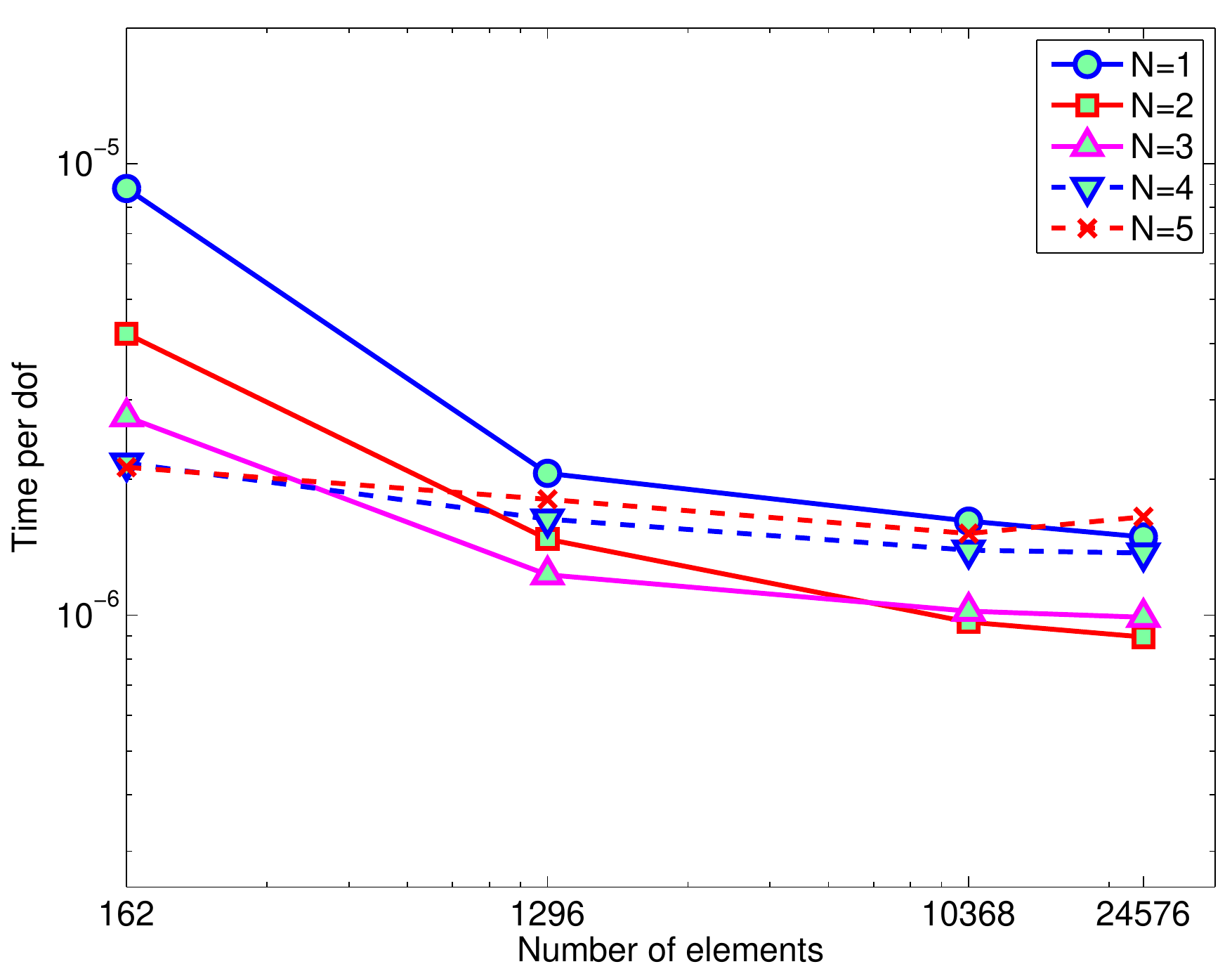}}
\subfloat[Tetrahedra]{\includegraphics[width=.245\textwidth]{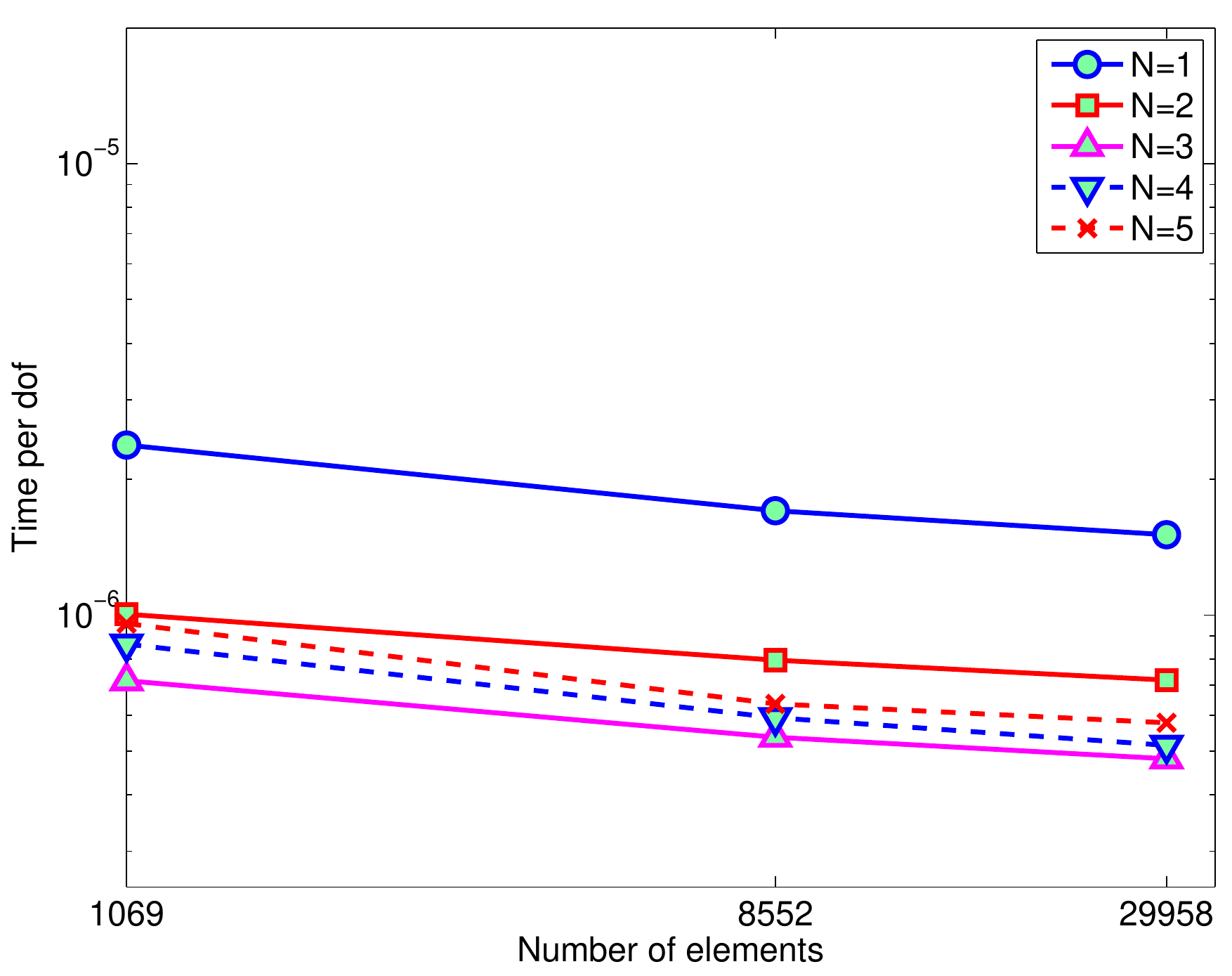}}
\caption{Time-per-dof cost of hexahedral, wedge, pyramidal, and tetrahedral volume kernels on meshes with $K\approx 100, 1000, 10000$ and $30000$ elements at various orders $N$.}
\label{fig:perDofTimes}
\end{figure} 

To illustrate the dependence of runtime on number of elements, we timed the cost of the volume kernel over 500 timesteps on meshes containing approximately $100, 1000, 10000$, and $30000$ elements.  For ease of presentation, we show results only for the volume kernel, though the behavior is similar for the surface and update kernels.  The order of approximation is varied from $N=1,\ldots, 5$ for each mesh, and the kernel runtime is reported in Figure~\ref{fig:perDofTimes}.  Since the computational structure of each volume kernel is the same between the SEM and GL formulations (with the exception of the pyramid\footnote{Under the GL formulation, the pyramid volume kernel may use the more efficient strong formulation.  We show timings, GFLOPS, and estimated bandwidth for the skew-symmetric pyramid kernel only.}) we show timings for only four volume kernels.  

Since the time-per-dof decreases as more elements are processed, if a mesh contains a small number of wedges or pyramids, the actual cost of a wedge or pyramid relative to a tetrahedron may be slightly greater than the reported values of Table~\ref{table:perDofTimes}, though this effect is less pronounced at higher orders of approximation.  For example, for $N=3$ on the largest hybrid mesh (consisting of 113664 elements, with 5120 prisms, 1536 pyramids, and 64000 tetrahedra), prisms and pyramids were respectively 1.7563 and 1.8849 times more expensive per dof than tetrahedra, instead of the estimated 1.6 and 1.28 times suggested by Table~\ref{table:perDofTimes}.

\subsubsection{Estimated GFLOPS and bandwidth}
\label{sec:gflopsbw}
Finally, we present estimated GFLOP and effective bandwidth counts for each kernel in Table~\ref{table:GFLOPS_BW}.  For reference, we include the estimated bandwidth and GFLOPS of a repeated entrywise multiplication kernel in Figure~\ref{fig:GFLOPS_BW_ref}.  For an array $A$ of sufficiently large size (running with the maximum number of threads), we execute $N_{\rm mult}$ times the command $A_{ij} = A_{ij}^2 - C$, where $C$ is some constant.  As $N_{\rm mult}$ increases, the kernel changes from being IO bound to compute-bound.  We achieve roughly $1970$ GFLOPS and 168 GB/s bandwidth at most on an Nvidia GTX 980.  
When utilizing \verb+float4+ operations, we observe the same peak bandwidth, but achieve $2930$ GFLOPS at peak, roughly $1.5\times$ the performance using only floats.  While these numbers are not necessarily indicative of peak performance,  we believe they are representative of  ``good'' performance for a given kernel.
\begin{figure}
\centering
\subfloat[Standard float operations]{\includegraphics[width=.32\textwidth]{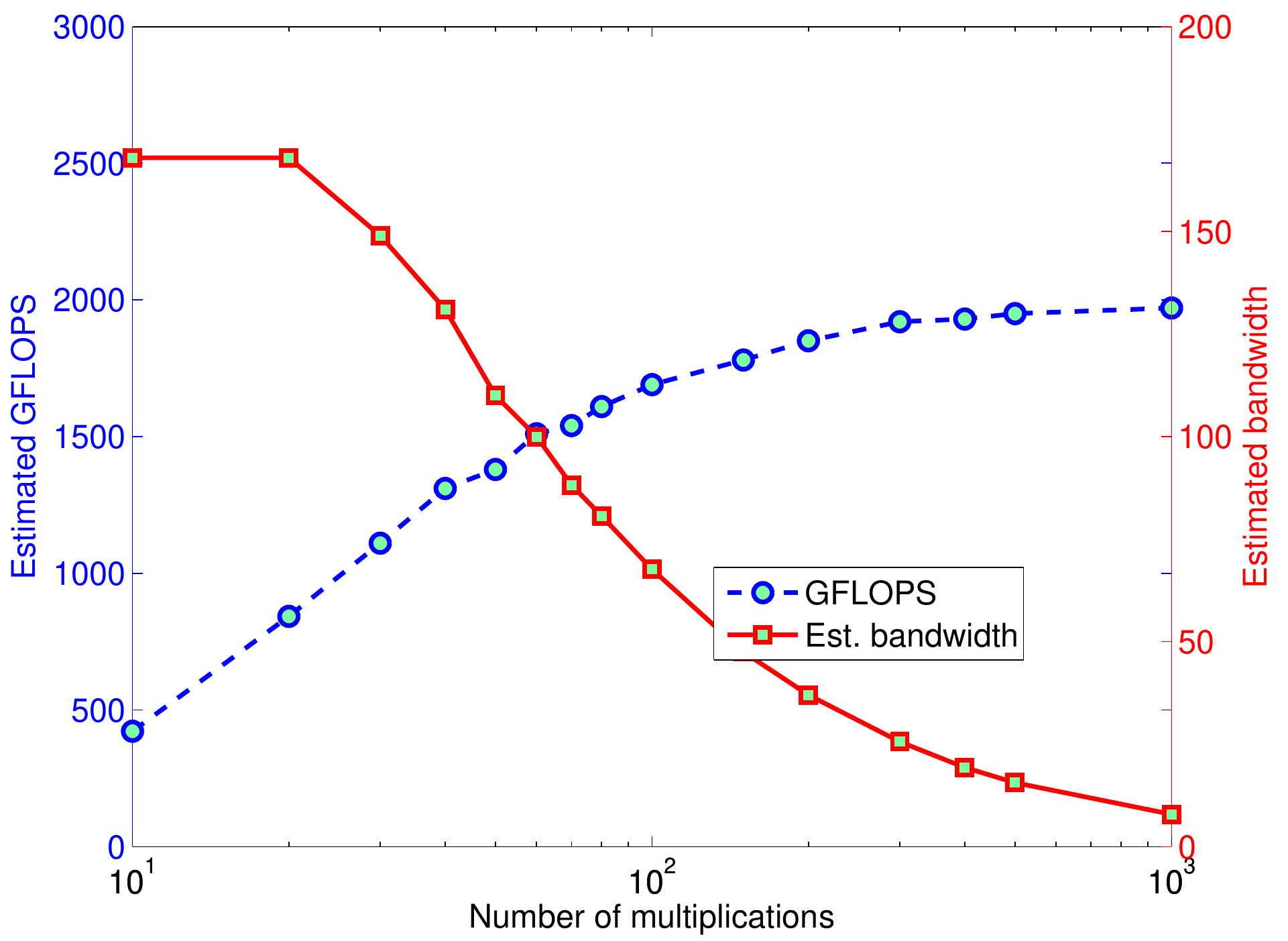}}
\hspace{3em}
\subfloat[Float4 operations]{\includegraphics[width=.32\textwidth]{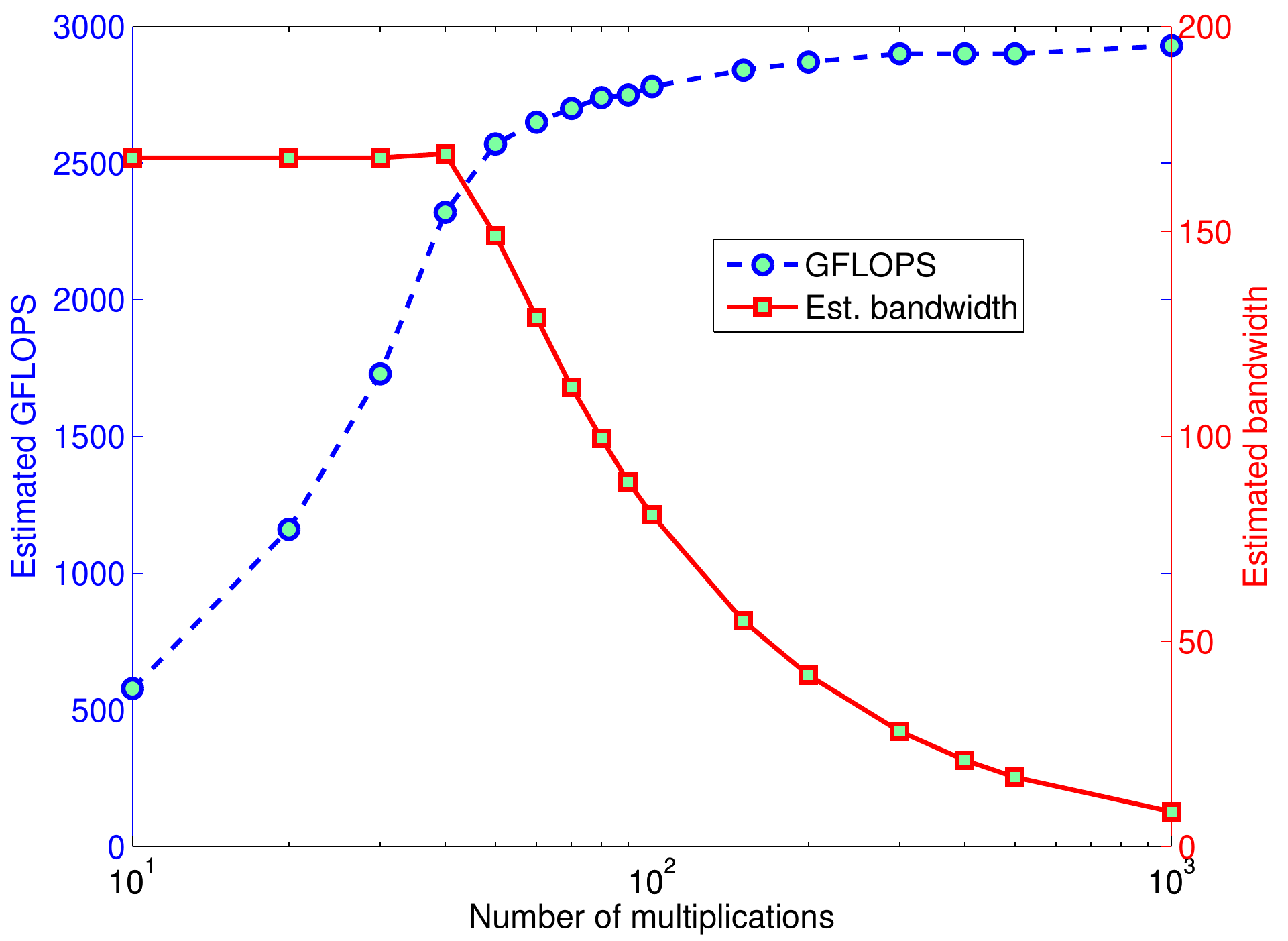}}
\caption{GFLOPS and bandwidth (for a kernel executing repeated entrywise array multiplication/addition) as a function of the number of repeated multiplications $N_{\rm mult}$.  As $N_{\rm mult}$ increases, the kernel goes from being IO bound to compute-bound.  Results from two kernels are presented: one where computations are done using only standard single precision floats, and one where computations are performed using the float4 data type.}
\label{fig:GFLOPS_BW_ref}
\end{figure}

With this in mind, we may compare the GFLOPS and bandwidth of kernels for each element. Comparing only the volume kernels, we see efficiency resulting from both effective hardware utilization and reduced computational load.  For example, the wedge volume kernel shows low estimated bandwidth with high estimated GFLOPS, indicating that its performance is compute-bound.  We note that the observed GFLOPS are relatively high, which may be due to the heavy use of \verb+float4+ operations within the wedge volume kernel.  On the other hand, the bandwidth and GFLOPS for the tetrahedron volume kernel are more balanced.  Finally, though it is similar in structure to the tetrahedral kernel, the pyramid volume kernel shows high estimated bandwidth and low GFLOPS, implying that it is IO bound.  This is likely due to the fact that the skew-symmetric form is used, and requires the loading of twice as many operators (derivative matrices and their transposes) as the tetrahedral kernel.  

For hexahedral kernels, the operators are explicitly stored in shared memory, removing caching effects.  As a result, the reported bandwidth and GFLOPS are both significantly lower than for other elements, but near-peak bandwidth numbers are achieved for $N\geq 3$, implying the kernel is IO bound.  This near-peak performance is likely due to the loading of multiple geometric change-of-variables factors per node per element and the low arithmetic intensity of tensor-product operations, which are effectively hidden during memory retrievals.  It has been noted in \cite{vNek} that, for vertex-mapped hexahedra, the IO bound nature of the hex volume kernel may be addressed by computing these geometric factors on the fly inside the kernel, loading only vertex positions (regardless of order).  

For wedges, pyramids, and tetrahedra, bandwidth is reported both with and without counting operator loads per element.  The reported bandwidth numbers with operator loads are often higher than peak rates, which is likely due to the fact that compiler optimizations and caching effects are not taken into account during estimates.  Removing operator loads from bandwidth counts reveals that the bandwidth for prism, pyramid, and tetrahedron kernels decreases as $N$ increases, verifying that operations for these three element types are compute-bound at high orders.  

\begin{table}
\centering
\subfloat[GFLOPS]{
\begin{tabular}{|c|ccccc|} 
\hline 
$N$ & 1 & 2 & 3 & 4 & 5 \\ 
\hline 
Hex SEM  & & & & & \\ 
\hline 
Volume & 100 &	134&	173&	199 &	232\\
Surface & 68&	51&	57&	47&	35\\
Update & 46	 &34&	49&	40&	44\\
\hline 
Hex GL  & & & & & \\ 
\hline 
Surface & 60&	44&	52&	43&	32\\
Update & 45	& 46&	49&	48&	48\\
\hline 
Wedge  & & & & & \\ 
\hline 
Volume & 242&	552&	1032&	1551&	1956\\
Surface & 197&	418&	566 &	452&	771\\
Update & 166&	403&	736 &	904&	960\\
\hline 
Pyramid  & & & & & \\ 
\hline 
Volume & 77&	161&	238&	259&	392\\
Surface & 187&	344&	850&	482&	922\\
Update & 136&	333&	699&	898&	670\\
\hline 
Tet  & & & & & \\ 
\hline 
Volume & 123&	200&	401&	550&	699\\
Surface & 149&	228&	423&	596&	791\\
Update & 161&	260&	599&	779&	931\\
\hline 
\end{tabular}
}
\subfloat[Est. bandwidth]{
\begin{tabular}{|ccccc|} 
\hline 
 1 & 2 & 3 & 4 & 5 \\ 
\hline 
 & & & & \\ 
\hline % SEM
152	& 160	& 169& 	165& 	167\\
88& 	68	& 81	& 70& 	54\\
159	& 112& 	155& 	125& 	133\\
\hline 
 & & & & \\ 
\hline % GL
77&	59&	74&	64&	49\\
160&	152&	155&	149&	148\\
\hline 
 & & & & \\ 
\hline % Wedge
217/149 &	271/156	&294/122 &	316/107	&300/79\\
166/103&273/98	&328/70	&247/31	&408/33\\
177/116 &309/132&	472/128 &533/99	&536/69
\\
\hline 
 & & & & \\ 
\hline % Pyr
220/102	&	444/106	&646/83	&	699/52	&	1054/49\\
163/107	&	236/100	&	510/131	&	271/46	&	497/55\\
155/103		&270/123	&	464/137		&546/116	&397/58	\\
\hline 
 & & & & \\ 
\hline % Tet
238/145 &	322/125&	595/139&	784/116&	974/95\\
122/81	&147/65	&	244/64	&327/58	&	421/50\\
205/147	&238/129&	434/160	&509/144&	559/112\\
\hline 
\end{tabular}
}
\caption{GFLOPS and estimated effective bandwidth counts for each kernel of each element type (after optimization of the number of elements processed per workgroup).  For wedges, pyramids, and tetrahedra, bandwidth is reported $(\text{with operator loads})/(\text{without operator loads})$. }
\label{table:GFLOPS_BW}  
\end{table}

\section{Conclusions}

We have presented in this paper a high order discontinuous Galerkin solver on hybrid, hex-dominant meshes.  A careful selection of basis functions controls the storage cost of the scheme, making it suitable for acclerators and GPUs.  Two energy-stable formulations are given, based on either SEM or Gauss-Legendre nodal bases for the hexahedron.  Explicit bounds on the spectra are computed in terms of the order-dependent constants in the trace inequalities for each type of element, and are used to determine local timestep restrictions for multi-rate timestepping.  Convergence rates are reported for each formulation, and computational cost is assessed by estimating the GFLOPS, bandwidth, and cost-per-dof of each element.  

The focus of this work has been mainly based presenting a DG solver for hybrid meshes on a single GPU.  Future work will focus on the following areas: 
\begin{itemize}
\item \textbf{Wedge basis:} while the LSC-DG wedge basis achieves optimal rates of convergence, there is additional complexity and computational cost compared to polynomial bases (necessity of quadrature, storage of derivatives of $J$).  A low-storage polynomial wedge basis could reduce costs for wedge volume kernels and surface/update kernels for wedges, pyramids, and tetrahedra.  
\item \textbf{Hybrid meshes compared to tetrahedral meshes:} while cost-per-dof comparisons have been made between tetrahedra and other element types, DG on hex-dominant meshes should be compared to optimized nodal DG on fully tetrahedral meshes.  The effect on efficiency of the ratio between hexes and elements of other types should also be carefully investigated.  
\item \textbf{Parallelization and scalability:} multiple GPUs are necessary to handle larger problem sizes.  Performance of multiple GPU-accelerated DG will require load balancing and communication strategies for multi-rate solvers on hybrid meshes.  
\end{itemize}
%We hope to also explore efficient solvers for piecewise material data that is locally varying.  

%\remark{Add discussion on future work.}
%Zheng Wang is supported by \remark{add source}.  

\section{Acknowledgements}

Jesse Chan is supported by the Rice University CAAM Department Pfieffer Postdoctoral Fellowship and by NSF (award number DMS-1216674).  He would additionally like to acknowledge and thank the Jennifer J. Young Memorial Fund.  Axel Modave is an Honorary Fellow of the Belgian American Educational Foundation (BAEF), and acknowledges Wallonie-Bruxelles International (WBI) and Shell for their support.  T. Warburton is supported partially by ANL (award number 1F-32301, subcontract on DOE DE-AC02-06CH11357).  

\bibliography{hybridg}{}
\bibliographystyle{plain}

\appendix

\section{Explicit trace inequalities for tensor product elements}

The constant in the surface trace inequality for tensor product elements may be computed explicitly.  Throughout these proofs, we define $\phi_j$ be the normalized Legendre polynomials, orthogonal with respect to the standard $L^2([-1,1])$ inner product. We begin by proving a result in 1D.  

\begin{lemma}
Let $u\in P^N([-1,1])$.  Then, 
\[
\LRb{u(-1)} + \LRb{u(1)} \leq \frac{(N+1)(N+2)}{2} \nor{u}_{L^2([-1,1])}^2
\]
\end{lemma}
\begin{proof}
In 1D, we wish to derive an expression for the spectral radius $\rho(M_s)$, where
\[
(M_s)_{ij} = \phi_j(-1)\phi_i(-1) +  \phi_j(1)\phi_i(1).
\]
and $\phi_j(x)$ are normalized Legendre polynomials.  Since Legendre polynomials are symmetric across $x=0$, $\phi_j(-1) = \phi_j(1)$ for $j$ even, and  $\phi_j(-1) = -\phi_j(1)$ for $j$ odd.  Then, assuming $i$ and $j$ do not share the same parity, we have
\[
(M_s)_{ij} = \phi_j(-1)\phi_i(-1) +  \phi_j(1)\phi_i(1) = 0.
\]
This implies $M_s$ is block diagonal, with each block corresponding to even/odd powers of the Legendre polynomial.  Both the even block $M^e$ and odd block $M^o$ reduce down to 
\[
(M^e)_{ij} = (M^o)_{ij} = 2\phi_j(-1)\phi_i(-1)  = 2\phi_j(1)\phi_i(1).
\]
These are both rank 1 matrices, such that
\[
\rho(M^e) = 2\sum_{j = {\rm even}} \phi_j(-1)^2, \quad \rho(M^o) = 2\sum_{j = {\rm odd}} \phi_j(-1)^2.
\]
For both even and odd $i,j$, $\phi_j(-1)^2 = (2i+1)/2$, implying that 
\[
\rho(M_s) = 2\max\LRp{\rho(M^e), \rho(M^o)} = \frac{(N+1)(N+2)}{2}.
\]
\end{proof}

The extremal polynomial for the above inequality may also be characterized explicitly.  
\begin{lemma}
Define $r_i$ to be the quadrature points of the $(N+2)$-point Gauss-Legendre-Lobatto (GLL) rule.  Let $u$ be the order $(N+1)$ Lagrange polynomial that is zero at $N$ interior GLL nodes and unity at either $r=1$ or $r = -1$.  Then,
\[
\LRb{u(-1)} + \LRb{u(1)} = \frac{(N+1)(N+2)}{2} \nor{u}_{L^2([-1,1])}^2
\]
\end{lemma}
\begin{proof}
We note that, while the $(N+1)$-point GLL quadrature does not integrate polynomials of order $2N$ exactly, polynomials of order $2N$ are integrated exactly by the $(N+2)$-point GLL quadrature.  

Since $u(r)$ is symmetric or antisymmetric across $r = 0$ (this may be shown by explicitly writing out the interpolating Lagrange polynomial), $u(1) = \pm u(-1)$ depending on whether $N$ is even or odd.  Then, $\nor{u}^2_{L^2}$ is given as
\[
\nor{u}^2_{L^2} = \int_{-1}^1u^2 = \sum_{j=1}^{N+2} w_j u(r_j) = w_1 u(-1)^2 + w_{N+2} u(1)^2 = u(1)^2 (w_1 + w_{N+2}).
\]
The surface integral of $u^2$ is just $u(1)^2 + u(-1)^2 = 2u(1)^2$.  As a result, 
\[
\frac{2u(1)^2 }{u(1)^2 (w_1 + w_{N+2})} = \frac{1}{w_1} = \frac{(N+1)(N+2)}{2}.
\]
by the fact that the GLL weights for an $(N+2)$-point rule are each $\frac{2}{(N+1)(N+2)}$ at the endpoints.  
\end{proof}

A curious coincidence is that, since the $N+1$-point Gauss-Radau-Jacobi (GRJ) rule contains all but one endpoint of the $(N+2)$-point GLL rule, the extremal polynomial is identical to the extremal polynomial for triangle faces, which is unity at the $-1$ GRJ point and zero at all others \cite{warburton2003constants}.  

%In 1D, this function may also be constructed in an alternative fashion: for even $N$, the extremal polynomial is $(I-\Pi_b)1$, where $\Pi_b$ is the $L^2$ projection onto polynomial bubble functions.  For odd $N$, the extremal polynomial is $(I-\Pi_b)r$.  This construction does not appear to carry over to non-tensor product elements.  

We may prove a similar result, replacing the mass matrix with the underintegrated Gauss-Legendre-Lobatto quadrature mass matrix.  
\begin{lemma}
Let $u\in P^N([-1,1])$.  Then, for $N\geq 1$, the following is bound is tight
\[
\LRb{u(-1)} + \LRb{u(1)} \leq \frac{N(N+1)}{2} \nor{u}_{\rm SEM}^2.
\]
\end{lemma}
\begin{proof}
We take the basis $\phi_i(r) = \ell_i(r)/\sqrt{w_i}$, where $\ell_i(r)$ is the Lagrange polynomial at the $i$th GLL node and $w_i$ are the quadrature weights of the $(N+1)$-point GLL rule.  Since this basis is orthogonal with respect to the SEM inner product, the constant in the trace inequality is the spectral radius of the surface mass matrix under this basis, which has entries $(M_s)_{ij} = \phi_i(-1)\phi_j(-1)+\phi_i(1)\phi_j(1)$.  Applying the Lagrange property, we see that $M_s$ is diagonal with only two nonzero diagonal entries 
\[
(M_s)_{11} = \frac{1}{w_1}, \qquad (M_s)_{N+1,N+1} = \frac{1}{w_{N+1}}.
\]
$\rho(M_s)$ is the maximum of these two entries, and the bound is proven by noting $w_1 = w_{N+1} = \frac{2}{N(N+1)}$.  The extremal polynomial for this bound is any convex combination of $\phi_1(r)$ or $\phi_{N+1}(r)$. 
\end{proof}

\subsection{Tensor product elements}
 \label{app:TPtrace}
For a $d$-dimensional hypercube $[-1,1]^d$, we may define an orthonormal tensor product basis as 
\[
\phi_{i_1,\ldots, i_d} = \prod_{k=1}^d \phi_{i_k}(r_{k}), \quad 0 \leq i_k \leq N,
\]
where $r_k$ is the coordinate in the $k$th direction.  Using the above basis, we may prove the following theorem:
\begin{theorem}
Let $u \in {\rm span}\LRc{\phi_{i_1,\ldots, i_d}}$ for $0\leq i_1,\ldots i_d \leq N$.  Then, 
\[
\nor{u}^2_{L^2\LRp{\partial \widehat{K}}} \leq d\frac{(N+1)(N+2)}{2}\nor{u}_{L^2\LRp{\widehat{K}}}
\]
Furthermore, this bound is tight.  
\end{theorem}

\begin{proof}
The surface mass matrix over the $d$-dimensional hypercube is 
\[
(M_s)_{i_1, \ldots, i_d, j_1, \ldots, j_d} = \sum_{k=1}^d \LRp{\phi_{i_k}(-1)\phi_{j_k}(-1) + \phi_{i_k}(1)\phi_{j_k}(1)} \prod_{l\neq k} \int_{-1}^1 \phi_{i_l}\phi_{j_l}\diff r_l.
\]
By the orthogonality of Legendre polynomials, this reduces to
\begin{align*}
(M_s)_{i_1, \ldots, i_d, j_1, \ldots, j_d} &= \sum_{k=1}^d (M^k_s)_{i_1, \ldots, i_d, j_1, \ldots, j_d}, \\ 
(M^k_s)_{i_1,\ldots, i_d, j_1,\ldots, j_d} &= \LRp{\phi_{i_k}(-1)\phi_{j_k}(-1) + \phi_{i_k}(1)\phi_{j_k}(1)} \prod_{l\neq k} \delta_{i_l j_l}.
\end{align*}
This results in a block diagonal matrix in all indices $i_l,j_l$ for $l\neq k$.  Also, by the same argument as in the 1D case, if pairs of indices $i_k,j_k$ do not share the same parity, 
\[
\LRp{\phi_{i_k}(-1)\phi_{j_k}(-1) + \phi_{i_k}(1)\phi_{j_k}(1)} = 0,
\]
implying that each matrix $M_s^k$ consists of two blocks of even and odd indices.  Additionally, each matrix $M_s^k$ contains the same entries as the 1D surface mass matrix and are identical up to a permutation.  An application of Weyl's inequality then gives
\[
\rho\LRp{M_s} \leq \sum_{k=1}^d\rho\LRp{M_s^k} = d\frac{(N+1)(N+2)}{2}.
\]

We may show that this bound is tight by constructing the extremal polynomial in $d$ dimensions as the tensor product of the 1D polynomial.  Take $u_{\rm 1D}(r_k)$ to be the 1D extremal polynomial in the $k$th coordinate and define 
\[
u(r_1,\ldots, r_d) = \prod_{k=1}^d u_{\rm 1D}(r_k).
\]
The $(N+2)^d$ tensor product GLL rule integrates this exactly, such that for the reference element $\widehat{K} = [-1,1]^d$, 
\begin{align*}
\nor{u}_{L^2\LRp{\widehat{K}}}^2 %&= \sum_{i_1 = 1}^{N+2}\sum_{i_2 = 1}^{N+2}\ldots \sum_{i_d = 1}^{N+2}  w_{i_1,\ldots, i_d} u^2(r_{1,i_1},\ldots, r_{d,i_d})\\
&=\sum_{i_1 = 1}^{N+2}\sum_{i_2 = 1}^{N+2}\ldots \sum_{i_d = 1}^{N+2}  w_{i_1}w_{i_2}\ldots w_{i_d} u^2_{\rm 1D}(r_{1,i_1})u^2_{\rm 1D}(r_{2,i_2})\ldots u^2_{\rm 1D}(r_{d,i_d})
= 2^{d} w_1^d.%\sum_{k=1}^d (w_{1} u^2_{\rm 1D}(-1) + w_{N+2} u^2_{\rm 1D}(1) )\sum_{i_l = 1, l\neq k}^{N+2} \prod_{l\neq k}w_{i_l} u^2_{\rm 1D}(r_{l,i_l}).
\end{align*}
where $2^d$ is the number of vertices on the $d$-dimensional hypercube.  

The surface norm of $u$ for a $d$-dimensional hypercube is the sum of the integrals over the $2d$ faces.  Since each face itself is a hypercube of dimension $d-1$, the integral of $u$ over the face is $2^{d-1} w_1$, and the surface norm of $u$ reduces to
\begin{align*}
\nor{u}_{L^2\LRp{\widehat{K}}}^2 &= \sum_{f=1}^{2d} 2^{d-1}w_1^{d-1} = d 2^d w_1^{d-1},
\end{align*}
and we may conclude that
\[
\frac{\nor{u}_{L^2\LRp{\partial \widehat{K}}}^2}{\nor{u}_{L^2\LRp{\widehat{K}}}^2} = \frac{d}{w_1} = d\frac{(N+1)(N+2)}{2}. 
\]
\end{proof}

For $N>1$, similar steps may be followed to prove the equivalent trace inequality under SEM quadrature in $d$ dimensions 
\[
\nor{u}^2_{{\rm SEM}\LRp{\partial \widehat{K}}} \leq d\frac{N(N+1)}{2}\nor{u}_{{\rm SEM}\LRp{\widehat{K}}}.
\]
We note that Ern and Burman also utilized GLL weights to prove a similar trace inequality in $d$ dimensions
\[
\nor{u}^2_{L^2\LRp{\partial \widehat{K}}} \leq d\frac{(N+1)(N+2)}{2}\LRp{2+\frac{1}{N}}^d \nor{u}^2_{L^2\LRp{\widehat{K}}}.
\]  
The additional factor of $(2+1/N)^d$ results from the equivalence constant between the $L^2$ norm and discrete $L^2$ norm resulting from underintegration using GLL quadrature.  

\section{Computed Markov and trace inequality constants}
\label{app:MTconsts}
For wedge, pyramid, and tetrahedral elements, we have not been able to determine explicit expressions for the constants $C_T(N)$ in trace inequalities over reference elements $\widehat{K}$
\[
\nor{u}^2_{L^2(\partial \widehat{K})} \leq C_T(N) \nor{u}^2_{L^2(\widehat{K})}.
\]
However, $C_T(N)$ may be computed as the maximum eigenvalue of a generalized eigenvalue problem
\[
M_s v = \lambda M v,
\]
where $M_s$ is the surface mass matrix and $M$ is the volume mass matrix.  We may similarly compute the trace inequality constants $C_{\rm SEM}(N)$  for the case when SEM quadrature is used for hexahedra and quadrilateral faces.  
\begin{table}
\centering                                                                                   
\begin{tabular}{|c|c|c|c|c|c|c|c|c|c|}                         
\hline                                                   
$N$ & 1 & 2 & 3 & 4 & 5 & 6 & 7 & 8 & 9 \\ 
\hline                                                   
Hexahedron &  9  & 18 &   30 &   45  &  63 &   84 &  108 &  135  & 165\\
\hline
Wedge & 9.93 & 18.56 & 29.03 & 42.99 & 58.80 & 78.01 & 99.27 & 123.76 & 150.48 \\
\hline                                                   
Pyramid & 11.68 & 20.89 & 32.84 & 47.59 & 65.17 & 85.60 & 108.90 & 135.07 & 164.11 \\
\hline                                                   
Tetrahedron & 12.22 & 20.46 & 29.18 & 41.65 & 54.45 & 71.10 & 88.32 & 109.04 & 130.67 \\ 
\hline                                                   
\end{tabular}
\caption{Computed trace constants $C_T(N)$ using the full mass matrix for various reference elements up to $N=9$. }
\label{table:computedTrace}
\end{table}
\begin{table}                                                 
\centering                                                                 
\begin{tabular}{|c|c|c|c|c|c|c|c|c|c|}                         
\hline                                                   
$N$ & 1 & 2 & 3 & 4 & 5 & 6 & 7 & 8 & 9 \\ 
\hline                                                                     
Hexahedron & 3 & 9 & 18 & 30 & 45 & 63 & 84 & 108 & 135 \\   
\hline             
Wedge & 46.46 & 51.70 & 63.94 & 83.32 & 104.76 & 132.20 & 161.22 & 195.95 & 232.52 \\
\hline                                                                       
Pyramid & 31.58 & 37.12 & 47.59 & 60.95 & 77.14 & 96.31 & 118.52 & 143.80 & 172.13 \\                                                         
\hline                                                                     
\end{tabular}                                                              
\caption{Computed trace constants $C_{\rm SEM}(N)$ using SEM quadrature for various reference elements up to $N=9$.  SEM quadrature is used to compute the mass matrix and surface mass matrices when quadrilateral faces are present.  The constants for tetrahedra are identical in both cases.}
\label{table:computedTraceSEM}
\end{table}              
For the Markov inequality on the reference element
\[
\nor{\Grad u}^2_{L^2(\widehat{K})} \leq C_M(N) \nor{u}^2_{L^2(\widehat{K})}
\]
we may similarly compute the constant $C_M(N)$ as the maximum eigenvalue of
\[
K v = \lambda M v,
\]
where $K_{ij} = \int_{\widehat K} \Grad\phi_j \cdot\Grad \phi_i$ is the stiffness matrix over the reference element.  We summarize computed values of $C_T(N)$, $C_{\rm SEM}(N)$ and $C_M(N)$ up to $N = 9$ in Tables~\ref{table:computedTrace}, \ref{table:computedTraceSEM} and \ref{table:markovConsts}.  For a SEM hex, we may determine the constant in the Markov inequality using the full mass matrix value of $C_M(N)$ and equivalence constants between discrete SEM and $L^2$ inner products.  

\begin{table}
\centering                                                                         
\begin{tabular}{|c|c|c|c|c|c|c|c|c|c|}                                               
\hline                                                                             
$N$ & 1 & 2 & 3 & 4 & 5 & 6 & 7 & 8 & 9 \\ 
\hline                                                                             
Hexahedron & 3.00 & 18.00 & 55.73 & 137.51 & 293.97 & 562.17 & 985.92 & 1616.24 & 2511.47 \\    
\hline                                                                             
Wedge &12.00 & 54.27 & 142.63 & 308.34 & 585.89 & 1021.64 & 1663.85 & 2574.06 & 3814.56 \\
\hline                                                                             
Pyramid &12.92 & 60.05 & 175.51 & 405.43 & 809.95 & 1460.93 & 2442.26 & 3849.94 & 5792.11 \\
\hline                                                                             
Tetrahedron &20.00 & 78.62 & 195.58 & 403.91 & 744.85 & 1265.54 & 2021.09 & 3073.26 & 4491.62 \\
\hline                                                                             
\end{tabular}                                                                      
\caption{Computed Markov constants $C_M(N)$ for various reference elements up to $N=9$. }
\label{table:markovConsts}
\end{table}

\end{document}